%% file: GGPRultramsurfsing-ArXiv-final.tex
\documentclass[11 pt, twoside, a4paper]{amsart}

\usepackage[headsepline,plainheadsepline]{scrpage2}
\usepackage{amsfonts, amsmath, amssymb, amsthm, amsopn, latexsym, graphicx, mathrsfs,  verbatim, enumerate, url, stmaryrd, enumitem, tikz} 
\usepackage{tikz-cd}
\usepackage[normalem]{ulem}

\usepackage{hyperref}
\usepackage[all]{hypcap} 
\usepackage{pst-all} 
\usepackage{Style} 
\usepackage{subcaption}

\usepackage{lineno} 

\usepackage{color}

\definecolor{DarkGreen}{RGB}{0, 164, 0}

\DeclareMathAlphabet{\mathbbm}{U}{bbm}{m}{n}

\newcommand{\executeiffilenewer}[3]{%
\ifnum\pdfs

trcmp{\pdffilemoddate{#1}}%
{\pdffilemoddate{#2}}>0%
{\immediate\write18{#3}}\fi%
}

\hypersetup{
colorlinks=true,
linkcolor=blue,
citecolor=red
}

\newtheorem{introthm}{Theorem}
\newtheorem*{thmno}{Theorem}
\newtheorem*{propno}{Proposition}

\newtheorem{thm}{Theorem}[section]

\newtheorem{cor}[thm]{Corollary}
\newtheorem{prop}[thm]{Proposition}

\theoremstyle{definition}
\newtheorem{rmk}[thm]{Remark}
\newtheorem{defi}[thm]{Definition}
\newtheorem{ex}[thm]{Example}

\theoremstyle{definition}

\newcommand{\monAUX}[1]{\mbox{\texttt{#1}}}
\newcommand{\mon}[1]{$\boxed{#1}$\color{red}\expandafter\monAUX\expandafter{\detokenize{#1}}\color{black}\ }

\newcommand{\exc}[1]{{E{(#1)}}}			
\newcommand{\primes}[1]{{\mathcal{P}(#1)}}		
\newcommand{\ExcZ}[1]{{\mathcal{E}(#1)_{\nZ}}}		
\newcommand{\ExcQ}[1]{{\mathcal{E}(#1)_{\nQ}}}		
\newcommand{\ExcR}[1]{{\mathcal{E}(#1)_{\nR}}}		
\newcommand{\genprimes}[1]{{\mathcal{P}(#1)}}		
\newcommand{\bra}[1]{{\langle#1\rangle}}				
\newcommand{\logfunc}{\rho}				
\newcommand{\dgr}[1]{{\Gamma_{#1}}}			
\newcommand{\branches}[1]{{\mathcal{B}(#1)}}		

\newcommand{\udb}[1]{{u_{#1}}}				

\newcommand{\Vgr}[1]{{\mathcal{V}(#1)}}				
\newcommand{\Egr}[1]{{\mathcal{A}(#1)}}				
\newcommand{\bvt}[1]{{\mc{BV}(#1)}}			
\newcommand{\cvxh}[1]{{\on{Conv}(#1)}}			

\newcommand{\FVal}[1]{{\hat{\mc{V}}_{#1}^*}}		
\newcommand{\Val}[1]{{\mc{V}_{#1}}}			
\newcommand{\udv}[1]{{u_{#1}}}				
\newcommand{\Rbvt}[1]{{\mc{BV}(#1)}}			
\newcommand{\Rbv}{{i_{\on{bv}}}}				
\newcommand{\apex}[1]{{v_{#1}}}				

\newcommand{\strt}[2]{{#1}_{#2}}			
\newcommand{\exct}[2]{{#1}_{#2}^{ex}}			

\newcommand{\etoile}[1]{{\on{Star}(#1)}}
\newcommand{\skel}[1]{\mc{S}_{{#1}}}
\newcommand{\skeldiv}[1]{\mc{S}_{{#1}}^*}

\newcommand{\cF}{\mathcal{F}}

\newcommand{\cO}{\mathcal{O}}

\newcommand{\divi}{\mathrm{div}}
\newcommand{\inte}{\mathrm{int}}

\renewcommand{\pr}{\on{pr}}

\newcommand{\refprop}[1]{\hyperref[prop:#1]{Proposition~\ref*{prop:#1}}}
\newcommand{\refthm}[1]{\hyperref[thm:#1]{Theorem~\ref*{thm:#1}}}
\newcommand{\refcor}[1]{\hyperref[cor:#1]{Corollary~\ref*{cor:#1}}}
\newcommand{\refdef}[1]{\hyperref[def:#1]{Definition~\ref*{def:#1}}}
\newcommand{\refrmk}[1]{\hyperref[rmk:#1]{Remark~\ref*{rmk:#1}}}
\newcommand{\reflem}[1]{\hyperref[lem:#1]{Lemma~\ref*{lem:#1}}}
\newcommand{\refsec}[1]{\hyperref[sec:#1]{Part~\ref*{sec:#1}}}
\newcommand{\reffig}[1]{\hyperref[fig:#1]{Figure~\ref*{fig:#1}}}
\newcommand{\refex}[1]{\hyperref[ex:#1]{Example~\ref*{ex:#1}}}
\newcommand{\refssec}[1]{\hyperref[ssec:#1]{Section~\ref*{ssec:#1}}}
\newcommand{\refpropss}[3]{\hyperref[prop:#1]{Propositions~\ref*{prop:#1}}, \ref{prop:#2}, \ref{prop:#3}}
\newcommand{\refprops}[2]{\hyperref[prop:#1]{Propositions~\ref*{prop:#1}} and \ref{prop:#2}}

\newcommand{\refcors}[2]{\hyperref[cor:#1]{Corollaries~\ref*{cor:#1}} and \ref{cor:#2}}
\newcommand{\reffigs}[2]{\hyperref[fig:#1]{Figures~\ref*{fig:#1}} and \ref{fig:#2}}
\newcommand{\refexs}[2]{\hyperref[ex:#1]{Examples~\ref*{ex:#1}} and \ref{ex:#2}}
\newcommand{\refdefs}[2]{\hyperref[def:#1]{Definitions~\ref*{def:#1}} and \ref{def:#2}}

\renewcommand{\emph}{\textbf} 

\title{Ultrametric properties for valuation spaces of normal surface singularities}

\author{Evelia R. Garc\'ia Barroso}
\address{Dpto. Matem\'aticas, Estad\'{\i}stica e Investigaci\'on Operativa, Universidad de La Laguna}
\email{\href{mailto:ergarcia@ull.es}{ergarcia@ull.es}}

\author{Pedro D. Gonz\'alez P\'erez}
\address{Dpto. \'Algebra, Geometr\'{\i}a y Topolog\'{\i}a, Universidad Complutense de Madrid}
\email{\href{mailto:pgonzalez@mat.ucm.es}{pgonzalez@mat.ucm.es}}

\author{Patrick Popescu-Pampu}
\address{Laboratoire Paul Painlev\'e, Universit\'e de Lille}
\email{\href{mailto:patrick.popescu-pampu@univ-lille.fr}{patrick.popescu-pampu@univ-lille.fr}}

\author{Matteo Ruggiero}
\address{Institut de Math\'{e}matiques de Jussieu-Paris Rive Gauche,  Universit\'e Paris Diderot}
\email{\href{mailto:matteo.ruggiero@imj-prg.fr}{matteo.ruggiero@imj-prg.fr}}

\date{24 July 2019}

\clearscrheadfoot
\lohead{\normalfont Ultrametric properties for valuation spaces of normal surface singularities \hfill \pagemark}

\lehead{\normalfont\pagemark\hfill E.~R. Garc\'ia Barroso, P. ~D. Gonz\'alez P\'erez, 
    P. Popescu-Pampu and M. Ruggiero}

\begin{document}


\begin{center}
      {\bf  \small This paper will appear in 
    Trans. of the American Math. Soc. \\ DOI: https://doi.org/10.1090/tran/7854 }
    \end{center}
\bigskip

\thispagestyle{empty}

\begin{abstract}
{ 
      Let $L$ be a fixed branch  -- that is, an irreducible germ of curve -- 
      on a normal surface singularity $X$. If $A,B$ are two other 
      branches, define $\udb{L}(A,B) := \dfrac{(L \cdot A) \: (L \cdot B)}{A \cdot B}$,
      where $A \cdot B$ denotes the intersection number of $A$ and $B$. 
      Call $X$  \textit{arborescent} if all the dual graphs of its  good resolutions are trees. 
      In a previous paper, the first three authors extended a 1985 theorem of 
      P\l oski by proving that whenever $X$ is arborescent, the function $\udb{L}$ 
      is an \textit{ultrametric} on the set of branches on $X$ different from $L$. In the present 
      paper we prove that, conversely, if  $\udb{L}$ is an ultrametric, then 
      $X$ is arborescent. We also show that for \textit{any} normal surface singularity, 
      one may find arbitrarily large sets of branches on $X$, characterized uniquely 
      in terms of the topology of the resolutions of their sum, in restriction to which 
      $u_L$ is still an ultrametric. Moreover, we describe the associated tree in terms 
      of the dual graphs of such resolutions. Then we extend our setting by allowing 
      $L$ to be an arbitrary \textit{semivaluation} on $X$ and by defining  
      $\udb{L}$ on a suitable space of semivaluations. We prove that any such 
      function is again an ultrametric if and only if $X$ is arborescent, and 
      without any restriction on $X$ we exhibit special  
      subspaces of the space of semivaluations in restriction to which 
      $\udb{L}$ is still an ultrametric. 
      }
\end{abstract}

\maketitle

\vspace{-3mm}

\tableofcontents

\section*{Introduction}
\bigskip

Let $X$ be a \textit{normal surface singularity}, which will mean for us throughout the paper 
a germ of normal complex analytic surface. 
A \textit{branch} on it is an irreducible germ of formal curve on $X$. 
In his 1985 paper \cite{ploski:1985}, P\l oski proved a theorem which may be reformulated 
in the following way:

\begin{thmno}
If $X$ is smooth, then the map which associates to any pair $(A,B)$ of branches on it the quotient $\dfrac{m(A) \:  m(B)}{A \cdot B}$ of the product of their multiplicities by their intersection number, is an ultrametric on the set of branches on $X$.
\end{thmno}

The first three authors proved  in  \cite[Theorem 4.18]{gbgppp:2016}  that this result generalizes to the case of 
\textit{arborescent singularities}, which are the normal surface singularities 
whose good resolutions (with simple normal crossing exceptional divisors) 
have trees as dual graphs: 

\begin{thmno}
Let $X$ be an arborescent singularity and $L$ a fixed branch on it. 
Then the map $\udb{L}$ which associates to any pair $(A,B)$ of branches on $X$ 
the quotient $\dfrac{(L \cdot A) \: (L \cdot B)}{A \cdot B}$, 
is an ultrametric on the set of branches on $X$ distinct from $L$.
\end{thmno}

Note that on arbitrary normal surface singularities the intersection 
numbers are defined in the sense of Mumford 
\cite{mumford:1961} and may take non-integral (but still rational) values. 

One may recover P\l oski's theorem as a particular case of the previous one.  
Indeed, smooth germs $X$ are arborescent and the ultrametric 
property of the quotients involved in P\l oski's theorem may be tested on any finite set of branches. 
Then it is enough to choose a smooth branch $L$ which is transversal 
to all the branches in a fixed finite set.

The main aspect of the approach of \cite{gbgppp:2016} was to express 
the intersection numbers of branches on a normal surface singularity $X$ in terms of 
intersection numbers of exceptional divisors on a resolution $X_{\pi}$ of $X$. What made 
ultimately everything work was the following inequality between the intersection 
numbers of the divisors of the basis $(\check{E}_u)_u$ 
of the vector space of real exceptional divisors of $X_{\pi}$ which is dual to 
the basis formed by the prime exceptional divisors $(E_u)_u$. 
This inequality (see \refprop{crucial}) was generalized by Gignac and the fourth author 
in \cite[Proposition 1.10]{gignac-ruggiero:locdynnoninvnormsurfsing}:

\begin{propno}
    {\it Let $X$ be a normal surface singularity and $X_{\pi}$ a good 
    resolution of it.  Let $E_u$, $E_v$ and $E_w$ 
     be not necessarily distinct exceptional prime divisors of $X_\pi$. 
     Then one has the inequality:
    \[
         (- \check{E}_u \cdot \check{E}_v)(- \check{E}_v \cdot \check{E}_w) \leq 
         (- \check{E}_v \cdot \check{E}_v)(- \check{E}_u \cdot \check{E}_w),
     \]
    with equality if and only if $v$ separates $u$ and $w$ 
   in the dual graph of $X_{\pi}$.
    }
\end{propno}

This inequality is also crucial in this paper, and has an intriguing  
reformulation in terms of spherical geometry (see \refprop{MW}). 

\medskip

Our paper  has two parts.
\refsec{branches} treats the case of the functions $\udb{L}$ restricted to finite 
sets of branches. 
 In  \refsec{valuations} we 
show how the results of the first part can be extended to 
the space of normalized semivaluations of $X$. Let us summarize our main results.
\medskip

     We prove a converse of one of the main theorems of \cite{gbgppp:2016}, which stated that 
         $u_L$ is an ultrametric whenever $X$ is arborescent (see \refthm{arborcase}):
      
         \begin{introthm}
              {\it The normal surface singularity $X$ is arborescent if and only if 
                either one or all the functions $\udb{L}$, for varying branches $L$ on $X$, are ultrametrics. }
        \end{introthm}

         More generally, if $X$ is a normal 
         surface singularity, and  $\cF$ is a finite set of branches on $X$ containing a fixed branch $L$, 
         we show that $u_L$ 
         is an ultrametric on  $\cF \setminus \{L\}$ 
         whenever the dual graph of the total transform 
         of the sum of the branches in $\cF$ in an arbitrary embedded resolution of it 
         satisfies a topological condition. Its formulation uses the notion of 
         \textit{brick-vertex tree} $\bvt{G}$ of a finite connected graph $G$. It is a finite tree 
        containing the vertices of $G$ and other vertices called  \textit{brick vertices}, 
        which encodes the way the vertices of $G$ get separated by an arbitrary 
        one of them (see \refssec{blockvertextree}). 
       We prove that (see  \refthm{ultramthm}): 

\begin{introthm}
{\it If the convex hull $\cvxh{\cF}$ of the branches of $\cF$ in the brick-vertex tree of the dual 
graph of the chosen embedded resolution does not contain brick vertices of valency 
at least $4$ in $\cvxh{\cF}$, then $\udb{L}$ is an ultrametric in 
restriction to $\cF \setminus \{ L \}$. Moreover, in this case the rooted tree of 
$\udb{L}$ restricted to $\cF \setminus \{ L \}$ is isomorphic to $\cvxh{\cF}$, 
rooted at the vertex corresponding to $L$.} 
\end{introthm}

         Note that this result does 
         not involve intersection numbers or genera of prime exceptional divisors. 
         It is always satisfied when $X$ is arborescent, which allows to recover  
          \cite[Theorem 4.18]{gbgppp:2016}. 
         \medskip

Let us pass to the \textit{semivaluations} of $X$ considered in  \refsec{valuations}. 
Compared to valuations, they 
may achieve the value $+ \infty$ on other elements of the local ring of $X$ 
than simply $0$. Allowing to work not only with valuations, but also with semivaluations,  
has the advantage that any branch on $X$ has an associated semivaluation, which 
associates to an element of the local ring of $X$ the intersection number of its divisor with $L$. 
Also, any prime exceptional divisor of a normal crossings resolution of $X$ has an associated 
semivaluation, which is in fact a valuation.  
Therefore, the vertices of the dual graphs of the total transforms of the sums of 
finite sets of branches on $X$ embed naturally in the space of semivaluations of $X$. 
In fact, this embedding can be extended to the whole dual graph, seen as a topological space. 
It is more convenient to our purpose, as it was in the model case of smooth $X$ treated 
in Favre and Jonsson's book \cite{favre-jonsson:valtree}, to consider a space of 
\textit{normalized} semivaluations.  
The normalization condition is simply to consider only semivaluations 
which take the value $1$ on the maximal ideal 
of the local ring of $X$. It ensures that one gets a topological space of dimension $1$. 
\medskip

      We generalize \refthm{arborcase} to arbitrary semivaluations on $X$ 
         (see \refthm{arborcase_val}). Namely, we replace the branch $L$, seen 
         as a particular semivaluation, by an 
         arbitrary normalized semivaluation $\lambda$ on $X$, and we consider an analog 
         $\udb{\lambda}$ of the function $\udb{L}$, defined this time on the 
         space of normalized semivaluations which are distinct from $\lambda$.
         We prove that:
         
         \begin{introthm}
                {\it The normal surface singularity $X$ is arborescent if and only if either one or 
              all the functions $\udb{\lambda}$, for varying semivaluations $\lambda$ of $X$, are ultrametrics.} 
         \end{introthm}
     
    We generalize \refthm{ultramthm} to arbitrary semivaluations on $X$ 
        (see \refthm{udbv_val}). Namely, we prove that for any normal surface 
        singularity $X$, any normalized semivaluation $\lambda$ on it, and any set $\cF$ (not necessarily 
        finite) of normalized semivaluations containing $\lambda$, the function 
        $\udb{\lambda}$ is an ultrametric in restriction to $\cF$ whenever $\cF$  
        satisfies a suitable topological condition in the space of normalized semivaluations of $X$. 
The topological conditions involved in the statements of  \refthm{ultramthm} and \refthm{udbv_val} 
are analogous, involving finite graphs in the first case and special types of {\em infinite graphs} in the 
second case. Let us compare both cases. 
\medskip

We show that the space of normalized semivaluations has a structure of 
\textit{connected graph of $\R$-trees of finite type}  (see \refprop{valXisagraph}). 
We extend the notion of \textit{brick-vertex tree} to such 
spaces (see \refssec{brickrtree}). In the case of the space of normalized semivaluations, there is only a 
finite number of brick vertices, which correspond bijectively to those of the dual graph 
of any normal crossings resolution of $X$. 
Using the brick-vertex tree of the space of normalized semivaluations of $X$, we 
prove analogs for the functions $\udb{\lambda}$ 
of the results formulated in terms of brick-vertex trees of finite graphs for the 
functions $\udb{L}$ (see \refssec{semivalgraphRtree}). 
In fact, the bricks are precisely the non-punctual \textit{cyclic elements} 
of the space of normalized semivaluations. 
In \refrmk{cyclelemtheory} we give historical details about the \textit{topological theory of 
cyclic elements}.

\medskip

In the whole paper, we deal for simplicity with \textit{complex} normal surface singularities. 
But our approach works also for singularities which are spectra of 
normal $2$-dimensional local rings defined over fields of arbitrary 
characteristic. Indeed, our treatment is ultimately based on the fact that the intersection matrix of a 
resolution of the singularity is negative definite (see \refthm{intersform} below), 
a theorem which is true in this greater generality, as shown by Lipman \cite[Lemma 14.1]{lipman:1969}.
For the description of semivaluation spaces associated to regular surface singularities 
over fields of any characteristic, we refer to Jonsson's paper \cite[Section 7]{jonsson:berkovich} 
-- see in particular its Section 7.11 for a discussion of the specificities of non-algebraically 
closed base fields.
Jonsson's approach can be directly generalized to any normal surface singularity defined 
over arbitrary fields, by applying 
his constructions to the sets of semivaluations centered at smooth points in any good resolution 
of the given singularity.

\medskip
\noindent\textbf{Acknowledgements.}
     {The authors would like to thank Charles Favre who, 
     viewing their papers \cite{gbgppp:2016} and 
     \cite{gignac-ruggiero:locdynnoninvnormsurfsing}, still in progress at that time, 
     suggested them to work together on a combination of both approaches. 
     The third author is grateful to Norbert A'Campo for a conversation which 
     allowed him to realize the spherical reformulation of the crucial inequality. 
     Part of the final revisions to the paper have been carried out while the fourth author 
     was visiting the Imperial College London and the Beijing International Center 
     for Mathematical Research. The fourth author would like to thank both institutions 
     for their welcoming.
     The authors are also grateful to the referee, whose remarks allowed to clarify 
     the paper. 
     This research was partially supported by the  French grants ANR-12-JS01-0002-01 SUSI, 
     ANR-17-CE40-0023-02 LISA, 
     ANR-17-CE40-0002-01 Fatou and  Labex CEMPI (ANR-11-LABX-0007-01), 
     and also by  the Spanish Projects  
     MTM2016-80659-P and MTM2016-76868-C2-1-P.

\section{Ultrametric distances on finite sets of branches} \label{sec:branches}

{Let $X$ be a normal surface singularity and $L$ a finite branch on it. 
Let $\udb{L}$ be the function introduced by the first three authors in \cite{gbgppp:2016}, 
which associates to every  pair $(A, B)$ of branches on $X$ which are different from $L$ 
the number  $(L \cdot A) \: (L \cdot B) (A \cdot B)^{-1}$. In this first part 
of the paper we study its behaviour on finite sets of branches on $X$. Our main 
results are that $\udb{L}$ is an ultrametric on any such set if and only if $X$ is 
\textit{arborescent} (see \refthm{arborcase}) 
and that even when $X$ is not arborescent, it is still an ultrametric in restriction to 
arbitrarily large sets of branches, which may be characterized topologically in terms 
of their total transform on any good resolution of their sum (see \refthm{ultramthm}). 
These theorems need a certain amount of preparation, which explains the need for 
a subdivision of this part into six sections. The content of each 
section is briefly described at its beginning. }

\bigskip
\subsection{Mumford's intersection number of divisors} \label{ssec:mumfordint}
$\:$  

\medskip
   In this section we recall \textit{Mumford's definition of intersection number} of 
   Weil divisors on a normal surface singularity $X$ (see \refdef{mumfint}). 
   This definition passes through an intermediate definition of \textit{total transform} 
   of such a divisor by a resolution of the singularity (see \refdef{totransf}), 
   which in turn uses basic properties of the intersection form 
   on such a resolution. That is why we begin the section by recalling the needed 
   theorems about the intersection theory on resolutions of $X$ 
   (see \refthm{intersform} and \refpropss{totinv}{negdual}{invdual}). We also introduce 
   many of the notions used elsewhere in the paper. The most important one for 
   the sequel is that of \textit{bracket} $\bra{u,v}$ of two prime divisorial valuations $u,v$ on 
   $X$ (see \refdef{bra}), which may be interpreted as Mumford's intersection 
   number of a pair of branches adapted to the two valuations (see \refprop{intexcep}). 
\bigskip

In the whole paper, we fix a {\bf normal surface singularity} $(X, x_0)$, that is, a germ of 
complex analytic normal surface. In particular, the germ is irreducible and has a 
representative which is smooth outside $x_0$.  In order to shorten the notations, 
most of the time we will write simply $X$ instead of $(X, x_0)$. 
We will denote by $\cO_X$ the local ring of $X$.

A {\bf branch} on $X$ is a germ at $x_0$ of irreducible formal curve lying 
on $X$. The set of branches on $X$ will be denoted by $\boxed{\branches {X}}$. 

By a {\bf divisor} on $X$ we will mean an integral 
Weil divisor, that is, an element of the free abelian group generated by the branches on $X$. 
As usual, a {\bf principal divisor} is the divisor $(f)$ of a formal meromorphic function $f$ on  $X$,  
that is, of an element of the fraction field of the completion of $\cO_X$ relative to its maximal ideal. 

A {\bf resolution} of $X$ is a proper bimeromorphic morphism $\pi \colon X_{\pi} \to X$ of 
complex analytic spaces, such that $X_{\pi}$ is smooth and $\pi$ is an isomorphism 
over $X \setminus \{x_0\}$.
If $\pi \colon X_{\pi} \to X$ is a resolution of $X$, we will say 
that $X_{\pi}$ is a {\bf model} of $X$. The 
{\bf reduced exceptional divisor of the resolution} $\pi$ will be denoted by $\boxed{\exc {\pi}}$ and 
its set of irreducible components by $\boxed{\primes {\pi}}$. By {\bf an exceptional divisor} on 
$X_{\pi}$ we mean, depending on the context, either an element of the abelian group 
$\boxed{\ExcZ {\pi}}$ freely generated by the elements of $\primes {\pi}$, of the associated 
$\Q$-vector space $\boxed{\ExcQ {\pi}}$, or of the associated $\R$-vector space 
$\boxed{\ExcR {\pi}}$.

The irreducible components of the 
reduced exceptional divisors of the various resolutions of $X$ will be called {\bf prime exceptional 
divisors}. By associating to a prime exceptional divisor its corresponding 
integer-valued valuation on the local ring $\cO_X$ (that is, the vanishing order along 
the divisor), we may identify $\primes {\pi}$ with a set of divisorial valuations on the local 
ring $\cO_X$ (see \refssec{valuation_spaces}). 
Therefore, denoting by $\boxed{E_u}$ the prime divisor on $X_{\pi}$ corresponding to 
$u \in \primes {\pi}$, we may think that $u$ also denotes the corresponding divisorial 
valuation on $\cO_X$. Whenever we will reason with several models at the same time, 
we will denote by $E^{\pi}_u$ instead of $E_u$ the prime divisor on the model 
$X_{\pi}$ corresponding to the divisorial valuation $u$.  
But when we will work with a fixed model, for simplicity we will drop from the 
notations this dependency on the model.  

We will say that the divisorial valuations $u$ on $\cO_X$ associated to prime divisors $E_u$ are 
{\bf prime divisorial valuations}. We will denote by $\boxed{\primes {X}}$ the set of prime 
divisorial valuations. It is the union of the subsets $\primes {\pi}$ of the set of 
divisorial valuations of $X$, when $\pi$ varies among the resolutions of $X$. 
If $u \in \primes {X}$ and $X_{\pi}$ is a model such that $u \in \primes {\pi}$, 
we say that $u$ {\bf appears on the model} $X_{\pi}$.

Given a resolution $\pi$ of $X$,  the intersection 
number of exceptional divisors of $X_{\pi}$ defines a symmetric bilinear form on the 
vector space $\ExcR {\pi}$, called its {\bf intersection form}. For simplicity, 
we will denote by $\boxed{D_1 \cdot D_2}$ the intersection 
number of the exceptional divisors $D_1$ and $D_2$, without mentioning the morphism 
$\pi$ explicitly. This convention may be motivated by the classical 
fact that \emph{the intersection number 
is birationally invariant} in the following sense:

\begin{prop}   \label{prop:totinv}
  If the model $X_{\pi_2}$ dominates the model $X_{\pi_1}$, then the 
  intersection number of two divisors of $X_{\pi_1}$ is equal 
  to the intersection number of their total transforms on $X_{\pi_2}$.
\end{prop}

\begin{proof} 
  Let $\psi: X_{\pi_2} \to X_{\pi_1}$ be the domination morphism between 
   the two models. 
    Recall the \textit{projection formula}, comparing intersection numbers on the two models 
    (see Hartshorne \cite[Appendix A.1]{hartshorne:alggeo}): 
       \begin{equation} \label{eqn:projform} 
            D_2 \cdot \psi^* D_1 = \psi_* D_2 \cdot D_1
       \end{equation}
          for every $D_1 \in \ExcR {\pi_1}$ and $D_2 \in \ExcR {\pi_2}$ (the left hand side 
          being computed on $X_{\pi_2}$ and the right hand side on $X_{\pi_1}$). Here 
          $\psi^* D_1$ denotes the total transform of $D_1$ by the morphism $\psi$ 
          and $\psi_* D_2$ denotes the direct image of $D_2$ by the same morphism. 
       Consider now two divisors $A, B$ on $X_{\pi_1}$. Then:
           $$ \psi^* A \cdot \psi^* B = \left( \psi_* \psi^* A\right) \cdot B = A \cdot B,$$
        the first equality being a consequence of the projection formula \eqref{eqn:projform} 
        applied to $D_1 = B$, $D_2 = \psi^* A$ and the second equality being a 
        consequence of the fact that $\psi_* \psi^* A = A$. 
\end{proof}

Note that the previous assertion does not remain true if one replaces total transforms 
of divisors by strict transforms. In particular, for fixed $u,v \in \primes {X}$, 
the intersection number $E_u^{\pi} \cdot E_v^{\pi}$ depends on the model $X_{\pi}$ on which 
$E_u^{\pi}$ and $E_v^{\pi}$ appear. Compare this fact with \refprop{invdual} below. 

One has the following fundamental theorem concerning the intersection form 
on a fixed model (see Du Val \cite{duval:1944} 
and Mumford \cite{mumford:1961} in what concerns point (\ref{negdef}) and Zariski 
\cite[Lemma 7.1]{zariski:1962} in what concerns point (\ref{antinef})):

\begin{thm}   \label{thm:intersform}
   Let $X_{\pi}$ be a model of the normal surface singularity $X$.  

   \begin{enumerate}
        \item \label{negdef} 
            The intersection form on the vector space $\ExcR {\pi}$ is negative definite.
        \item \label{antinef} 
             If $D \in \ExcR {\pi}\setminus \{0\}$ is such that $D \cdot H \geq 0$ 
           for all effective divisors  $ H \in \ExcR {\pi}$, then $-D$ is effective and it is 
           of full support in the basis  $(E_u)_{u \in \primes {\pi}}$, that is, all the coefficients 
           of its decomposition in this basis are positive. 
    \end{enumerate} 
\end{thm}

The second statement is a consequence of the following theorem of linear algebra, which will be used in the proof of \refprop{crucial}
(one may verify easily that Zariski's proof in \cite[Lemma 7.1]{zariski:1962} 
transcribes immediately in a proof of it): 

\begin{prop}  \label{prop:dualcone}
   Let $\mathcal{E}$ be a Euclidean finite dimensional vector space. 
   Consider a basis $\mathcal{B}$ of $\mathcal{E}$ such that the plane angles 
   generated by any 
   pair of its vectors are right or obtuse. Assume moreover that $\mathcal{B}$ cannot 
   be partitioned into two non-empty subsets orthogonal to each other. 
   Denote by $\sigma$ the cone generated by $\mathcal{B}$ and 
   let $\check{\sigma}$ be the cone generated by the dual basis.  Then 
   $\check{\sigma} \setminus 0$ is included in the interior of $\sigma$. 
\end{prop}

In order to get \refthm{intersform} (\ref{antinef}) from \refprop{dualcone}, one takes 
as Euclidean vector space $\mathcal{E}$ the space of exceptional divisors $\ExcR {\pi}$, 
endowed with the opposite of the intersection form and with the basis 
$(E_u)_{u \in \primes{\pi}}$. The hypothesis on the angles is satisfied because 
$E_u \cdot E_v \geq 0$ for all $u \neq v$. The hypothesis on the impossibility 
to partition the basis in two orthogonal non-empty subsets is equivalent 
to the connectedness of the exceptional divisor $\exc {\pi}$. In turn, this is a consequence of the 
hypothesis that $X$ is normal, as a special case of the so-called \textit{Zariski main  
theorem} (see \cite[Corollary 11.4]{hartshorne:alggeo}). 

If $D\in \ExcR {\pi}$ is a divisor such that $-D$ is effective, we will say that $D$ is {\bf anti-effective}. 
If $D \cdot H \geq 0$ for all effective divisors  $ H \in \ExcR {\pi}$, we will say that 
$D$ is {\bf nef (numerically eventually free)}. Usually one says in this case that $D$ is 
\emph{nef relative to the morphism $\pi$}, but in order to be concise we will drop 
the reference to $\pi$.

If $E_u$ is an exceptional prime divisor on the model $X_{\pi}$, 
we denote by $\boxed{\check{E}_u} \in \ExcQ {\pi}$ the dual divisor with respect 
to the intersection form. It is defined by:
\begin{equation} \label{eqn:dualdiv}
     \check{E}_u \cdot E_v = \delta_{u,v} \:  \mbox{ for all } v \in \primes{\pi},
\end{equation}
 where $\delta_{u,v}$ denotes Kronecker's delta. The existence and uniqueness  
 of this dual basis is a consequence of \refthm{intersform} (\ref{negdef}). 
 The fact that it lives in $\ExcQ {\pi}$ follows from the fact that all the intersection 
 numbers $E_u \cdot E_v$ are integers. One has the following immediate consequence 
 of formulae \eqref{eqn:dualdiv}:
     \begin{equation}  \label{eqn:devdiv}
           D = \sum_{v \in \primes {\pi} } \left( D \cdot \check{E}_v \right) E_v
     \end{equation}
 for all $D \in \ExcR {\pi}$. 
       
 As an immediate consequence of \refthm{intersform} (\ref{antinef}) 
 and of formula \eqref{eqn:devdiv} applied to the nef divisors $\check{E}_u$, we get:

 \begin{prop}  \label{prop:negdual}
       The divisors $\check{E}_u$ are anti-effective with full support  
       in the basis $(E_u)_{u \in \primes{\pi}}$, that is,  $\check{E}_u \cdot \check{E}_v < 0$ 
       for all $u,v \in \primes{\pi}$. 
   \end{prop}
   
   In contrast with the fact that the intersection numbers $E_u \cdot E_v$ depend on the 
   model on which they are computed, one has the following classical invariance property:
   
   \begin{prop}  \label{prop:invdual}
        Let $u,v \in \primes {X}$. Then the intersection number $\check{E}_u \cdot \check{E}_v$ 
        does not depend on the model on which it is computed. 
   \end{prop}

   \begin{proof} 
       Let $\psi: X_{\pi_2} \to X_{\pi_1}$ be the domination morphism 
       between two models of $X$.  In this proof we 
       will not drop the reference to the model on which one works, using the notations 
       $E_u^{\pi_i}, \check{E}_u^{\pi_i}$ for $i \in \{1, 2 \}$. 
        In view of \refprop{totinv}, it is enough to show that if 
        $u \in \primes {\pi_1}$, then the divisor $\check{E}_u^{\pi_2}$  is the total transform of the
	divisor $\check{E}_u^{\pi_1}$.
        
        By the projection formula \eqref{eqn:projform}, one has: 
           $E_v^{\pi_2} \cdot \psi^* \check{E}_u^{\pi_1} = 0$
       for all $v \in \primes {\pi_2}\setminus \{u\}$ and 
            $ E_u^{\pi_2} \cdot \psi^* \check{E}_u^{\pi_1} 
             = \psi_* {E}_u^{\pi_2} \cdot \check{E}_u^{\pi_1}
             = E_u^{\pi_1} \cdot \check{E}_u^{\pi_1}  =1. $
       This shows that one has indeed $\psi^* \check{E}_u^{\pi_1} = \check{E}_u^{\pi_2}$. 
   \end{proof}

   The following definition is inspired by the approaches of 
Favre-Jonsson in \cite[Appendix A]{favre-jonsson:dynamicalcompactifications} 
and Jonsson \cite[Section 7.3.6]{jonsson:berkovich}: 

  \begin{defi}  \label{def:bra}
    Let $u,v$ be two possibly equal prime divisorial valuations of $X$. Their {\bf bracket} 
    is defined by: 
     $$\boxed{ \bra{u,v} }:= - \check{E}_u \cdot \check{E}_v \in \Q_+^*.$$
    Here $E_u$ and $E_v$ denote the representing divisors on a model on which  
    both of them appear. 
\end{defi}

By \refprop{invdual},   the bracket is independent of the choice of a model 
on which both $u$ and $v$ appear.   We get in this way a function: 
    $$
         \bra{\cdot, \cdot} \colon \genprimes{X} \times \genprimes{X} \to \Q_+^*.
    $$

Till now we have worked with total transforms of divisors living on models of $X$, 
that is, on smooth surfaces. Let us consider now the case of a divisor $A$ 
on $X$. If $A$ is a principal divisor, then one may define its total transform $\pi^* A$ 
by a resolution $\pi$ as the divisor of the pull-back of a defining function of $A$. 
The total transform is independent of the choice of defining function. 
Moreover, as a 
consequence of the projection formula \eqref{eqn:projform}, which is still true if one 
works with a proper birational morphism between normal surfaces, 
the intersection number of the total transform of $A$ with any exceptional divisor on $X_{\pi}$ is $0$.
This property was converted by Mumford \cite{mumford:1961} 
into a \textit{definition} of the total transform of a not necessarily principal divisor on $X$:

\begin{defi} \label{def:totransf}
   Let $A$ be a  divisor on $(X, x_0)$ and  $\pi: X_{\pi} \to X$ a resolution 
   of $X$. The {\bf total transform} of $A$ on $X^{\pi}$ is the $\Q$-divisor 
   $\pi^*A =  \strt{A}{\pi} +  \exct{A}{\pi}$ on $X^{\pi}$ such that:
     \begin{enumerate}
         \item $\boxed{ \strt{A}{\pi}} $ is the {\bf strict transform} of $A$ on $X^{\pi}$. 
            Its support is the closure of \linebreak $\pi^{-1}(|A| \setminus \{x_0\})$ in $X_{\pi}$, 
            each one of its irreducible components being endowed with the same 
            coefficient as its image in $X$. 
         \item The support of the {\bf exceptional transform} $\boxed{\exct{A}{\pi} } $ 
                of $A$ on $X^{\pi}$ 
               is included in the exceptional divisor $\exc {\pi}$. 
         \item $\pi^*A \cdot  E_u =0$ for each irreducible component $E_u$ of $\exc {\pi}$. 
     \end{enumerate}
\end{defi}

The fact that such a divisor exists and is unique comes from the fact that condition 
(3) of the definition may be written as a square linear system of equations whose 
unknowns are the coefficients of $\exct{A}{\pi}$ in the basis $(E_u)_{u \in \primes {\pi}}$ 
of $\ExcR {\pi}$, and whose matrix is the intersection matrix 
$(E_u \cdot E_v)_{u,v \in \primes {\pi}}$. This matrix is non-singular, by 
\refthm{intersform} (\ref{negdef}). Note that we make here a slight abuse 
of language, as one gets a matrix only after having chosen a total order on the set 
$\primes {\pi}$. 

Note also that in \refdef{totransf}, one allows $X_{\pi}$ to be \textit{any} 
model of $X$, without 
imposing it to be adapted in any sense to the divisor $A$. We say that $\pi$ is an {\bf embedded 
resolution} of $A$ if the total transform $\pi^* A$ is a divisor with normal crossings. 
In this case, each branch of $A$ has a strict transform on $X_{\pi}$ which 
intersects {\em transversally} a {\em unique} prime exceptional divisor.
Therefore, one has the following immediate consequence of \refdef{totransf}: 
 
 \begin{prop} \label{prop:branchcase}
    Assume that $A$ is a branch and that $\pi$ is an embedded resolution of it. 
    Let $E_a \in \primes {\pi}$ be the unique prime exceptional divisor which intersects the strict 
    transform of $A$.   Then: 
         $$  \exct{A}{\pi} = - \check{E}_a. $$
  \end{prop}
  
  Let us introduce the following denomination for the divisor $E_a$: 
  
  \begin{defi}  \label{def:reprdiv}
    Let $A$ be a branch on $X$ and $\pi$ be an embedded resolution of it. 
    The unique prime exceptional divisor $E_a \in \primes {\pi}$ which intersects 
    the strict transform of $A$ on $X_{\pi}$ is called the {\bf representing 
    divisor} of $A$ on $X_{\pi}$. 
 \end{defi}
  
  Using the notion of total transform of divisors from 
  \refdef{totransf},  Mumford defined in the following way in \cite{mumford:1961} the intersection  
  number of two divisors without common branches on $X$: 

\begin{defi} \label{def:mumfint}
   Let $A, B$ be two divisors on $X$ without common branches. Then their 
   {\bf intersection number} $\boxed{A \cdot B} \in \Q$ is defined by:
       $$A \cdot B := \pi^* A \cdot \pi^* B,$$
    for any resolution $\pi$ of $X$. 
\end{defi}

This definition is independent of the resolution. 
In the special case in which both $A$ and $B$ are branches, 
we get the following interpretation of the {\em bracket}:

\begin{prop}  \label{prop:intexcep}
   Let $A, B$ be two distinct branches on $X$.  Consider 
   an embedded resolution $X_{\pi}$ of the divisor $A+B$. If $E_a$ and 
   $E_b$ are the possibly coinciding representing divisors of $A$ and $B$ on $X_{\pi}$, then:  
      $$A \cdot B = \bra{a,b}.$$
\end{prop}

\begin{proof}
   According to \refdef{mumfint}, we have $A \cdot B = \pi^* A \cdot \pi^* B$. 
   By bilinearity of the intersection product, $\pi^* A \cdot \pi^* B = \pi^* A \cdot B_{\pi} + 
   \pi^* A \cdot \exct{B}{\pi}$. The second term of this sum vanishes, by the projection formula 
   \eqref{eqn:projform}:
      $\pi^* A \cdot \exct{B}{\pi} = A \cdot \pi_* \exct{B}{\pi} = A \cdot 0 =0.$
Hence, we get $A \cdot B = \pi^* A \cdot B_{\pi}  = A_{\pi} \cdot B_{\pi} + \exct{A}{\pi} \cdot B_{\pi}.$ 
The first term of this last sum vanishes, because our hypothesis that $\pi$ is an embedded resolution 
of the divisor $A + B$ shows that the strict transforms $A_{\pi}$ and $B_{\pi}$ are disjoint. 
Consider now 
the relation $\exct{A}{\pi} \cdot \pi^* B =0$, symmetrical of the relation 
$\pi^* A   \cdot  \exct{B}{\pi} =0$ 
used before.  Using again the bilinearity of the intersection product, it may be written 
$\exct{A}{\pi} \cdot B_{\pi}  + \exct{A}{\pi} \cdot \exct{B}{\pi}=0$. 
Therefore: 
   \begin{equation}   \label{for:intexcep}
   A \cdot B =  \exct{A}{\pi} \cdot B_{\pi} = - \exct{A}{\pi} \cdot \exct{B}{\pi} = 
       - \check{E}_a \cdot \check{E}_b = \bra{a,b}, 
    \end{equation}
  the penultimate equality being a consequence of \refprop{branchcase}, and the last one 
  being just the definition of the bracket. 
\end{proof}

Notice that the case $a=b$  in  \refprop{intexcep} may occur 
when the strict transforms $A_{\pi}$ and $B_{\pi}$ intersect the same irreducible  
component of $E(\pi)$.

The next consequence of \refprop{intexcep} will be used in the proof of \refprop{noud}:

\begin{cor}  \label{cor:intexcep}
     Let $\pi$ be a resolution of 
   $X$. Let $A, B$ be two distinct branches on $X$ such that 
   the strict transforms $A_\pi$ and $B_\pi$ are disjoint. Then: 
\[
A \cdot B = - \exct{A}{\pi} \cdot \exct{B}{\pi}. 
\]
\end{cor}

\begin{proof}
This results from the proof of \refprop{intexcep}, which uses 
the fact that the modification $\pi$ is an embedded resolution of $A+B$ only in
the last two equalities in (\ref{for:intexcep}), what precedes them needing only 
the hypothesis of disjointness of the strict transforms.
\end{proof}

\medskip
\subsection{The angular distance} 
$\:$  

\medskip
    In this section we recall the notion of \textit{angular distance} $\logfunc$ of prime 
    divisorial valuations (see \refdef{angdist}),  
    introduced in a greater generality by Gignac and the last author in 
    \cite{gignac-ruggiero:locdynnoninvnormsurfsing} and by the first three authors in a slightly different 
    form in \cite{gbgppp:2016} for the restricted class of \textit{arborescent} singularities. The definition 
    uses the bracket of \refdef{bra}. The fact that $\logfunc$ is indeed a 
    distance depends on a crucial inequality of Gignac and the last author, which we recall 
    in \refprop{crucial}. We conclude the section with a list of reformulations 
    of this inequality (see \refprop{MW}). 
\bigskip

Let $X_{\pi}$ be a model of $X$ and let $u,v \in \primes {\pi}$ be two prime divisorial  
valuations appearing on it. By \refthm{intersform} (\ref{negdef}), the intersection 
form on  $\ExcR {\pi}$ is negative definite. Let us apply the 
Cauchy-Schwartz inequality to its opposite bilinear form and to the vectors 
$\check{E}_u, \check{E}_v \in \ExcR {\pi}$. Using  
\refprop{negdual} and \refdef{bra}, we 
get the following inequalities: 
  \begin{equation} \label{eqn:CS}
        0 <  \bra{u,v}^2 \leq \bra{u,u} \cdot \bra{v,v}, 
  \end{equation}
  with equality if and only $u = v$. This allows to define: 

\begin{defi}   \label{def:angdist}
    The {\bf angular distance} of the prime divisorial valuations $u, v \in \primes {X}$ is: 
    \begin{equation} 
          \boxed{\logfunc(u,v) }:= - \log \dfrac{\bra{u,v}^2}{\bra{u,u} \cdot \bra{v,v}} \in [0, \infty). 
    \end{equation}
  \end{defi}
  
  As an immediate consequence of inequality \eqref{eqn:CS} and of the 
  characterization of the case of equality, one gets: 
  
  \begin{prop}  \label{prop:nonneg}
        For every pair of prime divisorial valuations $(u,v)$ of $X$, one has 
          $\logfunc(u,v) \geq 0$, with equality if and only if $u =v$. 
    \end{prop}

  \begin{rmk}
     A slightly different notion was introduced before by the first three authors 
     in \cite[Definition 4.11]{gbgppp:2016}, in the special case of \textit{arborescent} 
     normal surface singularities. It was introduced almost simultaneously by the 
     last author and Gignac for arbitrary \textit{semivaluations} of 
     $X$ in \cite[Definition 2.39]{gignac-ruggiero:locdynnoninvnormsurfsing}.
  \end{rmk}
          
    As indicated by the name chosen in \refdef{angdist}, 
    $\logfunc$ is indeed a metric on the set $\primes {X}$ (see \refprop{MW} 
    \ref{distform} below). But this fact is not immediate. 
    It is a consequence of an inequality of Gignac and the last author 
    (see  \refprop{crucial} below). 
    In order to state this inequality, we need the following graph-theoretical notion 
    (see \refssec{blockvertextree} for our vocabulary concerning graphs):

\begin{defi}   \label{def:separ}
      Let $a,b,c$ be three not necessarily pairwise distinct vertices of the connected 
       graph $\Gamma$. 
     One says that $c$ {\bf separates $a$ from $b$} in $\Gamma$ if:
   \begin{itemize}
        \item  either $c \in \{a,b \}$; 
        \item  or $a$ and $b$ belong to distinct connected components of the topological 
           space $\Gamma \setminus \{c \}$. 
   \end{itemize}
\end{defi}

We apply the previous notion of separation to the \textit{dual graphs} of the 
\textit{good models}  of $X$:

\begin{defi}  \label{def:dualgraph}
     Let $\pi: X_{\pi} \to X$ be a resolution of $X$. The resolution $\pi$ and 
     the model $X_{\pi}$ are called {\bf good} if their exceptional divisor has 
     normal crossings and its prime components are smooth. The {\bf dual graph} 
     $\boxed{ \dgr{\pi} }$ of a good model $X_{\pi}$  
     has vertex set $\primes {\pi}$ and set of edges between any two vertices 
     $u, v \in \primes {\pi}$ in bijection with the intersection points on $X_{\pi}$ between 
     the associated prime divisors $E_u$ and $E_v$. 
\end{defi} 

Here comes the announced inequality of Gignac and the last author (see
\cite[Proposition 1.10]{gignac-ruggiero:locdynnoninvnormsurfsing}), which is crucial for 
the present paper:

\begin{prop}(\cite[Proposition 1.10]{gignac-ruggiero:locdynnoninvnormsurfsing}) \label{prop:crucial}
     Let $X_\pi$ be a good model of the normal surface singularity 
     $X$, and let $E_u$, $E_v$ and $E_w$ 
     be not necessarily distinct exceptional prime divisors of $\pi$. Then one has the inequality:
    \begin{equation}\label{eqn:positivity_EFH}
         (- \check{E}_u \cdot \check{E}_v)(- \check{E}_v \cdot \check{E}_w) \leq 
         (- \check{E}_v \cdot \check{E}_v)(- \check{E}_u \cdot \check{E}_w),
     \end{equation}
    with equality if and only if $v$ separates $u$ and $w$ 
   in the dual graph $\dgr{\pi}$ of $X_{\pi}$.
\end{prop}

\begin{proof}
    Let us sketch a slight variant of the original proof. We work with the opposite 
    of the intersection form, which is positive definite. Denote therefore 
    $\bra{V_1, V_2} := - V_1 \cdot V_2$ for any $V_1, V_2 \in \ExcR {\pi}$. Inequality 
    \eqref{eqn:positivity_EFH} may be rewritten as: 
       \begin{equation}\label{eqn:positivity_bis}
         \bra{ \check{E}_u   - 
              \frac{\bra{ \check{E}_u , \check{E}_v}}{\bra{ \check{E}_v , \check{E}_v}} 
                   \check{E}_v  \:  , \:  \check{E}_w } \geq 0.
        \end{equation}
     Using Equation \eqref{eqn:devdiv},  we see that the truth of the previous inequality 
     for all $w \in \primes{\pi} $ and fixed $u, v \in \primes{\pi} $ is equivalent to the 
     following statement: 
         \begin{equation} \label{eqn:statem}
              \mbox{the divisor } \: \: \check{E}_u   - 
              \frac{\bra{ \check{E}_u , \check{E}_v}}{\bra{ \check{E}_v , \check{E}_v}}  \check{E}_v
                    \: \: \mbox{  is effective}.
         \end{equation}
     The key of the proof of \eqref{eqn:statem}   is to understand geometrically the 
     previous expressions. Consider the linear hyperplane 
     $\mathcal{H}_w$ of      
    $\ExcR {\pi}$ spanned by the vectors $E_a$, for $a \in \primes{\pi} \setminus \{w\}$. 
    Those vectors form a basis of the hyperplane $\mathcal{H}_w$. Look at the dual basis relative 
    to the restriction of $\bra{\cdot , \cdot }$ to $\mathcal{H}_w$. As can be verified by an immediate 
    computation, the vector corresponding to $E_u$ in this dual basis is 
    exactly the vector occuring in \eqref{eqn:statem}. Now let us apply 
    \refprop{dualcone} to the Euclidean space 
    $\left(\mathcal{H}_w, \bra{\cdot , \cdot }\right)$ 
    and the basis $(E_a)_{a \in \primes{\pi} \setminus \{w\}}$. We deduce that the 
    coefficients of the elements of its dual basis in the starting basis are non-negative, 
    which is exactly the statement \eqref{eqn:statem}. 
    
    There is a slight difference with the hypotheses of \refprop{dualcone}. 
    There one assumed that the basis could not be partitioned in two non-empty orthogonal 
    subsets. Here we are 
    in a situation in which the dual graph is not necessarily connected. Namely, as we 
    work in the hyperplane $\mathcal{H}_w$, we drop the component $E_w$ from the 
    exceptional divisor, therefore the dual graph of the remaining components gets 
    decomposed in a finite positive number of connected components. The associated 
    partition of $\primes{\pi} \setminus \{w\}$ induces an 
    orthogonal direct sum decomposition of $\mathcal{H}_w$, each term of this sum 
    having a connected dual graph. The dual basis of 
    $(E_a)_{a \in \primes{\pi} \setminus \{w\}}$ is the union of the dual bases of the individual 
    terms of this orthogonal direct sum. Apply then \refprop{dualcone} 
    to each such term. One gets in this way easily the characterization 
    of the case of equality in (\ref{eqn:positivity_bis}).     
\end{proof}

The point \ref{ineqspher} in the following reformulation of  \refprop{crucial}
was already stated by the third author in the summary \cite{popescu-pampu:2016} of the 
work \cite{gbgppp:2016}.

\begin{prop}  \label{prop:MW}
      Let  $X_\pi$ be a good model  of $X$, and let $E_u$, $E_v$ and $E_w$ 
     be not necessarily distinct exceptional prime divisors of $\pi$. Then the following 
      statements hold:
         \begin{enumerate}[label=(\Roman*)]
            \item   \label{ineqbra}
               $\bra{u,v} \cdot \bra{v, w} \leq \bra{v,v} \cdot \bra{u, w}$, with equality if and only 
               if $v$ separates $u$ from $w$ in the dual graph $\dgr{\pi}$.
            \item \label{distform} 
                  The function $\logfunc$ is a metric on the finite set $\primes {\pi}$, 
                 with equality in the triangle inequality  
                 $\logfunc(u,v) + \logfunc(v,w) \geq \logfunc(u,w)$  
                 if and only if $v$ separates $u$ from $w$ in $\dgr{\pi}$. 
            \item  \label{ineqspher}
                 Endow the real vector space $\ExcR{\pi}$ with the Euclidean structure
                 equal to the opposite of the intersection form. On its unit 
                 sphere, consider the pairwise distinct vectors which are positively proportional to 
                 $\check{E}_u, \check{E}_v, \check{E}_w$. Join them by shortest geodesics, 
                 obtaining a spherical triangle called simply $uvw$. This triangle has all its 
               angles in the interval $(0, \pi/2]$. Moreover, it is rectangle at $v$ if and only if 
               $v$ separates $u$ from $w$ in $\dgr{\pi}$.
         \end{enumerate}
\end{prop}

\begin{proof} The equivalence of the inequality \eqref{eqn:positivity_EFH} with the inequality
    \ref{ineqbra} and the assertion on the triangle inequality in \ref{distform} 
    are a simple consequence of \refdefs{bra}{angdist} and  \refprop{nonneg}.
        
    The reformulation \ref{ineqspher} needs a little more explanations. First, note that 
    inequality (\ref{eqn:positivity_EFH}) may be rewritten as: 
       $$     \dfrac{-\check{E}_u \cdot \check{E}_v}{\sqrt{(-\check{E}_u \cdot \check{E}_u) 
                      (-\check{E}_v  \cdot \check{E}_v)}}  
            \cdot \dfrac{-\check{E}_v  \cdot \check{E}_w}{\sqrt{(-\check{E}_v \cdot \check{E}_v) 
                      (-\check{E}_w  \cdot \check{E}_w)}}
             \leq   \dfrac{-\check{E}_u \cdot \check{E}_w}{\sqrt{(-\check{E}_u \cdot \check{E}_u) 
                      (-\check{E}_w  \cdot \check{E}_w)}}  .$$
     Measuring the angles using the opposite of the intersection form (which is indeed 
     a Euclidean metric on the real vector space $\ExcR{\pi}$, by  
     \refthm{intersform} (\ref{negdef})), the previous inequality 
     may be rewritten as: 
          \begin{equation} \label{eqn:spherineq}  
               \cos (\angle \check{E}_u \check{E}_v) \cdot \cos (\angle \check{E}_v \check{E}_w) 
                      \leq 
                \cos (\angle \check{E}_u \check{E}_w) .
           \end{equation}
             
      Recall now the spherical law of cosines for a geodesic triangle on a unit sphere, 
      whose edges have lengths denoted $a, b, c \in (0, \pi)$, the angle opposite to the 
      edge of length $a$ being denoted $A \in (0, \pi)$ (see for instance Prasolov and Tikhomirov 
      \cite[Section 5.1, page 87]{prasolov-tikhomirov}, Ratcliffe \cite[Theorem 2.5.3]{ratcliffe} or  
      Van Brummelen \cite[Chapter 6]{vanbrummelen}):
          \[ \cos a = \cos b \cdot \cos c +  \sin b \cdot \sin c  \cdot \cos A.\]  
      Applying it to the spherical triangle $uvw$, with preferred vertex $v$, 
      we see that the inequality \eqref{eqn:spherineq} 
      is equivalent to the fact that the angle at vertex $v$  
      belongs to the interval $(0, \pi/2]$. The fact that one has equality if and only if 
      the angle is $\pi/2$ is the content of the \textit{spherical Pythagorean theorem}, which may also 
      be obtained as a consequence of the spherical law of cosines.
\end{proof}

\begin{rmk} $\,$  
     We may speak about the spherical triangle with vertices at $u, v, w$, 
         without mentioning the model on which we work because, by \refprop{invdual}, 
          this triangle is independent of the model up to isometry. Note that a spherical triangle 
           may have $2$ or $3$ angles $\geq \pi /2$, but that in our case at most one angle 
          is equal to $\pi /2$, the two other ones being acute. This results from the fact that 
           \textit{if $v$ separates $u$ from $w$, then neither $u$ separates $v$ from $w$, nor 
            $w$ separates $u$ from $v$. }

   There exist other kinds of extensions of the usual Pythagorean theorem to the three kinds of bidimensional Riemannian geometries of constant curvature (see for instance 
      Maraner \cite{maraner} and Foote \cite{foote}).

For the moment we have no applications of the spherical geometrical 
       viewpoint \ref{ineqspher},  
       but we think that it is intriguing and that it is worth formulating, as a very 
       vivid way of remembering the inequality of \refprop{crucial}. 
\end{rmk}

\medskip
\subsection{A reformulation of the ultrametric problem} \label{ssec:reformultrametric}
$\:$  

\medskip
    In this section we begin the study of the function $\udb{L}$ 
    introduced by the first three authors in \cite{gbgppp:2016}, defined whenever $L$ is a fixed 
    branch on the normal surface singularity $X$. 
    Given a finite set $\cF$  of branches,  in \refcor{equivultr} 
    we  reformulate the condition that for every  branch $L \in \cF$
    the function $u_L$ is an ultrametric on $\cF \setminus \{L\}$
    as the condition that the angular distance on 
    $\cF$ is \textit{tree-like}.
   Then we recall the correspondence between tree-like distances
    on finite sets $\cF$ and metric trees having a subset of vertices labeled by $\cF$ 
    (see \refprop{addtree}).
\bigskip

Let $L$ be a fixed branch on $X$. If $A, B$ are two other branches, assumed 
to be distinct from $L$, let us define the following (see \cite{gbgppp:2016}): 

\begin{equation}  \label{eqn:potultram}
 \boxed{\udb{L}(A, B)} := \left\{ \begin{array}{lcl}
              \dfrac{(L \cdot A) \: (L \cdot B)}{A \cdot B}, \: & \mbox{ if }  & A \neq B, \\
              \\
              0 , \:  & \mbox{ if }  & A = B .
                    \end{array} \right. 
\end{equation}

The following vocabulary was introduced in \cite{gbgppp:2016}:

\begin{defi}   \label{def:arbsing}
   A normal surface singularity is called {\bf arborescent} if the dual graphs of its 
   good models are trees. 
\end{defi}

In  \cite[Theorem 4.18]{gbgppp:2016}, the first three authors proved the following theorem, 
as a generalization of a theorem of P\l oski \cite{ploski:1985} 
concerning the case where $X$ is \textit{smooth}: 

\begin{thm}  \label{thm:initultram}
     If $X$ is an arborescent singularity, then for every branch $L$ on $X$, the 
     function $u_L$ is an ultrametric on the set 
     $\branches {X} \setminus \{ L \}$ of branches on $X$ which are distinct from $L$. 
\end{thm}

The present paper is an outgrowth of our desire to understand in which measure 
\refthm{initultram} extends to other normal surface singularities. 

Let us begin with a reformulation of the ultrametric inequality for $\udb{L}$, whose 
simple proof is left to the reader:

\begin{prop}  \label{prop:reformultra}
Let  $L, A, B, C$ be four pairwise distinct branches on $X$.
Consider an embedded resolution $\pi$ of their sum. Denote by $l, a, b, c$ the 
prime divisorial valuations corresponding to the representing divisors  on $X_{\pi}$ of $L, A, B$ 
and respectively $C$  (see \refdef{reprdiv}).
Then the following inequalities are equivalent, as well as the corresponding equalities:
\begin{enumerate}
\item $\udb{L}(A, B) \leq \max \{ \udb{L}(A, C), \udb{L}(B, C)\}$. 
\item $(A \cdot B)  (L \cdot C) \geq \min  \{ 
                   (A \cdot C)  (L \cdot B), \:   (B \cdot C)  (L \cdot A) \}$.
\item $\bra{a, b} \cdot \bra{l, c} \geq \min  \{ 
                  \bra{a, c} \cdot \bra{l, b}, \:   \bra{b, c} \cdot \bra{l, a} \}$.
\item $\rho(a, b) + \rho(l, c) \leq \max  \{ 
                   \rho(a, c) + \rho(l, b), \:  \rho(b, c) + \rho(l, a) \}$.
\end{enumerate} 
\end{prop}

The next proposition is subtler:

\begin{prop}   \label{prop:subtle}
    Let $\cF$ be a set of branches on $X$. If $u_L$ is an ultrametric on $\cF \setminus \{L\}$
    for one branch $L$ in $\cF$, then the same is true for any branch of $\cF$. 
\end{prop}

  \begin{proof} This proof is inspired by the explanations of B\"ocker and Dress in 
       \cite[Lemma 6, Corollary 7, Remark 5]{bocker-dress:1998}. 
       Let $L$ and $M$ be two distinct branches on $X$. We assume that $u_L$ is an 
       ultrametric on $\cF \setminus \{L\}$. We want to prove that $u_M$ is an ultrametric 
       on $\cF \setminus \{M\}$. 
       
       Consider three pairwise distinct branches 
       $A, B, C$ in $\cF \setminus \{M\}$ 
       (if this set has less then three elements, then there is nothing to prove). If 
       $L \in \{ A, B, C \}$, then the equivalence of (1) and (2) 
       in \refprop{reformultra} shows that the ultrametric inequalities of the restriction of $u_M$ 
       to $\{A, B, C\}$ are equivalent to the ultrametric inequalities of the restriction 
       of $u_L$ to $\{M, A, B, C\} \setminus \{L\}$. 
       
       Assume now that $L \notin \{A, B, C\}$. Using again the equivalence of (1) and (2) 
       in \refprop{reformultra}, we see that the fact that $u_M$ is an ultrametric in restriction to 
       $\{A, B, C\}$ is equivalent to the fact that among the products 
       $ (B \cdot C)  (M \cdot A), (A \cdot C)  (M \cdot B),  (A \cdot B)  (M \cdot C)$, 
       two are equal and the third one is not less than them. An immediate computation 
       shows that this is equivalent to the fact that:  
         \begin{equation}  \label{eq:reformbranch}
            \begin{array}{c}
                 \mbox{\em among the products }  u_L(B, C) \cdot u_L(M, A), u_L(A, C) \cdot  u_L(M, B), \\  
                u_L(A,B) \cdot  u_L(M,C),  \mbox{\em two are equal and the third one is 
                  not greater than them. }
            \end{array}
         \end{equation}
         
        This is the statement which we will prove. 
        If the six values taken by $u_L$ in restriction to pairs of distinct elements of the set 
        $\{M, A, B, C\}$ are equal, then the assertion (\ref{eq:reformbranch}) is obvious, the three 
        products being equal. 
        
        Assume therefore that not all six values are equal. In order to follow the next reasoning, 
        we recommend to the reader to draw the edges of a tetrahedron with vertices 
        $M, A, B, C$ and to look successively at its faces. The basic fact which will be used 
        many times for various triples, is that in an ultrametric space, among the distances 
        between three points, two are equal and the third one is not bigger than them.

        Up to permuting the labels $M, A, B, C$, 
        we may consider that $u_L(M, A) > u_L(A, B)$. As $u_L$ is ultrametric on 
        $\{M, A, B \}$, we get the relations 
        $u_L(M, A) = u_L(M, B) > u_L(A, B)$. Let us compare now $u_L(M, A)$ to $u_L(M, C)$:
        \medskip
                       
           \noindent
           $\bullet$ {\bf Suppose that $u_L(M, C) < u_L(M, A)= u_L(M, B)$.} As $u_L$ is ultrametric on  
                   $\{M, A, C\}$ and on $\{M, B, C\}$, we deduce that $u_L(A, C)= u_L(M, A)$ and 
                   $u_L(M, B) = u_L(B, C)$. Therefore $u_L(A, C)= u_L(M, A) = u_L(M, B) = u_L(B, C)$ 
                   and this number is strictly bigger than both $u_L(A, B)$ and $u_L(M, C)$. Therefore:
                      $$  u_L(B, C) \cdot u_L(M, A) =  u_L(A, C) \cdot  u_L(M, B) > 
                             u_L(A,B) \cdot  u_L(M,C).$$

           \noindent
              $\bullet$ {\bf Suppose that $u_L(M, C) = u_L(M, A)= u_L(M, B)$.} As $u_L$ is ultrametric 
                   on $\{A, B, C\}$, we have the relations $u_L(A, B) \leq u_L (B, C) = u_L(C, A)$, 
                   up to permutation.   Therefore:
                       $$  u_L(B, C) \cdot u_L(M, A) =  u_L(A, C) \cdot  u_L(M, B) \geq 
                             u_L(A,B) \cdot  u_L(M,C).$$

           \noindent
              $\bullet$ {\bf Suppose that $u_L(M, C) > u_L(M, A)= u_L(M, B)$.} Using again the fact 
                  that $u_L$ is ultrametric on  
                   $\{M, A, C\}$ and on $\{M, B, C\}$, we deduce that $u_L(C, A) = u_L(M, C) = u_L(B, C)$. 
                   Therefore we get again:
                      $$  u_L(B, C) \cdot u_L(M, A) =  u_L(A, C) \cdot  u_L(M, B) > 
                             u_L(A,B) \cdot  u_L(M,C).$$ 
         We see that the assertion (\ref{eq:reformbranch}) is true in all cases, which proves the 
         proposition. 
  \end{proof}

In \refprop{reformultra}, the branches $L, A, B, C$ were fixed. 
By applying this proposition to all the quadruples in a finite set of branches $\cF$, and by using also  
\refprop{subtle}, we get immediately:

\begin{cor}  \label{cor:equivultr}
    Let $\cF \subset \branches{X}$ be a finite set of branches on $X$. Consider an 
    embedded resolution $\pi$ of their sum and denote by $\cF_{\pi} \subset \primes {\pi}$ 
    the set of prime exceptional divisors representing the elements 
    of $\cF$ in $X_{\pi}$ according to \refdef{reprdiv}. Then the following properties are equivalent: 
       \begin{enumerate}
             \item 
                 For some $L \in  \cF$, the function $\udb{L}$ is an 
                 ultrametric on $\cF \setminus \{L\}$.

             \item For every $L \in \cF$, the function $\udb{L}$ is an 
                 ultrametric on $\cF \setminus \{L\}$.
                 
             \item The bracket  $\bra{\cdot, \cdot}$ satisfies the inequality: 
                 $$\bra{a, b} \cdot \bra{l, c} \geq \min  \{ 
                   \bra{a, c} \cdot \bra{l, b}, \:   \bra{b, c} \cdot \bra{l, a} \}, \: \: \mbox{for all} \: \:  (a,b,c,l) \in 
                        (\cF_{\pi})^4 .$$
                   
             \item \label{reform4point} 
                 The angular distance  $\rho$ satisfies the inequality: 
                  $$\rho(a, b) + \rho(l, c) \leq \max  \{ 
                   \rho(a, c) + \rho(l, b), \:  \rho(b, c) + \rho(l, a) \},  \: \: \mbox{for all} \: \:  
                        (a,b,c,l) \in (\cF_{\pi})^4.$$
        \end{enumerate}
\end{cor}

Let us introduce the following vocabulary concerning  the metrics which 
satisfy condition (\ref{reform4point}) of  \refcor{equivultr}: 

\begin{defi}   \label{def:treelike}
Let $S$ be a finite set. One says that a distance $\delta$ on $S$ is {\bf tree-like} if,  for all $(a,b,c,d) \in S^4$, one has the following \emph{4-point condition}:
\begin{equation}   \label{eqn:4pointmax}
\delta(a,b) + \delta(c,d) \leq \max \{ \delta(a,c) + \delta(b,d) , \delta(a,d) + \delta(b,c) \}.
\end{equation}

This means that, up to a permutation of the three sums, one has:
\begin{equation}   \label{eqn:4point}
\delta(a,b) + \delta(c,d) \leq  \delta(a,c) + \delta(b,d)  =  \delta(a,d) + \delta(b,c).
\end{equation}
\end{defi}

The term \textit{4-point condition} was introduced by Buneman in \cite{buneman:1974}. 
We chose the name tree-like for the previous kind of metrics because such finite 
metric spaces may be interpreted geometrically as special kinds of trees (see 
\refprop{addtree} below). Let us introduce first more vocabulary about trees: 

 \begin{defi}  \label{def:labtree}
       A {\bf finite tree} is a finite simply connected simplicial complex of dimension $1$. 
       The {\bf convex hull} $\boxed{\cvxh {\cF}}$ of a set $\cF$ of vertices of a tree 
       is the subtree obtained as the union of the paths joining pairwise the elements 
       of $\cF$. 
       If $S$ is a finite set, then an {\bf $S$-tree} is a finite tree whose set of vertices 
       contains the set $S$ and such that all its vertices of valency $1$ or $2$ 
       are elements of $S$. An {\bf isomorphism} of $S$-trees is an isomorphism of 
       trees which is the identity in restriction to the set $S$. 
 \end{defi}

 Given two $S$-trees, the fact that all their vertices of valency $1$ are elements 
  of $S$ implies that there exists at most one isomorphism between them. 
  When $S$ has $4$ elements,  there are exactly $5$ different $S$-trees up to isomorphism. 
  They are represented in  \reffig{fivetrees}, together with the names we will 
  use for them in the sequel. 
  
  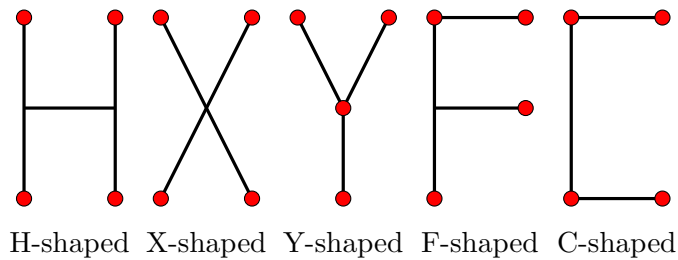
\begin{figure}
\centering
\begin{tikzpicture} [scale=0.6]

    \draw[very thick](0,-2) -- (0,2) ;
     \draw[very thick](2,-2) -- (2,2) ;
     \draw[very thick](0,0) -- (2,0) ; 
         \node[draw,circle,inner sep=2pt,fill=red] at (0,-2) {};
       \node[draw,circle,inner sep=2pt,fill=red] at (0,2) {};
       \node[draw,circle,inner sep=2pt,fill=red] at (2,-2) {};
       \node[draw,circle,inner sep=2pt,fill=red] at (2,2) {};
            \draw   (1,-2.5) node[below] {H-shaped};
       
       \draw[very thick](3,-2) -- (5,2) ;
     \draw[very thick](5,-2) -- (3,2) ;
         \node[draw,circle,inner sep=2pt,fill=red] at (3,-2) {};
       \node[draw,circle,inner sep=2pt,fill=red] at (3,2) {};
       \node[draw,circle,inner sep=2pt,fill=red] at (5,-2) {};
       \node[draw,circle,inner sep=2pt,fill=red] at (5,2) {};
            \draw   (4,-2.5) node[below] {X-shaped};
       
       \draw[very thick](7,-2) -- (7,0) ;
     \draw[very thick](7,0) -- (6,2) ;
     \draw[very thick](7,0) -- (8,2) ; 
         \node[draw,circle,inner sep=2pt,fill=red] at (7,-2) {};
       \node[draw,circle,inner sep=2pt,fill=red] at (7,0) {};
       \node[draw,circle,inner sep=2pt,fill=red] at (6,2) {};
       \node[draw,circle,inner sep=2pt,fill=red] at (8,2) {};
           \draw   (7,-2.5) node[below] {Y-shaped};
       
        \draw[very thick](9,-2) -- (9,2) ;
     \draw[very thick](9,2) -- (11,2) ;
     \draw[very thick](9,0) -- (11,0) ; 
         \node[draw,circle,inner sep=2pt,fill=red] at (9,-2) {};
       \node[draw,circle,inner sep=2pt,fill=red] at (9,2) {};
       \node[draw,circle,inner sep=2pt,fill=red] at (11,2) {};
       \node[draw,circle,inner sep=2pt,fill=red] at (11,0) {};
           \draw   (10,-2.5) node[below] {F-shaped};
       
        \draw[very thick](12,-2) -- (12,2) ;
     \draw[very thick](12,-2) -- (14,-2) ;
     \draw[very thick](12,2) -- (14,2) ; 
         \node[draw,circle,inner sep=2pt,fill=red] at (12,-2) {};
       \node[draw,circle,inner sep=2pt,fill=red] at (12,2) {};
       \node[draw,circle,inner sep=2pt,fill=red] at (14,-2) {};
       \node[draw,circle,inner sep=2pt,fill=red] at (14,2) {};
           \draw   (13,-2.5) node[below] {C-shaped};
    
           \end{tikzpicture}  
        \caption{The $5$ possible $S$-trees, when $S$ has $4$ elements.} 
        \label{fig:fivetrees}
     \end{figure}

 \begin{defi} \label{def:metrictree}
       A {\bf metric tree} is a finite tree endowed with a map from its set of edges 
       to the set of positive real numbers. The number associated to an edge is called 
       its {\bf length}. The {\bf induced distance} of a  
       metric $S$-tree is the distance on $S$ associating to each pair of elements  
       of $S$ the sum of length of the edges lying on the unique path joining them in the tree. 
 \end{defi}
 
 An example of metric $S$-tree is shown in  \reffig{lengthtree}. Here 
 $S  = \{a, \dots , e\}$. Denoting by 
 $\delta$ the induced distance on $S$, one has for instance $\delta(a, d) = 3+2+2$ and 
 $\delta(b, c)= 2 + 1$. 
  
  \begin{figure}
\centering
\begin{tikzpicture} [scale=0.7]

    \draw[very thick](0.4,0) -- (-1,0) ;
     \draw[very thick](0,0) -- (2,0) ;
    \draw[very thick](2,0) -- (4,2) ; 
     \draw[very thick](2,0) -- (4,0) ; 
    \draw[very thick](2,0) -- (4,-2) ;
    
       \draw (1.2,0) node[above]{$2$} ;
       \draw (-0.3,0) node[above]{$3$} ;
         \draw (3,1.2) node[above]{$1$} ;
         \draw (3.1,0) node[above]{$2$} ;
         \draw (3,-1.2) node[below]{$3$} ;
         
           \draw   (-1,-0.2) node[below] {$a$};
           \draw  (0.4,-0.2) node[below] {$b$};
           \draw  (4,1.8) node[below] {$c$};
            \draw (4,-0.2) node[below] {$d$};
            \draw (4,-2.2) node[below] {$e$};

       \node[draw,circle,inner sep=2pt,fill=red] at (0.4,0) {};
       \node[draw,circle,inner sep=1pt,fill=black] at (2,0) {};
       \node[draw,circle,inner sep=2pt,fill=red] at (-1,0) {};
       \node[draw,circle,inner sep=2pt,fill=red] at (4,2) {};
       \node[draw,circle,inner sep=2pt,fill=red] at (4,0) {};
       \node[draw,circle,inner sep=2pt,fill=red] at (4,-2) {};
    \end{tikzpicture}  
        \caption{An $\{a,b,c,d,e\}$-tree endowed with a length function.} 
        \label{fig:lengthtree}
     \end{figure}
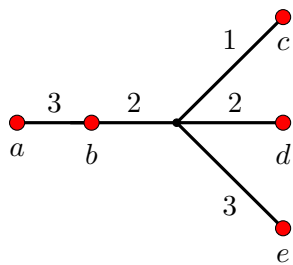

\medskip

It is immediate to check that the distance induced by a metric $S$-tree on the 
finite set $S$ satisfies the $4$-point condition. Therefore, it is tree-like, in the sense 
of \refdef{treelike}. Conversely, one has the 
following proposition (see Buneman's paper \cite{buneman:1974} and 
the successive generalizations of Bandelt and Steel \cite{bandelt-steel:1995} and 
 B\"ocker and Dress \cite{bocker-dress:1998}):

   \begin{prop}  \label{prop:addtree}
          Let $S$ be a finite set and $\delta$ be a distance on it. 
         If $\delta$ is tree-like, then there exists a unique $S$-tree $T$ endowed with 
         a length function such that the induced distance on $S$ is equal to $\delta$.             
  \end{prop}
  
  The main idea of the proof of the previous proposition is that an $S$-tree is determined 
  up to isomorphism by the isomorphism types of the convex hulls of all quadruples of 
  elements of $S$, which are in turn determined by the inequalities which are equalities 
  in the $4$-point 
  condition and in the triangle inequalities concerning them. 
  More precisely, given a quadruple $Q \subset S$ (see  \reffig{fivetrees}): 
      \begin{itemize} 
         \item the $H$-shaped and $X$-shaped $Q$-trees 
               are those $Q$-trees for which one has only strict triangle inequalities;  
          among them, the $H$-shaped tree is characterized by the fact that 
             one has a strict inequality in the $4$-point condition \eqref{eqn:4point}, 
             for a convenient labeling of the elements of $Q$ by the letters $a,b,c,d$; 
          \item the $Y$-shaped $Q$-trees are those $Q$-trees such that for 
              exactly one triple of points of $Q$, all the corresponding 
              triangle inequalities are strict;
          \item the $F$-shaped $Q$-trees are those $Q$-trees such that 
             for exactly two triples of points of $Q$, one of the corresponding triangle inequalities 
             is an equality;
          \item the $C$-shaped $Q$-trees are those $Q$-trees such that for 
            all triple of points of $Q$, one of the corresponding triangle inequalities is an 
            equality.
       \end{itemize}

  \refprop{addtree} allows us to define:
  
  \begin{defi}  \label{def:treehull}
      Let $\delta$ be a tree-like metric on a finite set $S$. Then the unique 
      $S$-tree endowed with a length function such that the induced distance 
      on $S$ is equal to $\delta$ is called the {\bf tree hull} of the metric 
      space $(S,\delta)$. 
  \end{defi}

\medskip
\subsection{A theorem about special metrics on the set of vertices of a graph}
\label{ssec:blockvertextree}
$\:$  

\medskip
      Let $X_{\pi}$ be a good model of $X$. Consider the angular distance $\logfunc$ 
      on the vertex set $\Vgr{\dgr {\pi}}= \primes {\pi}$ of the associated dual graph 
      $\dgr {\pi}$.
     In \refprop{MW}, we saw that the cases of equality in the  
     triangle inequalities associated to the metric space $(\Vgr{\dgr {\pi}}, \logfunc)$ 
     are characterized by separation properties in $\dgr {\pi}$.          
     The aim of this section 
     is to prove that if a metric $\delta$ on the set of vertices $\Vgr{\Gamma}$ 
     of a connected graph $\Gamma$ satisfies this 
     kind of constraint, then it becomes tree-like (in the sense of \refdef{treelike}) 
     in restriction to special types of subsets $\cF$ of $\Vgr{\Gamma}$
     (see \refthm{valblocks}). Moreover, the tree hull of $(\cF, \delta)$ 
     (according to \refdef{treehull}) may be described as the convex hull 
     of $\cF$ in a tree canonically associated to the graph $\Gamma$, its \emph{brick-vertex tree}  
     $\bvt{\Gamma}$ (see \refdef{brick-vertex}). 
\bigskip

In the sequel, we will use the following notion of graph:

\begin{defi}   \label{def:graphcell}
   A {\bf graph} $\Gamma$ is a finite cell complex of dimension 
at most $1$. In particular, it may have loops or multiple edges, and it may have 
connected components which are simply points. We will denote 
by $\boxed{\Vgr {\Gamma}}$ its set of vertices and by $\boxed{\Egr {\Gamma}}$ its set of edges. 
The {\bf valency} of a vertex $v$ of $\Gamma$ 
is the number of germs of edges adjacent to $v$ (a loop based at $v$ counting twice, 
as it contributes with two germs in this count). 
\end{defi}

If we want to insist on the graph $\Gamma$ in which 
we compute the valency (in situations where we deal with several graphs at the same time), 
we will speak about the {\bf $\Gamma$-valency} of a vertex $v$. 

It will be important for us to look at the edges of a connected graph 
$\Gamma$ according to their separation properties: 

\begin{defi} \label{def:sepgraph}
    Let $\Gamma$ be a connected graph. 
    A {\bf cut-vertex} of $\Gamma$ is a vertex whose removal disconnects $\Gamma$. 
    A {\bf bridge} of $\Gamma$ is an edge 
    such that the removal of its interior disconnects $\Gamma$.  
      The graph $\Gamma$ is called {\bf  separable} if it admits at least one cut-vertex 
      (see \reffig{sep}). 
    Otherwise, it is called  {\bf nonseparable}. 
  \end{defi}
  
  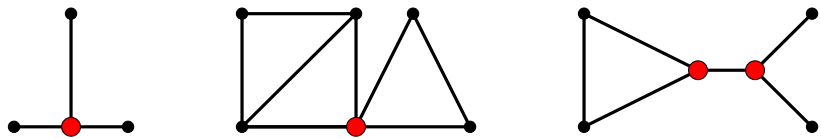
\begin{figure}
\centering
\begin{tikzpicture} [scale=1.5]

    \draw[very thick](0,0) -- (0.5,0) -- (0.5,1) ;
    \draw[very thick](1,0) -- (0.5,0) ;
       \node[draw,circle,inner sep=1.5pt,fill=black] at (0,0) {};
       \node[draw,circle,inner sep=2.5pt,fill=red] at (0.5,0) {};
       \node[draw,circle,inner sep=1.5pt,fill=black] at (0.5,1) {};
       \node[draw,circle,inner sep=1.5pt,fill=black] at (1,0) {};

     \draw[very thick](2,0) -- (3,0) -- (3,1) -- (2,1) -- (2,0)--(3,1);
     \draw[very thick](3,0) -- (4,0) -- (3.5,1) -- (3,0) -- (2,0);
        \node[draw,circle,inner sep=1.5pt,fill=black] at (2,0) {};
        \node[draw,circle,inner sep=2.5pt,fill=red] at (3,0) {};
        \node[draw,circle,inner sep=1.5pt,fill=black] at (3,1) {};
        \node[draw,circle,inner sep=1.5pt,fill=black] at (2,1) {};
         \node[draw,circle,inner sep=1.5pt,fill=black] at (4,0) {};
        \node[draw,circle,inner sep=1.5pt,fill=black] at (3.5,1) {};
     
     \draw[very thick](6,0.5) -- (6.5,0.5) -- (7,0) ;
     \draw[very thick](6.5,0.5) -- (7,1) ;
     \draw[very thick](5,0) -- (5,1) -- (6, 0.5) -- (5,0) ;
        \node[draw,circle,inner sep=1.5pt,fill=black] at (5,0) {};
        \node[draw,circle,inner sep=1.5pt,fill=black] at (5,1) {};
        \node[draw,circle,inner sep=2.5pt,fill=red] at (6,0.5) {};
        \node[draw,circle,inner sep=2.5pt,fill=red] at (6.5,0.5) {};
        \node[draw,circle,inner sep=1.5pt,fill=black] at (7,0) {};
        \node[draw,circle,inner sep=1.5pt,fill=black] at (7,1) {};
       
    \end{tikzpicture}  
        \caption{A few separable graphs and their cut-vertices marked in red.} 
        \label{fig:sep}
     \end{figure}

    The only nonseparable graphs which are trees are the segments. 
     All the other nonseparable graphs have the property that any two 
    of their edges are contained in a {\bf circuit}, that is, 
    a union of edges whose underlying topological space is homeomorphic to a circle. 
    The trees may be characterized as the connected 
    graphs all of whose edges are bridges. 
       
       Every connected graph contains a distinguished family of nonseparable subgraphs, 
       its \textit{blocks}, among which we distinguish the \textit{bricks}  and the \textit{bridges}:
              
\begin{defi}     \label{def:blocks}   
    The {\bf blocks} of a connected graph $\Gamma$ are its maximal subgraphs 
      which are nonseparable (see \reffig{bricktree}). 
     A block which is equal to an edge of $\Gamma$ is called a {\bf bridge}, 
     otherwise it is called a {\bf brick}. 
\end{defi}

The notions of bridge introduced in Definitions \ref{def:sepgraph} and \ref{def:blocks}   are 
equivalent. 

The blocks of a connected graph $\Gamma$ may be characterized as the unions of edges of each equivalence class for the following equivalence relation on the set $\Egr{\Gamma}$: 
two edges are equivalent if they are either equal or if they are both contained in the same circuit.  
Trees may be characterized as the connected finite graphs which have no bricks.

It is elementary to check that the following construction leads indeed to a tree: 

\begin{defi}  \label{def:brick-vertex}
     The {\bf brick-vertex tree} $\boxed{\bvt{\Gamma}}$ of a connected graph $\Gamma$ 
     is the tree whose  vertex set is the union of the set of 
     bricks of $\Gamma$ and of the set of its vertices. 
     The set of its edges consists of the bridges of $\Gamma$, 
     and of new edges connecting a brick of $\Gamma$ 
     to a vertex of $\Gamma$ (seen as vertices of $\bvt{\Gamma}$)  
     if and only if the brick contains the vertex.
     A vertex of $\bvt{\Gamma}$ associated to a brick of $\Gamma$ 
     will be called a \emph{brick-vertex}. 
         \end{defi}

     If $a$ is a vertex (resp.  if $B$ is a brick) of $\Gamma$,  we will denote by $\overline{a}$  
     (resp.   $\overline{B}$) 
     the vertex of $\bvt{\Gamma}$ defined by it. If $e= \{ a, b \}$ is a bridge of $\Gamma$, 
     then $\overline{e}
     = \{ \overline{a}, \overline{b} \}$ is also a bridge of $\bvt{\Gamma}$.
     Similarly, if $\cF$ is a set of vertices 
     of $\Gamma$, we denote by $\overline{\cF}$ the same set seen as a set of 
     vertices of $\bvt{\Gamma}$.

     \medskip 

Examples of planar brick-vertex trees are shown in \reffigs{bricktree}{biggraph}.
The bricks are emphasized by shading the plane regions spanned by their vertices and edges.

\begin{figure}[ht]
\centering
\def\svgwidth{0.4 \columnwidth}
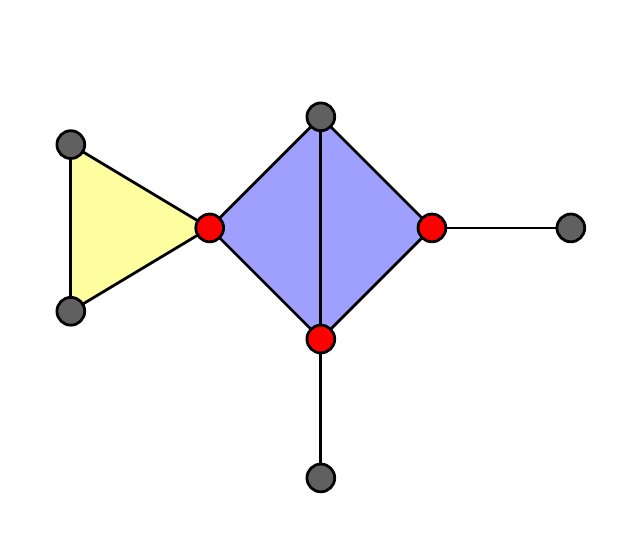
\hspace{1cm}
\def\svgwidth{0.4 \columnwidth}
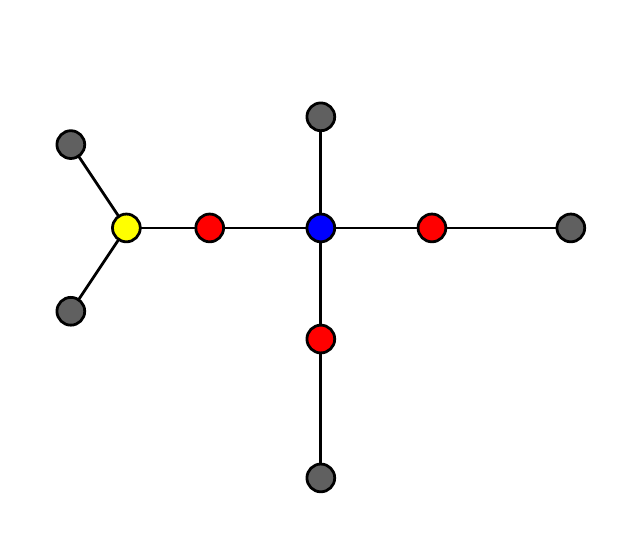
\caption{The brick-vertex tree of a connected graph.}\label{fig:bricktree}
\end{figure}

\begin{rmk}   \label{rmk:prevblocktree}
    Whitney introduced the blocks of a finite graph in his 1932 paper \cite{whitney:1932}, 
    under the name of  \textit{components}. His definition was slightly different: the blocks were the 
    final graphs (necessarily inseparable) of a process which chooses at each step a cut-vertex 
    of the graph and decomposes the connected component which contains it into the 
    connected subgraphs which are joined at that vertex. The term \textit{block} seems to have 
    been introduced for this concept in Harary's 1959 paper \cite{harary:1959}. In Tutte's 1966 book 
    \cite{tutte:1966}, the blocks are called \textit{cyclic elements}, 
    a term originating from general topology (see \refrmk{cyclelemtheory}). The use of the term 
    \textit{brick} for the blocks which are not bridges seems to be new. 
   A construction related to the brick-vertex tree is known under the name of 
   \textit{cut-tree} (see Tutte's book \cite[Section 9.5]{tutte:1966}),  \textit{block-cut tree} 
   (see Harary's book \cite[Page 36]{harary:1969}), or  \textit{block tree} 
   (see Bondy and Murty's book \cite[Section 5.2]{bondy-murty:2008}). 
   In that construction, which was introduced by Gallai \cite{gallai:1964} and 
   Harary and Prins \cite{harary-prins:1966}, one considers only the set of cut-vertices 
   of $\Gamma$, instead of the full set of vertices, and all the 
   blocks, not only the bricks.  Later on, Kulli \cite{kulli:1976} 
   introduced the \textit{block-point tree} of a connected graph, in which one still 
   considers all the blocks, but also all the vertices, not only the cut-vertices.
\end{rmk}

The following proposition, which uses the notations explained after \refdef{brick-vertex}, 
is the reason why we introduced the notion of brick-vertex tree:

\begin{prop}  \label{prop:equivsep}
       Let $a, b,c$ be three not necessarily pairwise distinct vertices of 
       the connected graph $\Gamma$. Then the following properties are 
       equivalent: 
         \begin{enumerate}
           \item $a$ separates $b$ from $c$ in the graph $\Gamma$; 
           \item  $\overline{a}$ separates $\overline{b}$ from $\overline{c}$ 
               in the brick-vertex tree $\bvt{\Gamma}$. 
         \end{enumerate}
\end{prop}

\begin{proof}
First notice that if $b=c \neq a$, then $a$ does not separate $b$ and $c$ neither in 
$\Gamma$ nor in $\bvt{\Gamma}$, while if $a$ coincides with either $b$ or $c$, 
then it separates $b$ from $c$ both in $\Gamma$ and $\bvt{\Gamma}$ (see 
\refdef{separ}).
Hence, we may suppose that $a,b,c$ are pairwise distinct.

\medskip
\noindent
$\bullet$ {\bf Suppose first that $a$ does not separate $b$ from $c$ in $\Gamma$.}  
Therefore, there exists a path 
$\gamma$ joining $b$ and $c$ in $\Gamma \setminus \{a\}$.
Decompose $\gamma$ in a finite sequence of concatenating edges $e_j$ with 
endpoints $v_{j-1},v_{j}$ for $j=1, \ldots, n$, with $v_0=b$, $v_n = c$, and $v_j \neq a$ for all $j$.
We construct a path $\tilde{\gamma}$ joining $\overline{b}$ and $\overline{c}$ 
in $\bvt{\Gamma} \setminus \{a\}$ as follows:
  \begin{itemize} 
      \item[--] If $v_{j-1}$ and $v_j$ belong to a brick $B$, then we replace the edge $e_j$ with 
          the concatenation of the two edges $\{\overline{v_{j-1}}, \overline{B} \}$, $\{\overline{B},\overline{v_{j}}\}$ 
          of $\bvt{\Gamma}$.
      \item[--] If the edge $e_j$ connecting $v_{j-1}$ and $v_j$ is a bridge, then we 
      consider the associated edge $\overline{e_j} = \{\overline{v_{j-1}},\overline{v_j}\}$ of $\bvt{\Gamma}$. 
   \end{itemize}

\medskip
\noindent
$\bullet$ {\bf Suppose now that $\overline{a}$ does not separate $\overline{b}$ from 
$\overline{c}$ in $\bvt{\Gamma}$. } 
Therefore, there exists a path $\tilde{\gamma}$ joining $\overline{b}$ and $\overline{c}$ in 
$\bvt{\Gamma} \setminus \{\overline{a}\}$. Denote by $\epsilon_j=\{w_{j-1}, w_j\}$ 
the sequence of edges of $\tilde{\gamma}$   
(notice that, as $\bvt{\Gamma}$ is a \textit{tree},  the edges are determined by their extremities). 
Therefore, every vertex $w_j$ of this path corresponds to a vertex or to a brick of $\Gamma$. 
We construct a path $\gamma$ joining $b$ and $c$ in $\Gamma \setminus \{a\}$ as follows.  
The endpoints of every edge $\epsilon_j$ of $\tilde{\gamma}$ either correspond simultaneously  to 
vertices of $\Gamma$, or one corresponds to a vertex and the other to a brick of $\Gamma$.
In the first case, we define $e_j$ to be the unique bridge of $\Gamma$ which projects to 
$\epsilon_j$. 
In the second case, since the vertices $\overline{b}$ and $\overline{c}$ of $\bvt{\Gamma}$ 
correspond to vertices of $\Gamma$ we can assume, up to replacing $j$ by $j+1$ if necessary, that 
$w_{j-1} = \overline{v_{j-1}}$ and $w_{j+1} = \overline{v_{j+1}} $ correspond to vertices $v_{j-1}$ and $v_{j+1}$ of $\Gamma$, and that 
$w_j = \overline{B}$ corresponds to the unique brick $B$ containing them.
Notice that $a$ could be a vertex of $B$. Since $v_{j-1}$ and $v_{j+1}$ belong to $B$, 
there exist two paths of $\Gamma$ inside the brick $B$ joining $v_{j-1}$ to $v_{j+1}$, 
which intersect only at their endpoints. Therefore, at least one of them does not pass through $a$.
We define then $\gamma_{j-1, j+1}$ to be such a path avoiding $a$ and contained 
inside the brick $B$ of $\Gamma$. Finally, the path $\gamma$ of $\Gamma$ 
obtained as the union of all the  
previous elementary paths $e_j$ and $\gamma_{j-1, j+1}$  
joins indeed $b$ and $c$ without passing through $a$.
\end{proof}

\begin{rmk}\label{rmk:bridges_are_not_blocks}
    \refprop{equivsep} holds also if we replace the brick-vertex tree by Kulli's block-point tree 
    (see \refrmk{prevblocktree} for its definition), 
    the proof being completely analogous.
    In fact, we could work in this first part of the paper with the block-point tree of $\Gamma$.
    We chose to work with \refdef{brick-vertex} since it has the advantage 
    of extending directly to graphs of $\nR$-trees (see \refssec{brickrtree}). 
     Notice that for a tree $\Gamma$, its brick-vertex tree coincides with $\Gamma$, 
     while its block-point tree is isomorphic to the barycentric subdivision of $\Gamma$. 
\end{rmk}

By \refprop{equivsep}, the brick-vertex tree of $\Gamma$ encodes precisely the way in which 
the vertices of $\Gamma$ get separated by the elimination of one of them. 

Recall the reformulation of \refprop{crucial} 
given in  \refprop{MW} \ref{distform}. It states 
that if one looks at the angular distance $\logfunc$ on the vertex set $\Vgr{\dgr{\pi}}$ 
of the dual graph $\dgr{\pi}$ of a good model $X_{\pi}$ of $X$, then one has an equality 
$\logfunc(u,v) + \logfunc(v,w) = \logfunc(u,w)$  in the triangular inequality associated 
to the triple $(u,v,w)$ of vertices of  $\dgr{\pi}$ if and only if $v$ separates $u$ from 
$w$ in $\dgr{\pi}$. The following theorem, which is the main result of this section, 
describes special subsets of vertices of the graphs endowed with metrics having the 
same formal property (recall that the convex hull of a finite set of vertices of 
a tree was introduced in \refdef{labtree}):

\begin{thm}  \label{thm:valblocks}
   Let $\Gamma$ be a finite connected graph and $\delta : \Vgr{\Gamma}^2 \to [0, \infty)$ 
   a metric such that one has the equality:
   \begin{equation}\label{eqn:additive_graph_metric}
\delta(a,b) + \delta(b,c) = \delta(a,c)
   \end{equation}
   if and only if the vertex $b$ separates $a$ from $c$ in $\Gamma$. 
   Consider a set $\cF$ of vertices of $\Gamma$, and their convex hull $\cvxh{\overline{\cF}}$  
   in the brick-vertex tree $\bvt{\Gamma}$ of $\Gamma$.  
   If each brick of $\Gamma$ has $\cvxh{\overline{\cF}}$-valency at 
   most $3$, then the restriction of $\delta$ to $\cF$ is tree-like  
   and its tree hull (see \refdef{treehull}) is isomorphic as an $\cF$-tree to 
   $\cvxh{\overline{\cF}}$. 
\end{thm}

\begin{proof}
   Assume that $\cF \subset \Vgr{\Gamma}$ satisfies the hypotheses of the theorem. 
  Consider four pairwise 
  distinct points $a,b,c,d \in \cF$ and the convex hull 
    $\cvxh{\overline{a}, \overline{b}, \overline{c}, \overline{d}}$ 
    of their images in the brick-vertex  tree $\bvt{\Gamma}$.

    We will consider several cases, according to the shape of this convex hull.
   In every case we will prove that in restriction to $\{a,b,c,d\}$, 
   the metric $\delta$  satisfies the 4-point condition and 
    that the shape of $\cvxh{\overline{a}, \overline{b}, \overline{c}, \overline{d}}$ 
    is determined by the cases of equality in the $4$-point conditions and in 
    the triangle inequalities associated to the four triples of points 
    among $a,b,c$ and $d$ (see the explanations following 
    \refprop{addtree}). Then,  thanks to  \refprop{addtree}, we conclude 
     that the tree hull of $(\{a,b,c,d\}, \delta)$ in the sense of \refdef{treehull} 
     is indeed isomorphic as a $\{a,b,c,d\}$-tree to the convex hull 
   $\cvxh{\overline{a}, \overline{b}, \overline{c}, \overline{d}}$, 
   finishing the proof of the proposition.

   \medskip
   
  \noindent
  $\bullet$  {\bf Assume that $\cvxh{\overline{a}, \overline{b}, \overline{c}, \overline{d}}$   is 
           H-shaped.} 
           Denote by $\mu$ and $\nu$ the two $3$-valent vertices of 
 $\cvxh{\overline{a}, \overline{b}, \overline{c}, \overline{d}}$. We may assume, up to renaming 
 the four points, that $\mu$ and $\nu$ separate $\overline{a}$ and $\overline{b}$ from 
 $\overline{c}$ and $\overline{d}$, as illustrated in \reffig{Hshaped}.  
     We claim that there exists then a cut-vertex $p$ of $\Gamma$ with the following properties:
     \begin{enumerate}
         \item[(a)] $p$ separates both $a$ and $b$ from both $c$ and $d$; 
         \item[(b)] either $p$ does not separate $a$ from $b$ or it does not 
              separate $c$ from $d$. 
     \end{enumerate}

     \begin{figure}
\centering
\begin{tikzpicture} [scale=0.7]

    \draw[very thick](0,0) -- (-1,1) ;
    \draw[very thick] (0,0) -- (-1,-1) ; 
     \draw[very thick](0,0) -- (2,0) ;
    \draw[very thick](2,0) -- (3,1) ; 
     \draw[very thick](2,0) -- (3,-1) ; 
    
       \draw (-1,1) node[above]{$\overline{a}$} ;
       \draw (-1,-1) node[below]{$\overline{b}$} ;
       \draw (3, 1) node[above]{$\overline{c}$} ;
       \draw (3,-1) node[below]{$\overline{d}$} ;
       \draw (0, 0.2) node[above]{$\mu$} ;
       \draw (2, 0.2) node[above]{$\nu$} ;

       \node[draw,circle,inner sep=1pt,fill=black] at (0,0) {};
       \node[draw,circle,inner sep=1pt,fill=black] at (2,0) {};
       \node[draw,circle,inner sep=1.5pt,fill=red] at (-1,1) {};
       \node[draw,circle,inner sep=1.5pt,fill=red] at (-1,-1) {};
       \node[draw,circle,inner sep=1.5pt,fill=red] at (3,1) {};
       \node[draw,circle,inner sep=1.5pt,fill=red] at (3,-1) {};
    \end{tikzpicture}  
    \caption{The case of an $H$-shaped tree  in the proof of \refthm{valblocks}.} 
    \label{fig:Hshaped}
     \end{figure}

     In order to prove this, let us consider two cases: 
     \medskip

       \noindent 
       (i) {\em One of the points $\mu$ and $\nu$ of $\bvt{\Gamma}$ is a cut-vertex 
              of $\Gamma$.}
              %
            Assume for instance that $\mu = \overline{p}$, where $\overline{p}$ is a cut-vertex 
            of $\bvt{\Gamma}$. The convex hull 
            $\cvxh{\overline{a}, \overline{b}, \overline{c}, \overline{d}}$ 
            having the shape illustrated in  \reffig{Hshaped}, we see that $p$ has the 
            announced properties. 
            \medskip

        \noindent
        (ii) {\em Both points $\mu$ and $\nu$ of $\bvt{\Gamma}$ are bricks of $\Gamma$.} 
             By construction, all edges of $\bvt{\Gamma}$ join either two vertices 
             coming from $\Gamma$, or a brick-vertex with a vertex coming from $\Gamma$. 
             We deduce that there exists necessarily a cut-vertex 
             $\overline{p}$ in the interior of the 
             geodesic $[\mu \nu]$ of $\bvt{\Gamma}$. Again, since the convex hull 
            $\cvxh{\overline{a}, \overline{b}, \overline{c}, \overline{d}}$ 
            has  the shape illustrated in \reffig{Hshaped}, we see that $p$ has the 
            announced properties. 
    \medskip

     Using the fact 
     that $p$ satisfies properties (a) and (b) above, the hypothesis that 
     $\delta$ is a distance on $\Vgr{\Gamma}$  and 
     the characterization of the equality 
     in the triangle inequality, we get: 
     $$\begin{array}{c}
            \delta(a,b)  + \delta(c,d)     <  
          (\delta(a,p) + \delta (b,p) ) + (\delta(c,p) + \delta(d,p)) = 
           \\
            (\delta(a,p) + \delta (c,p) ) + (\delta(b,p) + \delta(d,p))   =   
            \delta(a,c) + \delta (b,d) =
             \\
             (\delta(a,p) + \delta (d,p) ) + (\delta(b,p) + \delta(c,p))  =  
              \delta (a,d) + \delta(b,c).
          \end{array}$$
   This shows that $\delta$ satisfies the $4$-point condition in restriction to 
   $\{a,b,c,d\}$, and that one has a strict inequality in this condition.  
  In addition, one has by \refprop{equivsep} and the hypothesis
    that there is no equality among the $4$ triangle 
   inequalities concerning triples of points among $a,b,c,d$.

      \medskip
\noindent
      $\bullet$ {\bf Assume that $\cvxh{\overline{a}, \overline{b}, \overline{c}, \overline{d}}$   is 
           X-shaped.}   
           Denote by $\mu$ the unique point of this graph which is of valency $4$ 
           (see Figure \ref{fig:Xshaped}). 
           By hypothesis, no brick of $\cvxh{\cF}$ is of valency $\geq 4$. Therefore, 
           $\mu = \overline{p}$, where $p$ is a separating vertex of $\Gamma$.  
           Moreover, $p$ separates pairwise the points $a,b,c,d$. Therefore: 
               $$\begin{array}{c} 
            \delta(a,b)  + \delta(c,d)   =  
            (\delta(a,p) + \delta (b,p) ) + (\delta(c,p) + \delta(d,p)) = \\
             (\delta(a,p) + \delta (c,p) ) + (\delta(b,p) + \delta(d,p)) =  
            \delta(a,c) + \delta (b,d) = 
             \\
            (\delta(a,p) + \delta (d,p) ) + (\delta(b,p) + \delta(c,p)) = 
            \delta (a,d) + \delta(b,c).
          \end{array}$$
           This shows again that $\delta$ satisfies the $4$-point relation in restriction to 
   $\{a,b,c,d\}$. As in the previous case, one has no equality among the $4$ triangle 
   inequalities concerning triples of points among $a,b,c,d$.

      \begin{figure}
         \centering
         \begin{tikzpicture} [scale=0.6]

    \draw[very thick](0,0) -- (-1,1) ;
    \draw[very thick] (0,0) -- (-1,-1) ; 
    \draw[very thick](0,0) -- (1,1) ; 
     \draw[very thick](0,0) -- (1,-1) ; 
    
       \draw (-1,1) node[above]{$\overline{a}$} ;
       \draw (-1,-1) node[below]{$\overline{b}$} ;
       \draw (1, 1) node[above]{$\overline{c}$} ;
       \draw (1,-1) node[below]{$\overline{d}$} ;
       \draw (0, 0.2) node[above]{$\mu$} ;

       \node[draw,circle,inner sep=1pt,fill=black] at (0,0) {};
        \node[draw,circle,inner sep=1.5pt,fill=red] at (-1,1) {};
       \node[draw,circle,inner sep=1.5pt,fill=red] at (-1,-1) {};
       \node[draw,circle,inner sep=1.5pt,fill=red] at (1,1) {};
       \node[draw,circle,inner sep=1.5pt,fill=red] at (1,-1) {};
    \end{tikzpicture}  
    \caption{The case of an $X$-shaped tree in the proof of  \refthm{valblocks}.} 
    \label{fig:Xshaped}
     \end{figure}

     \medskip

     In the remaining cases we assume that $\bar{a}$, $\bar{b}$, $\bar{c}$ and $\bar{d}$ are as in 
     \reffig{YFC}.

  \begin{figure}
\centering
\begin{tikzpicture} [scale=0.6]

       \draw[very thick](7,-2) -- (7,0) ;
     \draw[very thick](7,0) -- (6,2) ;
     \draw[very thick](7,0) -- (8,2) ; 
         \node[draw,circle,inner sep=2pt,fill=red] at (7,-2) {};
       \node[draw,circle,inner sep=2pt,fill=red] at (7,0) {};
       \node[draw,circle,inner sep=2pt,fill=red] at (6,2) {};
       \node[draw,circle,inner sep=2pt,fill=red] at (8,2) {};
       
       \draw (6,2) node[above]{$\overline{a}$} ;
       \draw (8,2) node[above]{$\overline{b}$} ;
       \draw (7, 0) node[right]{$\overline{d}$} ;
       \draw (7,-2) node[right]{$\overline{c}$} ;

           \draw   (7,-2.5) node[below] {Y-shaped};
       
        \draw[very thick](9,-2) -- (9,2) ;
     \draw[very thick](9,2) -- (11,2) ;
     \draw[very thick](9,0) -- (11,0) ; 
         \node[draw,circle,inner sep=2pt,fill=red] at (9,-2) {};
       \node[draw,circle,inner sep=2pt,fill=red] at (9,2) {};
       \node[draw,circle,inner sep=2pt,fill=red] at (11,2) {};
       \node[draw,circle,inner sep=2pt,fill=red] at (11,0) {};
       
       \draw (11,0) node[above]{$\overline{b}$} ;
       \draw (11,2) node[above]{$\overline{d}$} ;
       \draw (9, 2) node[above]{$\overline{c}$} ;
       \draw (9,-2) node[right]{$\overline{a}$} ;

           \draw   (10,-2.5) node[below] {F-shaped};
       
        \draw[very thick](12,-2) -- (12,2) ;
     \draw[very thick](12,-2) -- (14,-2) ;
     \draw[very thick](12,2) -- (14,2) ; 
         \node[draw,circle,inner sep=2pt,fill=red] at (12,-2) {};
       \node[draw,circle,inner sep=2pt,fill=red] at (12,2) {};
       \node[draw,circle,inner sep=2pt,fill=red] at (14,-2) {};
       \node[draw,circle,inner sep=2pt,fill=red] at (14,2) {};

       \draw (14,-2) node[right]{$\overline{d}$} ;
       \draw (14,2) node[above]{$\overline{a}$} ;
       \draw (12, 2) node[above]{$\overline{b}$} ;
       \draw (12,-2) node[left]{$\overline{c}$} ;

           \draw   (13,-2.5) node[below] {C-shaped};
    
           \end{tikzpicture}  
        \caption{The $Y$-shaped, $F$-shaped and $C$-shaped trees in the proof of     
                \refthm{valblocks}.} 
        \label{fig:YFC}
     \end{figure}

             \medskip
             \noindent
      $\bullet$
       \textbf{Assume that $\cvxh{\overline{a}, \overline{b}, \overline{c}, \overline{d}}$   is 
           Y-shaped}. 
        By  \refprop{equivsep}, the point $d$ separates simultaneously $a$ from $b$,  
        $b$ from $c$ and $a$ from $c$. 
        Using this fact and the hypotheses of the theorem, we get that: 
         \[
        \delta (a, b) + \delta (c, d) =   \delta (a, c) + \delta (b, d)=  \delta (a, d) + \delta (b, c)  = \delta (a, d) + \delta (b,d) + \delta (c,d).  
        \]
          Thus the 4-point condition \eqref{eqn:4point}
   is verified with equalities in this case. Reasoning as in the previous cases,  one gets 
   that the only equalities among the triangle inequalities are of the form 
    $\delta(x,y) = \delta(x,d) + \delta(d,y)$ for $x, y \in \{a, b, c\}$, $x\ne y$. 

\medskip 
\noindent
$\bullet$
 \textbf{Assume that $\cvxh{\overline{a}, \overline{b}, \overline{c}, \overline{d}}$   is 
           F-shaped.}  
  By  \refprop{equivsep}, we have that neither $c$ nor $d$ separates $a$ from $b$ 
  but $c$ separates $b$ from $d$ and also $c$ separates $a$ from $d$. We obtain the following 
  triangle (in)equalities:
  \[ 
  \begin{array}{c}
   \delta (a,b) < \delta(a, c) + \delta(b,c), \quad  
   \quad  \delta (a,b) < \delta(a, d) + \delta(b,d); 
  \\
   \delta (b,d) = \delta(b, c) + \delta(c,d), \quad
     \quad \delta (a,d) = \delta(a, c) + \delta(c,d).
  \end{array}
  \] 
It is immediate to see from these relations that the four point condition 
\eqref{eqn:4point} holds with a strict inequality, where 
the right hand side of \eqref{eqn:4point}  is equal to $\delta (a,c) + \delta(b,c) + \delta (c,d) $.

  \medskip
  \noindent
  $\bullet$          
 \textbf{Assume that $\cvxh{\overline{a}, \overline{b}, \overline{c}, \overline{d}}$   is 
           C-shaped.}  
    By \refprop{equivsep}, we have that
  $b$  separates $a$ from $d$, that $b$ separates $a$ from $c$ and that 
  $c$ separates $b$ from $d$.  The triangle inequalities become equalities in this case:         
  \[
   \delta (a,d) = \delta(a, b) + \delta(b,d),
\quad 
   \delta (a,c) = \delta(a, b) + \delta(b,c)
 \quad \mbox{and} \quad 
   \delta (b,d) = \delta(b, c) + \delta(c,d).
   \]
     It follows that  the 4-point condition \eqref{eqn:4point} holds with a strict inequality, where 
    the right hand side  \eqref{eqn:4point}  is equal to $\delta(a, b) + 2 \delta(b,c) + \delta(c,d)$.           
\end{proof}

\begin{ex}   \label{ex:lightgreen}
Consider \reffig{biggraph}. In the {\em left picture}, we have a graph $\Gamma$. 
Here $\cF=\{a_1, \ldots, a_{13}\}$ is depicted in light-green.  In this example, 
all the vertices in $\cF$ are of valency $1$ (which is not a hypothesis of \refthm{valblocks}). 
The cut vertices are in red. Shaded areas correspond to bricks.
Dark-green shaded edges represent some of the bridges (the one whose endpoints are both cut points).
In the {\em right picture}, we have represented the brick-vertex tree $\bvt{\Gamma}$. 
The light-green shaded subgraph is the set $\cvxh{\cF} \subset \bvt{\Gamma}$.
Notice that there are $4$ brick-vertices of $\bvt{\Gamma}$ which have valency at least $4$ 
($3$ of them have valency $4$ and one of them has valency $5$).
But at those vertices the convex hull $\cvxh{\bvt{\cF}}$ has only valency $3$.
This convex hull has also two points of valency $4$, but both of them are cut-vertices. 
    Therefore, we have here a situation in which is satisfied the hypothesis of 
    \refthm{valblocks} that each brick of $\Gamma$ has $\cvxh{\overline{\cF}}$-valency at most $3$. 
\end{ex}

\begin{figure}[ht]
\centering
\def\svgwidth{0.4 \columnwidth}
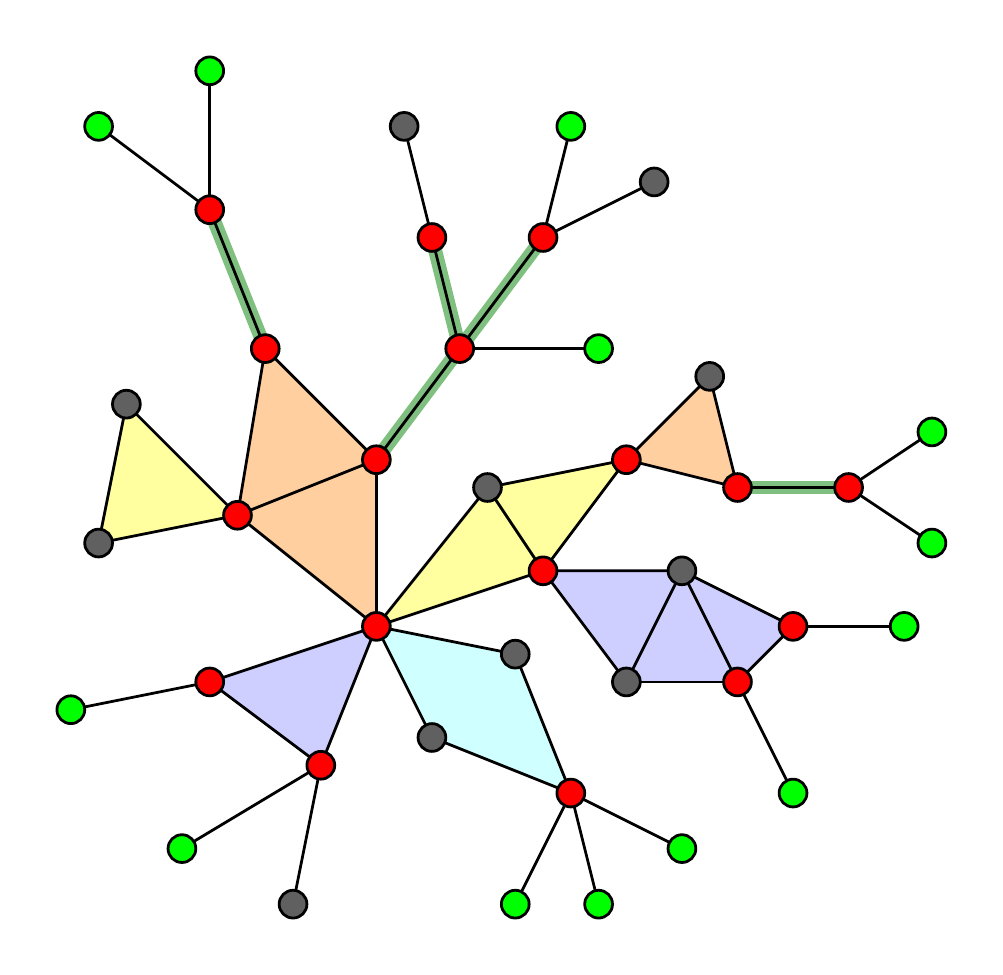
\def\svgwidth{0.4 \columnwidth}
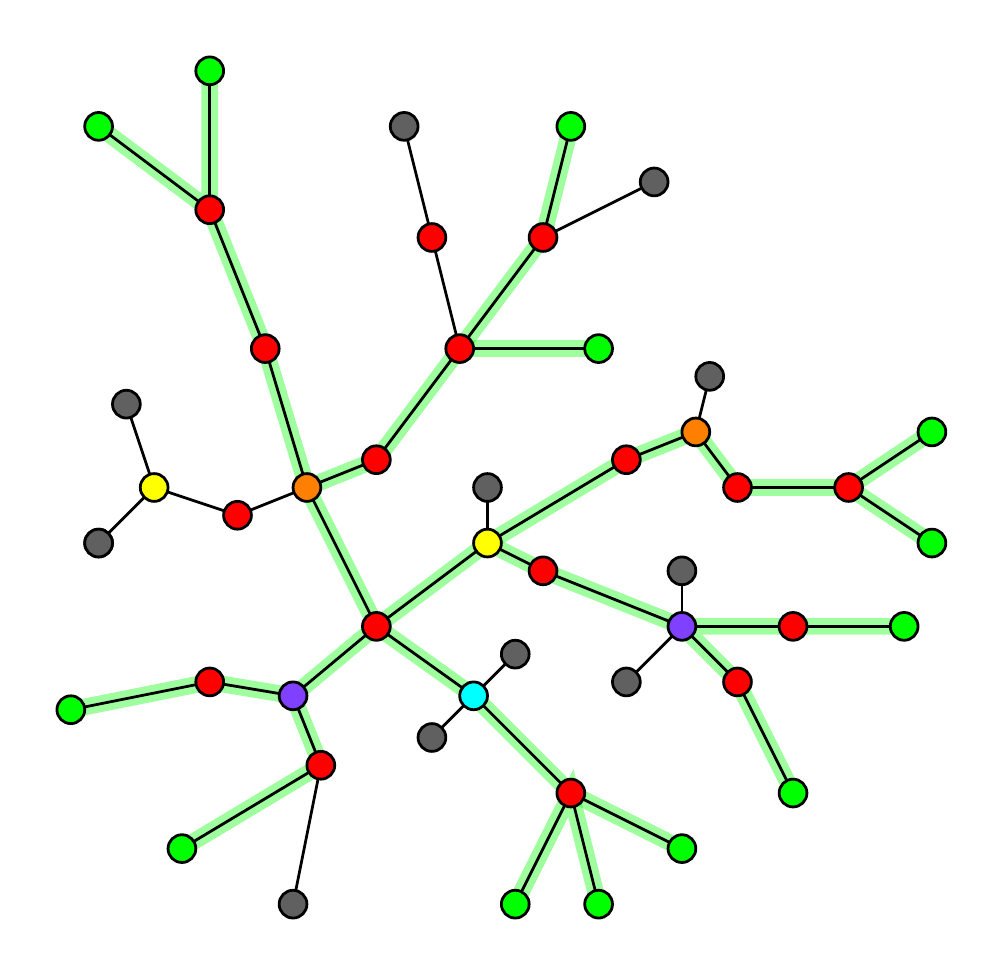
\caption{\refex{lightgreen}, in which the hypothesis of 
     \refthm{valblocks} about valencies of bricks is satisfied.}\label{fig:biggraph}
\end{figure}

\medskip
\subsection{Applications to finite sets of branches on normal surface singularities}
$\:$  

\medskip
        The main result of this section (\refthm{ultramthm}) is the announced 
        generalization to arbitrary 
        normal surface singularities of the fact that $\udb{L}$ is an ultrametric on 
        arborescent singularities (see \refthm{initultram}). This generalization,  
        stating that in general $\udb{L}$ is an ultrametric in restriction to special sets 
        of branches, describable topologically on any embedded resolution of their sum, 
        is an immediate corollary of \refthm{valblocks} of the previous section.

\medskip
 Applying \refthm{valblocks}   to the angular distance $\rho$, we get:

\begin{cor}   \label{cor:logcosfour}
   Let $X$ be a normal surface singularity 
   and $\pi$ be a good resolution of $X$. 
   Consider a subset $\cF $ of the set 
   of vertices of the dual graph $\dgr{\pi}$ and its convex hull $\cvxh{\cF}$  
   in the brick-vertex tree $\bvt{\dgr{\pi}}$ of $\dgr{\pi}$.  
   If each brick of $\dgr{\pi}$ has $\cvxh{\cF}$-valency at 
   most $3$, then the restriction of $\logfunc$ to $\cF$ is tree-like  
   and the associated tree is isomorphic as an $\cF$-tree to $\cvxh{\cF}$. 
\end{cor} 
 
In order to state the next results, it is convenient to introduce the following vocabulary:

\begin{defi} \label{defi:injres}
    If $\cF \subset \branches{X}$ is a finite set of branches on $X$, 
    then an {\bf injective resolution} 
   of $\cF$ is an embedded resolution of their sum 
   such that different branches in $\cF$ have different representing divisors 
   (in the sense of \refdef{reprdiv}).
\end{defi}

If $\pi$ is an injective resolution of $\cF$, then we have a canonical injection of $\cF$ in $\primes {\pi}$. 
We will identify sometimes $\cF$ and its image, saying for instance that $\cF$ 
is a subset of the set of vertices of $\dgr{\pi}$.

We deduce immediately from \refcors{logcosfour}{equivultr} 
the following theorem:

\begin{thm}  \label{thm:ultramthm}
    Let $X$ be a normal surface singularity. Consider a finite set $\cF$ of branches 
    on it and denote by $L$ one of them. Let $\pi$ be an injective resolution of 
    the sum of branches in $\cF$. Identify 
    $\cF$ with the set of prime divisors representing its elements.   
     If each brick of $\dgr{\pi}$ has $\cvxh{\cF}$-valency at 
    most $3$, then the function $\udb{L}: (\cF \setminus \{L\})^2 \to [0, \infty)$ 
    is an ultrametric and the associated rooted 
    $\cF$-tree is isomorphic to $\cvxh{\cF}$.     
\end{thm}

Note that \refthm{initultram} is indeed a special case of 
\refthm{ultramthm}. This is a consequence of the fact that for arborescent 
singularities, $\dgr{\pi}$ has no bricks. 

\begin{rmk}
    The rooted tree associated to $u_L$ in \refthm{ultramthm} is \textit{end-rooted} 
    in the sense of \cite[Definition 3.5]{gbgppp:2016}, that is, its root is of valency $1$. It corresponds 
    to a supplementary element associated to the set of closed balls of the ultrametric, 
    which may be thought as a ball of infinite radius. The approach of the paper 
    \cite{gbgppp:2016} was to work exclusively with rooted trees associated to ultrametrics. 
    By contrast, in the present paper our trees are associated to metrics satisfying the $4$-point 
    condition (see \refdef{treelike}), therefore they are not canonically rooted. 
    One may translate one approach into the other one using \refprop{reformultra}.
\end{rmk}

An important aspect of \refthm{ultramthm} is that 
\textit{it depends only on the topology of the 
total transform of the branches on an embedded resolution of their sum}, and neither  
on special properties of the values of the intersection numbers of the prime 
exceptional divisors, nor on their genera. 

\begin{ex}\label{ex:tetrahedron}
The condition on the valency of brick-points in \refthm{ultramthm} (and of analogous theorems like \refthm{udbv_val}) is not necessary in general.
For example, 
 consider  a singularity $X$ whose 
minimal good resolution has a tetrahedral dual graph. Denote by $E_1, E_2, E_3, E_4$ 
the exceptional primes, and assume that they all have the same self-intersection $-k$, where $k \geq 4$.
By symmetry, $\check{E}_i \cdot \check{E}_j$ is constant for any $1 \leq i \neq j \leq 4$.
The brick-vertex tree has here a brick-vertex of valency $4$, but the $4$-point condition is satisfied.
See \refexs{tetrahedron_val}{tetrahedron_val_modif} for a deeper analysis of this example. 
\end{ex}

\medskip
\subsection{An ultrametric characterization of arborescent singularities}
$\:$  

\medskip
  The aim of this section is to prove a converse to \refthm{initultram}. 
  Namely, we prove that if $\udb{L}$ is an ultrametric for 
  every branch $L$ on $X$, then $X$ is arborescent (see \refthm{arborcase}). 
\bigskip

In the next proposition we show that if the normal surface singularity is not 
arborescent, then one may find four branches on it such that for any one of them, 
called $L$, the associated function $\udb{L}$ is not an ultrametric on the set of 
remaining three branches (even if the proposition is not stated like this, 
the fact that its conclusion may be formulated in this way  
is a consequence of \refprop{reformultra}):

\begin{prop}\label{prop:noud}
   Let $X_{\pi}$ be a good model of $X$. 
   Assume that $a, b, m, p$ are four pairwise distinct vertices of the dual graph 
   $\dgr{\pi}$, such that: 
      \begin{itemize}
          \item  both $m$ and $p$ are adjacent to $a$; 
          \item  $a$ does not separate $b$ from either $m$ or $p$. 
      \end{itemize}
   Denote by $x_m$ the intersection point of $E_a$ and $E_m$  and 
   by $x_p$ the intersection point of $E_a$ and $E_p$.
   Let $A$ and $B$ be branches on $X$ whose representing divisors on $X_{\pi}$ 
   are $E_a$ and $E_b$ respectively.   
   Then there exist branches $C_m$ and $C_p$ whose strict transforms on $X_{\pi}$ 
   pass through $x_m$ and 
   $x_p$ respectively, such that: 
     \begin{equation} \label{eqn:noeq}
           (A \cdot B)  (C_m \cdot C_p)  < (C_m \cdot A) (C_p \cdot B) < 
                   (C_m \cdot B) (C_p \cdot A). 
      \end{equation}
\end{prop}

\begin{figure}[ht]
\centering
\def\svgwidth{0.5 \columnwidth}
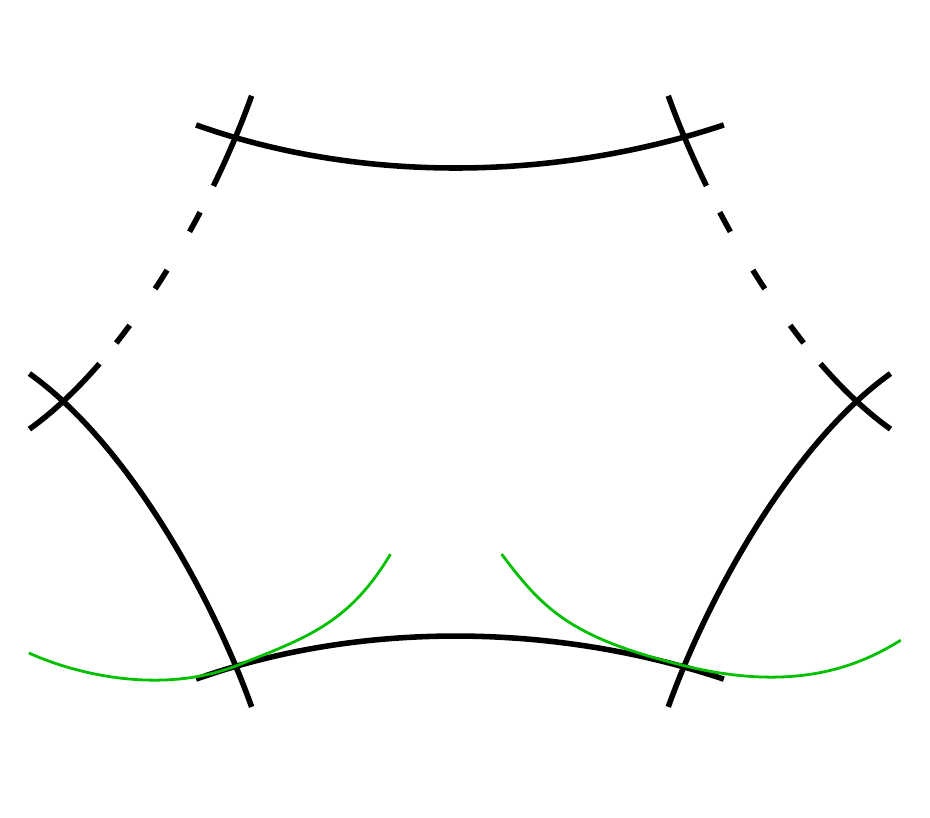
\caption{Geometric situation of \refprop{noud}.}\label{fig:cycle}
\end{figure}

\begin{proof}
Consider a branch $C_m$ whose strict transform $\strt{(C_m)}{\pi}$ passes through the point $x_m$, 
is smooth and tangent to the prime exceptional divisor $E_a$.
Denote by $s \in \N^*$ the intersection number $\strt{(C_m)}{\pi} \cdot E_a$.
As $\strt{(C_m)}{\pi} \cdot E_m =1$ and the intersection numbers of $\strt{(C_m)}{\pi}$ with the 
other irreducible components of the exceptional divisor of $\pi$ are all $0$, we deduce that:
$$ \exct{(C_m)}{\pi}= - \check{E}_m - s \check{E}_a.$$

Consider an analogous branch $C_p$ whose strict transform passes through $x_p$, and such 
that one has $\strt{(C_p)}{\pi} \cdot E_a=t \in \N^*$.
One gets:
$$  \exct{(C_{p})}{\pi}= - \check{E}_p - t \check{E}_a.$$
See \reffig{cycle} for the relative positions of prime exceptional divisors and 
strict transforms of branches.

    As the strict transforms $\strt{(C_m)}{\pi}$ and $\strt{(C_p)}{\pi}$ are disjoint, 
    \refcor{intexcep} implies that:
    $$ C_m \cdot C_p = -  \exct{(C_m)}{\pi} \cdot \exct{(C_p)}{\pi}. $$
  We use the analogous equalities for the other intersection numbers appearing in 
  \eqref{eqn:noeq} (in each case, the strict transforms of the corresponding branches 
  by the modification $\pi$ are again disjoint). As $\exct{A}{\pi} = - \check{E}_a$ 
  and $ \exct{B}{\pi} = - \check{E}_b$, 
  the system of inequalities 
  \eqref{eqn:noeq}  becomes: 
      \begin{equation}    \label{eqn:doublineq}
          \begin{array}{c}
                \bra{a, b} \cdot (\bra{m, p} + t \bra{m, a} + s \bra{a, p} + ts \bra{a,a})  <   \\
                (\bra{m, a} + s \bra{a,a}) (\bra{p,b} + t \bra{a,b}) < \\
                 (\bra{m, b} + s \bra{a,b}) (\bra{p,a} + t \bra{a,a}) . 
          \end{array}
      \end{equation}
                  We want to show that we may find pairs $(s,t) \in \N^* \times \N^*$ such that 
      \eqref{eqn:doublineq} holds. Let us consider in turn both inequalities. 
      
      \noindent
        $\bullet$ 
         The left-hand 
         inequality in (\ref{eqn:doublineq}) becomes:
       \begin{equation}   \label{eqn:leftsimpl}
            ( \bra{a,a} \bra{b,p} - \bra{a,b} \bra{a,p}) s + (\bra{a,m} \bra{b,p} - \bra{a,b} \bra{m,p}) >0.
       \end{equation}
      Note that the left-hand side of (\ref{eqn:leftsimpl}) is a polynomial of degree $1$ in the 
      variable $s$. 
      By     \refprop{MW} and the hypothesis that $a$ does not separate $b$ from $p$ 
      in the dual graph of $\pi$, the  coefficient 
      $\bra{a,a} \bra{b,p} - \bra{a,b} \bra{a,p}$ of $s$  is positive. 
      Therefore, the inequality \eqref{eqn:leftsimpl} becomes true for $s$ big enough. 
      
      \noindent
      $\bullet$
      Similarly, the right-hand inequality of \eqref{eqn:doublineq} becomes: 
          \begin{equation}   \label{eqn:rightsimpl}
            (\bra{a,a} \bra{b,m} - \bra{a,b} \bra{a,m}) t - ( \bra{a,a} \bra{b,p} - \bra{a,b} \bra{a,p}) s + 
               \bra{a,p} \bra{b,m} - \bra{a,m} \bra{b,p} >0.
       \end{equation}
           Assume that $s$ was chosen such that \eqref{eqn:leftsimpl} holds. 
           The left-hand side of \eqref{eqn:rightsimpl} is then a polynomial of degree $1$ in the 
           variable $t$. Its dominating coefficient $\bra{a,a} \bra{b,m} - \bra{a,b} \bra{a,m}$ is $>0$, 
           because $a$ does not separate $b$ from $m$. Therefore, the inequality    
            \eqref{eqn:rightsimpl} 
           becomes true for $t$ big enough.  
\end{proof}

We get the announced characterization of arborescent singularities:

\begin{thm} \label{thm:arborcase}
     Let $X$ be a normal surface singularity. Then the following properties are equivalent:
         \begin{enumerate}
              \item \label{forany}  
                 For every branch $L \in \branches{X}$, the function $\udb{L}$  is 
                  an ultrametric on the set $\branches{X} \setminus \{L\}$. 
             \item \label{exists} 
                  There exists a branch $L \in \branches{X}$, such that the function $\udb{L}$  is 
                 an ultrametric on the set $\branches{X} \setminus \{L\}$. 
              \item  \label{bracketultra}
                  The bracket $\bra{ \cdot, \cdot}$ satisfies the following inequality: 
                     $$\bra{a, b} \cdot \bra{l, c} \geq \min  \{ 
                   \bra{a, c} \cdot \bra{l, b}, \:   \bra{b, c} \cdot \bra{l, a} \}, \: \: \mbox{for all} 
                        \: \:  (a,b,c,l) \in 
                        (\genprimes{X})^4 .$$
              \item   \label{arbsingul}
                  The singularity $X$ is arborescent. 
         \end{enumerate}
\end{thm}

\begin{proof} $\:$ 
        The equivalences (\ref{forany}) $\Longleftrightarrow$ (\ref{exists})
             $\Longleftrightarrow$ (\ref{bracketultra}) are direct consequences of 
             \refcor{equivultr}.

         \noindent
         The implication (\ref{arbsingul}) $\Longrightarrow$ (\ref{forany}) 
             is a direct consequence of \refthm{initultram}.

          \noindent  
              In order to prove the implication (\ref{exists})  $\Longrightarrow$ (\ref{arbsingul}) 
              we proceed by contradiction, and suppose that $X$ is not arborescent. 
              We will show that for every choice of branch $L$, there exist branches $A,C_m,C_p$ 
              such that the quadruple $L,A,C_m,C_p$ does not satisfy the $4$-point condition.
              Fix a good model $X_\pi$ of $X$, which is an embedded resolution of the branch $L$. 
              Denote by $E_l$ the exceptional prime representing $L$ in $X_\pi$, and look at $l$ 
              as a vertex in the dual graph $\dgr{\pi}$ of $\pi$.
              By \refprop{noud}, it suffices to find three vertices $a,m,p$ in $\dgr{\pi}$ 
              such that $m$ and $p$ are adjacent to $a$, and $a$ does not separate $l$ 
              from either $m$ or $p$.

              As $X$ is not arborescent, the dual graph $\dgr{\pi}$ contains a cycle $\Theta$. 
              Replacing perhaps $X_\pi$ by another model obtained from it by blowing up points  
              of the divisor represented by $\Theta$, we may assume that  $\Theta$ 
              has at least $4$ vertices. If $l$ is a vertex of $\Theta$ we take $a,m,p$ 
             three other successive  vertices of $\Theta$ and apply \refprop{noud}.
              Otherwise, $l$ does not belong to $\Theta$. As $\dgr{\pi}$ 
              is connected, there exists a path $\Pi$ inside it connecting $l$ to a vertex $d$ 
              of $\Theta$ such 
              that $d$ is the only vertex common to $\Theta$ and to this path. 
              As $\Theta$ has at least $4$ vertices, one may find three successive vertices 
              $m, a, p$ of it, which are different from $d$. 
              Then the vertices $a$, $m$ and $p$ satisfy the condition we were looking for.
\end{proof}


\section{Ultrametric distances on valuation spaces}\label{sec:valuations}

In this second part of the paper,  we generalize the results of \refsec{branches} 
to the setting of valuation spaces.
We keep denoting by $(X, x_0)$ a normal surface singularity and by $\mc{O}_X$ its local 
ring. We denote by $R$  the completion $\hat{\mc{O}}_{X}$ of its local ring relative to 
its maximal ideal  and by  $\mf{m}$  the unique maximal ideal of $R$.

\medskip

\subsection{Semivaluation spaces of normal surface singularities}\label{ssec:valuation_spaces}
$\:$

\medskip

In this section we recall the definitions of \textit{semivaluations} and 
\textit{valuations} of $X$, as well as that of \textit{normalized} such objects. 
Then we recall the classification of semivaluations into 
\textit{divisorial}, \textit{quasi-monomial} (in particular \textit{irrational}), \textit{curve} and \textit{infinitely singular}.

\medskip

Let $[0, + \infty]$ be the union of the set of non-negative real numbers and of the 
single-element set $\{ + \infty \}$, endowed with the usual total order. 
In this paper we will consider the following  notion of semivaluation:

\begin{defi}\label{def:semival}
 A \emph{semivaluation} on $X$  (or on $R$) is a function $\nu\colon R\to [0, + \infty]$ satisfying the following axioms:  
        \begin{enumerate} 
            \item $\nu(0) = + \infty$ and $\nu(1) =0$; 
            \item $\nu(\phi\psi) = \nu(\phi) + \nu(\psi)$  for all $\phi, \psi\in R$; 
            \item $\nu(\phi + \psi)\geq \min\{\nu(\phi), \nu(\psi)\}$ for all $\phi, \psi\in R$;
           \item $ 0 < \nu ( \mf{m} ) < + \infty$; \label{fval}
         \end{enumerate}
where $\nu(\mf{m}) := \min\{\nu(\phi) \st \phi \in \mf{m}\}$.
The semivaluation $\nu$ is \emph{normalized} if in addition $\nu(\mf{m}) = 1$.
The semivaluation $\nu$ is a \emph{valuation} if 
$\nu^{-1}(+ \infty) = \{ 0\}$.  
The set of semivaluations on $X$ will be denoted by $\boxed{\FVal{X}}$, while the set of 
normalized semivaluations will be denoted by $\boxed{\Val{X}}$.
\end{defi}

\begin{rmk}
There are more general notions of semivaluations,  
which do not require the condition (\ref{fval}) on \refdef{semival}, or which take values on the non-negative part of the additive semigroup $\R^2$, with respect to the lexicographical ordering.
In the literature, the semivaluations of \refdef{semival} 
are usually called \textit{centered} 
(which makes reference to the condition  
$\nu(\mf{m})>0$), \textit{finite} (meaning that $\nu(\mf{m}) < +\infty$) and of \textit{rank $1$} 
(since they take values on the non-negative part of $(\R,+)$). 
\end{rmk}

If $\nu$ is a semivaluation on $X$, so is $\lambda \nu$ for any 
$\lambda \in \nR_+^* := (0,+\infty)$. In particular, any semivaluation is proportional 
to a normalized one.

\begin{rmk}
The normalization with respect to the maximal ideal is not the only possible one.
It is sometimes useful to normalize with respect to other ideals of $R$.  
A typical choice (see \cite{favre-jonsson:valtree,favre-jonsson:eigenval} for the smooth setting) 
is to normalize with respect to the value taken on a given irreducible element $x$ of $R$, that is, by considering only semivaluations which satisfy $\nu(x)=1$.
In this case a special care must be taken for the curve semivaluation $\nu_C$ with $C=\{x = 0\}$, 
since $\inte_C(x)=+\infty$ (see below for the definitions of $\nu_C$ and $\inte_C$).
\end{rmk}

 If $\mf{a}$ is an ideal of $R$,  we denote
 $\nu(\mf{a}) := \min\{\nu(\phi) \st \phi \in \mf{a}\}$ for any semivaluation $\nu$.
One may  define equivalently a semivaluation $\nu$ as a function on the set of ideals of $R$
 satisfying similar properties as those in  \refdef{semival} (see \cite{gignac-ruggiero:locdynnoninvnormsurfsing}).
 
 \medskip 
Note that for any semivaluation $\nu$, the set $\nu^{-1}(+ \infty)$ is a prime 
ideal of $R$. Therefore, it defines either the point $x_0$ or a branch on $X$.

\begin{defi}  \label{def:specsemival}
The \emph{support} of a semivaluation of $R$ is the vanishing locus of the prime 
ideal $\nu^{-1}(+ \infty)$.
\end{defi}

The spaces $\FVal{X}$ and $\Val{X}$ come equipped with 
natural topologies:

\begin{defi} \label{def:weaktop}
  The \emph{weak topologies} on the sets $\FVal{X}$ and $\Val{X}$ are the weakest ones 
   such that the maps $\nu \mapsto \nu(\phi)$ 
   are continuous for any $\phi \in R$.
\end{defi}

In the foundational work \cite{zariski:resolutionofsingularities}, 
Zariski gave a classification of semivaluations according to 
some algebraic invariants (\textit{rank}, \textit{rational rank}, \textit{transcendence degree}). 
Those different kinds of semivaluations can also be characterized by their geometric properties.
We recall here a few facts about this classification in our setting.

\medskip 
\noindent
$\bullet$ \emph{Divisorial valuations}.
They are the valuations associated to the prime exceptional divisors, 
as seen in \refssec{mumfordint}.
Let $X_\pi$ be a good model of $X$, and $E \in \primes{\pi}$ be any irreducible 
(and reduced) component of the exceptional divisor $\pi^{-1}(x_0)$. Then the map 
$\boxed{\divi_E}$, 
which associates to a function $\phi \in R$ the order of vanishing of $\phi \circ \pi$ along $E$, 
defines a valuation of $X$. 
We say that a valuation is \emph{divisorial} if it is of the form 
$\lambda\:   \divi_E$, with $\lambda \in \nR_+^*$. When $\lambda = 1$, the 
divisorial valuation is called \emph{prime}, a denomination already 
used in \refsec{branches}.
For any exceptional prime $E \in \primes{\pi}$, we denote by $\boxed{\nu_E} := b_E^{-1}\divi_E$ 
the normalized valuation proportional to $\divi_E$, where $\boxed{b_E}:=\divi_E(\mf{m}) \in \nN^*$ 
is the \emph{generic multiplicity} of $\nu_E$.
Finally, for any good model $X_{\pi}$ of $X$, we denote by $\boxed{\skeldiv{\pi}}$ 
the set of normalized divisorial valuations associated to the primes of $\pi$.

\medskip
\noindent 
$\bullet$ \emph{Quasi-monomial and irrational valuations}. 
Quasi-monomial valuations of $X$ are constructed as follows.
Let $X_\pi$ be a good model of $X$, 
and let $P \in \exc{\pi}$ be any point in the exceptional divisor $\exc{\pi}$ of $\pi$.
Pick local coordinates $(x,y)$ at $P$ adapted to $\exc{\pi}$ 
(i.e., so that $\exc{\pi} \subseteq \{xy=0\}$ locally at $P$).
For any $(r,s) \in (\nR_+^*)^2$, we may consider the monomial valuation $\mu_{r,s}$ 
on the local ring of $X_{\pi}$ at $P$, defined on the set of monomials in $x$ and $y$ 
by setting $\mu_{r,s}(x)=r$ and $\mu_{r,s}(y)=s$, and extended to any element 
$\phi$ of this ring by taking the minimum of $\mu_{r,s}$ on the set of monomials appearing in $\phi$.
The valuation $\nu_{r,s}$ defined by 
$\boxed{\nu_{r,s}}:=\pi_* \mu_{r,s}: \phi \mapsto \mu_{r,s}(\phi \circ \pi)$ 
is an element of $\FVal{X}$, called a \emph{quasi-monomial valuation}.
If $r$ and $s$ are rationally dependent, it turns out that $\nu_{r,s}$ is a divisorial valuation 
(associated to an exceptional prime obtained after a toric modification of $X_{\pi}$  
in the coordinates $(x,y)$).
If $r$ and $s$ are rationally independent, we call the valuation $\nu_{r,s}$ an \emph{irrational valuation}.
Notice that we can also define $\nu_{r,s}$ when either $r$ or $s$ vanishes.
For example, suppose that $\exc{\pi}=\{x=0\}=E$ locally at $P$.
Then the valuation $\nu_{1,0}$ coincides with $\divi_E$, while $\nu_{0,1}$ 
is not a centered valuation: it would correspond up to a multiplicative constant 
to the order of vanishing along the branch determined by the projection of $\{y=0\}$ to $X$.

\medskip
\noindent 
$\bullet$ \emph{Curve semivaluations}. 
They are the semivaluations associated to branches in $\branches{X}$.
Given such a branch  $L$,  a \emph{curve semivaluation associated to $L$} 
is any positive real multiple of 
$\boxed{\inte_L}$, which in turn is defined by $\inte_L(\phi) := L \cdot (\phi) $, where 
$\phi \in R$ and $(\phi)$ denotes the divisor of $\phi$.
As for divisorial valuations, we denote by $\boxed{\nu_L} :=m(L)^{-1} \inte_L$ the normalized 
semivaluation proportional to $\inte_L$, where $m(L) \in \nN^*$ is the multiplicity of $L$.
Notice that curve semivaluations are never valuations, since $\inte_L(\phi)= + \infty$ 
for any $\phi \in R$ vanishing on $L$. In fact, the support of $\inte_L$ according to 
\refdef{specsemival} is exactly $L$.

\medskip 
\noindent 
$\bullet$ \emph{Infinitely singular valuations}. These are the remaining elements of $\FVal{X}$. They are characterized by having rank and rational rank equal to $1$, and transcendence degree equal to $0$. They are also characterized as valuations whose value group is not finitely generated over $\nZ$.
They can be thought as curve semivaluations associated to branches of infinite multiplicity (see \cite[Chapter 4]{favre-jonsson:valtree}).

\medskip 
\begin{defi} \label{def:skeleton}
Given a good model $X_{\pi}$, we denote by $\boxed{\skel{\pi}}$ the set of centered
normalized quasi-monomial valuations 
described above, for all the points 
$p \in \pi^{-1}(x_0)$, and call it the \emph{skeleton of $X_{\pi}$}. 
\end{defi}

Notice that $\skel{\pi}$ 
admits a structure of finite connected graph, with set of vertices $\skeldiv{\pi}$, and edges between 
two points $\nu_E$ and $\nu_F$ for each intersection point between $E$ and $F$ in $\pi^{-1}(x_0)$. 
This graph is homeomorphic to the dual graph $\Gamma_{\pi}$ introduced in Definition 
\ref{def:dualgraph}.

\begin{rmk}
In \refsec{branches}, we considered only divisorial valuations. 
Given such a valuation $u$, we denoted by $E_u$ the exceptional prime associated to it.
Since here we consider other types of valuations, not associated to exceptional primes, 
we prefer to denote by $\nu \in \Val{X}$ any kind of valuation, and write $\nu=\nu_E$ 
if $\nu$ is the divisorial valuation associated to the exceptional prime $E$.
\end{rmk}

\medskip

\subsection{Valuation spaces as projective limits of dual graphs}\label{ssec:projlim}
$\:$

\medskip

The aim of this section is to explicit some basic relations   
between dual graphs, skeleta and the valuation space.
\medskip

Let $\pi \colon X_\pi \to X$ be a good resolution of the normal surface singularity $X$ and 
$\nu \in \FVal{X}$ a semivaluation of $X$. By the valuative criterion of properness, 
$\nu$ has a unique \emph{center} in $X_\pi$, which lies in the exceptional 
divisor of $\pi$.
The center is characterized as the unique scheme-theoretic point $\xi \in X_\pi$ 
so that $\nu$ takes non-negative values on the local ring $\mc{O}_{X_\pi, \xi}$ of elements of the fraction field of $R$ whose pullbacks to $X_{\pi}$ 
are regular at $\xi$, and strictly positive values exactly on its maximal ideal $\mf{m}_\xi$.

Then one can define as follows a retraction $r_\pi$ from $\Val{X}$ to the skeleton $\skel{\pi}$ 
of the good model $X_\pi$ (see \refdef{skeleton}).
Let $\nu \in \Val{X}$ be a normalized semivaluation, and $\xi \in \pi^{-1}(x_0)$ its center.
If $\xi$ is the generic point of an exceptional prime $E$, or if it is a closed point belonging to a unique exceptional prime $E$ of $\primes{\pi}$, then we set $r_\pi(\nu) :=\nu_E$, 
the divisorial valuation associated to $E$.
If $\xi$ is a closed point $P$ belonging to the intersection of two exceptional primes $E$ and $F$, 
then $\nu=\pi_* \mu$, where $\mu$ is a semivaluation centered at $P$. Pick local coordinates $(x,y)$ 
at $P$ so that $E=\{x=0\}$ and $F=\{y=0\}$. Then we set $r_\pi(\nu)$ to be the quasi-monomial 
valuation $\pi_* \mu_{r,s}$ at $P$ with weights $r=\mu(x)$ and $s=\mu(y)$.
By a result of Thuillier's paper \cite{thuillier:homotopy}, the map $r_\pi:\Val{X} \to \skel{\pi}$ is 
in fact a strong deformation retract.

If $\pi':X_{\pi'}\to (X,x_0)$ is another good resolution dominating $\pi$, then we have  
$r_\pi = r_\pi \circ r_\pi'$. Hence we get a natural continuous map from the valuation space $\Val{X}$ to the projective limit $\displaystyle \lim_{\substack{\longleftarrow\\\pi}} \skel{\pi}$ of the skeleta, which turns out to be a homeomorphism (see \cite[Theorem 7.5]{vaquie:valuations} and \cite[p. 399]{favre:holoselfmapssingratsurf}).
This approach can be taken in order to construct the valuation space $\mc{V}_X$ directly as the projective limit of the dual graphs of the good resolutions of $(X,x_0)$.

In particular, we can characterize arborescent singularities as the normal surface singularities $X$ for which the valuation space $\Val{X}$ is contractible.
Indeed, if $X$ is arborescent, then the dual graph of each good resolution $\pi$ is a tree, 
hence $\skel{\pi}$ is contractible, and so is $\Val{X}$ that deformation retracts onto it. 
Similarly, if $X$ is not arborescent, then we can find a non-trivial loop on the dual graph 
of a good resolution $\pi$, and its image inside $\skel{\pi} \subset \Val{X}$ gives a non-trivial 
loop inside $\Val{X}$.

\subsection{B-divisors on normal surface singularities}
$\:$

\medskip

In the first part of the paper, it was crucial to associate a dual to any prime 
divisor on a model of $X$. By looking at the divisor as a prime divisorial valuation, 
and by collecting its associated dual divisors on all the models, one gets a particular 
b-divisor, in the sense of \refdef{bdiv}. In this section we explain how to  
extend the previous construction to all semivaluations on $X$ (see \refdef{bdivsemival}). 
As an application, we show how to extend to the space of normalized 
semivaluations the notions of bracket (see \refdef{bracketval}) and of angular distance 
(see \refdef{angdistval}).

\medskip

Let $\nu \in \FVal{X}$. 
One may define unambiguously the value $\nu(D)$ taken by $\nu$ on any divisor $D \in \ExcR {\pi}$ 
(see for instance \cite[Section 7.5.2]{jonsson:berkovich} for the case where $R$ is regular, 
which extends without changes to our case, or \cite[Section 2.2]{gignac-ruggiero:locdynnoninvnormsurfsing}).
The idea is to define first $\nu(D)$ when $D$ is prime, by evaluating $\nu$ on a local defining function of $D$, and to extend it 
then by linearity. Such local defining functions may be taken as pull-backs of elements of 
the localization of $R$ at the defining prime ideal $\nu^{-1}(+ \infty)$ of the support of $\nu$, 
to which $\nu$ extends canonically.

Any semivaluation on $X$ induces a dual divisor on $X_\pi$, according 
to the next proposition (see \cite[Page 400]{favre:holoselfmapssingratsurf} or 
\cite[Proposition 2.5]{gignac-ruggiero:locdynnoninvnormsurfsing}):

\begin{prop}  \label{prop:dualdivprop}
    For any semivaluation $\nu \in \FVal{X}$, there exists a unique divisor $Z_\pi(\nu) \in \ExcR{\pi}$ 
    such that $ \nu(D)=Z_\pi(\nu) \cdot D$ for each $D \in \ExcR {\pi}$.
\end{prop}

We will use the following name for this divisor:

\begin{defi}
     The divisor $\boxed{Z_\pi(\nu)}$ characterized in \refprop{dualdivprop} 
     is called the \emph{dual divisor} of $\nu$ in the model $X_{\pi}$.
\end{defi}

The name alludes to the fact that for a divisorial valuation $\divi_E$, 
we have $Z_\pi(\divi_E)=\check{E}$. Here $\check{E}$ denotes the dual divisor 
of $E$, as defined by relations \eqref{eqn:dualdiv}.

\begin{defi}  \label{def:bdivsemival}
    The collection $\boxed{Z(\nu)}=(Z_\pi(\nu))_\pi$, where $\pi$ varies among all good 
   resolutions of $X$, is called the \emph{b-divisor associated to $\nu$}.
\end{defi}

This name is motivated by the fact that $Z(\nu)$ is a b-divisor in the 
following sense, due to Shokurov \cite{shokurov:prelimitingflips} (the letter ``b'' is the initial 
of ``birational''):

\begin{defi}  \label{def:bdiv}
    A collection $(Z_\pi)_\pi$, where $\pi$ varies among all good 
   resolutions of $X$ and $Z_{\pi} \in \ExcR {\pi}$, is called a \emph{b-divisor} of $X$ if 
   for any pair of models $(\pi, \pi')$ such that $\pi'$ dominates $\pi$, one has 
     $  \psi_* Z_{\pi'} = Z_{\pi}$,  
    if $\pi' = \pi \circ \psi$.
\end{defi}

In  \refsec{branches}, we noticed that the intersection of two dual divisors does 
not depend on the model used to compute it (see \refprop{invdual}).
This allows to define the intersection number $Z(\nu) \cdot Z(\mu)$ of two 
b-divisors associated to divisorial valuations $\nu, \mu \in \FVal{X}$.
In the general case of an arbitray pair of semivaluations $(\nu, \mu)$ of $X$, 
the intersection number $Z_\pi(\nu) \cdot Z_\pi(\mu)$ 
\textit{may depend} on the model $\pi$. In fact, we always have
$Z_{\pi'}(\nu) \cdot Z_{\pi'}(\mu) \leq Z_\pi(\nu) \cdot Z_\pi(\mu)$, 
for any model $\pi'$ dominating $\pi$. 
More precisely, the intersection remains constant as far as $\nu$ and $\mu$ 
have different centers in $X_\pi$ 
(see \cite[Proposition 2.13]{gignac-ruggiero:locdynnoninvnormsurfsing}), 
while it decreases if the centers coincide 
(see \cite[Proposition 2.17]{gignac-ruggiero:locdynnoninvnormsurfsing}).
This allows to define:
\[     \boxed{Z(\nu) \cdot Z(\mu)} := \inf_\pi \big(Z_\pi(\nu) \cdot Z_\pi(\mu)\big) \in [-\infty,0).  \]

We refer to \cite{boucksom-favre-jonsson:degreegrowthmeromorphicsurfmaps, favre:holoselfmapssingratsurf, gignac-ruggiero:locdynnoninvnormsurfsing} 
for further details on b-divisors associated to semivaluations.

Recall that in \refdef{bra} was introduced the \textit{bracket} of two prime divisorial valuations.
The next definition extends the bracket to arbitrary pairs of semivaluations:

\begin{defi}    \label{def:bracketval}
Let $\nu,\mu \in \FVal{X}$ be two semivaluations of $X$. 
Their \emph{bracket} is defined by: 
 \[    \boxed{ \bra{\nu,\mu} }:= - Z(\nu) \cdot Z(\mu) \in (0,+\infty].   \]
When $\nu = \mu$, the self-bracket $\boxed{\alpha(\nu)} :=\bra{\nu,\nu}$ 
is called the \emph{skewness} of $\nu$.
\end{defi}

\begin{rmk}
The skewness $\alpha(\nu)$ has been analysed for germs of smooth surfaces in \cite{favre-jonsson:valtree}, where it was defined as the supremum of the ratio between the values of $\nu$ and of the multiplicity function.
With this interpretation, the skewness is sometimes called the \textit{Izumi constant} of $\nu$, a denomination which refers to the works \cite{izumi:LCIordergerms,izumi:integritylocalanalalgebras} of Izumi.
Its study has been the focus of several works, see e.g. \cite{rees:izumi,ein-lazarsfeld-smith:uniformapproxabhyankar,moghaddam:izumi,rond-spivakovsky:izumiabhyankar,boucksom-favre-jonsson:izumi}.
The b-divisor interpretation given by Favre and Jonsson is more recent, and it has been used to study several properties of valuation spaces for smooth and singular surfaces (see e.g. \cite{jonsson:berkovich,gignac-ruggiero:locdynnoninvnormsurfsing}). 
\end{rmk}

Let us consider now the restriction of the bracket to the space $\Val{X}$ of 
normalized semivaluations. 
The skewness is always finite for quasi-monomial valuations, while it is always infinite 
for curve semivaluations; it can be any value in $(0,+\infty]$ for infinitely 
singular valuations (see \cite[Theorem 3.26]{favre-jonsson:valtree} for the smooth case, and \cite[Proposition 2.17]{gignac-ruggiero:locdynnoninvnormsurfsing} for the singular case). 
We denote by $\boxed{\mc{V}_X^\alpha}$ the set of normalized valuations 
with finite skewness.

More generally, one can show (see \cite[Proposition 2.13]{gignac-ruggiero:locdynnoninvnormsurfsing}) that $\bra{\nu,\mu}$ is determined on a model $X_\pi$, i.e., 
$\bra{\nu,\mu}=-Z_\pi(\nu) \cdot Z_\pi(\mu)$ as far as $\nu$ and $\mu$ have different 
centers on $X_\pi$. As for two distinct normalized semivaluations, there is always a model 
on which their centers are disjoint, we deduce that: 

\begin{prop}\label{prop:finbra}
The bracket of two distinct normalized semivaluations is always finite.
\end{prop}

Carrying on the analogies with the divisorial case of \refsec{branches}, 
we define the notion of angular distance of semivaluations, 
as introduced in \cite{gignac-ruggiero:locdynnoninvnormsurfsing}.

\begin{defi}    \label{def:angdistval}
The {\bf angular distance} of two normalized semivaluations $\mu, \nu \in \Val{X}$ is: 
\begin{equation} 
\boxed{\logfunc(\nu,\mu)}:= - \log \dfrac{\bra{\nu,\mu}^2}{\alpha(\nu) \cdot \alpha(\mu)} \in [0, \infty]
\end{equation}
if $\nu \neq \mu$, and $0$ if $\nu = \mu$.
\end{defi}

\begin{rmk}
The function $\logfunc$ defines an \emph{extended distance} on $\Val{X}$ 
(see \cite[Proposition 2.40]{gignac-ruggiero:locdynnoninvnormsurfsing}), in the sense that 
it vanishes exactly on the diagonal, it is symmetric, and it satisfies the triangular inequality 
(like a standard distance), but it may take the value $+\infty$ in some cases.
In fact, $\logfunc(\nu,\mu) = +\infty$ exactly when $\nu \neq \mu$ and at least 
one of the semivaluations $\nu$ and $\mu$ has infinite skewness. 
This locus can be precisely determined, by reducing first to the smooth case 
using \cite[Lemma 2.43]{gignac-ruggiero:locdynnoninvnormsurfsing}, 
and by describing then the skewness of a semivaluation in terms of 
its \textit{Puiseux parameterization}, as in \cite[Chapter 4]{favre-jonsson:valtree} (when 
one works over $\C$) or using Jonsson's approach in \cite[Section 7]{jonsson:berkovich} 
(when one works over an 
arbitrary field, possibly of positive characteristic). 
In particular, $\logfunc$ defines a distance on $\mc{V}_X^\alpha$, hence on 
the set of normalized quasi-monomial valuations.
The topology induced by $\logfunc$ on $\Val{X}$ is usually called the \emph{strong topology}, 
in order to distinguish it from the weak topology introduced in \refdef{weaktop}.
\end{rmk}

\subsection{Ultrametric distances on semivaluation spaces of arborescent singularities}
$\:$

\medskip

In \refssec{reformultrametric} we started the study of the function $\udb{L}$, that culminated with the characterization of arborescent singularities given in \refthm{arborcase}.
This section is devoted to the proof of an analog for semivaluation spaces 
(see \refthm{arborcase_val}).
We will study functions $\udv{\lambda}$ depending on an arbitrary semivaluation} 
$\lambda \in \Val{X}$, defined on $\Val{X} \times \Val{X}$.
In the particular case in which $\lambda$ is the curve semivaluation $\inte_L$ associated to a branch $L$ on $X$, 
we get $\udb{\inte_L}= \udb{L}$ (see \refrmk{ud_val_properties}).

\begin{defi}\label{def:ud_val}
Let $X$ be a normal surface singularity, and let  
$\lambda \in \FVal{X}$ be any semivaluation.
Let $\nu_1, \nu_2 \in \Val{X}$ be any normalized semivaluations on $X$.
We set: 
\begin{equation}\label{eqn:def_udv}
\boxed{\udv{\lambda}(\nu_1, \nu_2)}    :=
\begin{cases}
    \displaystyle\frac{\bra{\lambda,\nu_1} \cdot \bra{\lambda,\nu_2}}{\bra{\nu_1,\nu_2}} & 
     \text{if } \nu_1 \neq \nu_2,\\
    0 & \text{if } \nu_1 = \nu_2.
\end{cases}
\end{equation}
\end{defi}

\begin{rmk}\label{rmk:ud_val_properties}
Since $\bra{\nu_1,\nu_2} < +\infty$ when $\nu_1 \neq \nu_2$ (see \refprop{finbra}), 
the function $\udv{\lambda}$ is well defined with values in $[0,+\infty]$, 
and it vanishes if and only if $\nu_1=\nu_2$.
The value $+\infty$ is sometimes achieved. In fact, while the denominator 
is always strictly positive, if $\lambda$ is normalized 
we have $\bra{\lambda,\nu} = +\infty$ 
if and only if $\lambda=\nu$ and $\alpha(\lambda)=+\infty$.
In particular, $\udv{\lambda}$ takes only finite values if $\alpha(\lambda) < +\infty$, 
while it always takes finite values on $(\Val{X} \setminus \{\lambda\})^2$.

Notice that if $\nu_1$ and $\nu_2$ tend to the same semivaluation $\nu$ 
in the strong topology, then 
$\frac{\bra{\lambda,\nu_1}\cdot \bra{\lambda,\nu_2}}{\bra{\nu_1,\nu_2}}$ tends to 
$\frac{\bra{\lambda,\nu}^2}{\alpha(\nu)}$.
This value is finite as long as $\nu \neq \lambda$, and it is $0$ if and only if 
$\alpha(\nu)= +\infty$. This always happens when $\nu$ is a curve semivaluation, 
and never happens for quasi-monomial valuations.

Notice also that $\udv{\lambda}$ can be extended to $(\FVal{X})^2$,  
setting $\udv{\lambda} (\nu_1, \nu_2 ) := \frac{\bra{\lambda,\nu_1} \cdot \bra{\lambda,\nu_2}}{\bra{\nu_1,\nu_2}}$ if $\nu_1$ and $\nu_2$ are non-proportional, and equal to zero otherwise. 
In fact, by homogeneity of the bracket, we have 
$\udv{\lambda}(b_1 \nu_1,b_2 \nu_2)=\udv{\lambda}(\nu_1,\nu_2)$ 
for any $b_1, b_2 \in (0, + \infty)$ and also  $\udv{b  \lambda} = b^2 
\udv{\lambda}$, for any $b \in (0, + \infty)$.

Finally, \refdef{ud_val} clearly generalizes \eqref{eqn:potultram}.
In fact, if $L,A,B$ are branches on $X$, then $\udb{L}(A,B)=\udv{\inte_L}(\inte_A,\inte_B)$, 
where $\inte_L, \inte_A, \inte_B$ are the curve semivaluations associated to $L,A,B$ respectively.
\end{rmk}

The aim of this section is to prove the following generalization of \refthm{arborcase}:

\begin{thm} \label{thm:arborcase_val}
      Let $X$ be a normal surface singularity.  Then the following properties are equivalent:
\begin{enumerate}
   \item \label{forany_val}  
       For every semivaluation $\lambda \in \FVal{X}$, the function $\udv{\lambda}$ 
      is an extended ultrametric distance on $\Val{X}$. 
   \item \label{exists_val} 
      There exists a semivaluation $\lambda \in \FVal{X}$, such that the function $\udv{\lambda}$ 
     is an extended ultrametric distance on $\Val{X}$. 
   \item \label{arbsingul_val}
     The singularity $X$ is arborescent. 
\end{enumerate}
\end{thm}

Before starting the proof, let us give some definitions and preliminary results, 
analogous to those described in \refsec{branches}.

\begin{defi}
   Let $X$ be a normal surface singularity, and $\mu, \nu_1, \nu_2 \in \Val{X}$ 
   be three normalized semivaluations.
    We say that \emph{$\mu$ separates $\nu_1$ and $\nu_2$} 
    (or \emph{the couple $(\nu_1, \nu_2)$}) if either $\mu\in \{\nu_1, \nu_2\}$, or 
    $\nu_1$ and $\nu_2$ belong to different connected components of $\Val{X} \setminus \{\mu\}$.
\end{defi}

Notice that in the previous definition we can consider $\Val{X}$ endowed indifferently 
with either the weak or the strong topology, since the connected components of 
$\Val{X} \setminus \{\mu\}$ are the same for the two topologies.

\begin{prop}[{\cite[Proposition 2.15]{gignac-ruggiero:locdynnoninvnormsurfsing}}]\label{prop:positivity_bdivisors}
    Let $X$ be a normal surface singularity 
     and $\mu, \nu_1, \nu_2 \in \Val{X}$ be three normalized semivaluations. Then we have:
  \begin{equation}\label{eqn:positivity_bdivisors}
     \bra{\mu,\nu_1}\cdot \bra{\mu,\nu_2} \leq \bra{\mu,\mu} \cdot \bra{\nu_1,\nu_2}.
  \end{equation}
    Moreover, the equality holds if and only if $\mu$ separates $\nu_1$ and $\nu_2$.
\end{prop}

Notice that, by homogeneity, \refprop{positivity_bdivisors} holds also for non-normalized valuations.

\begin{prop}\label{prop:quadineg}
      Let $X$ be a normal surface singularity, and $\nu_j \in \Val{X}$, for $j=1, \ldots, 4$, 
      be four  normalized semivaluations.
      Suppose that there exists $\mu \in \Val{X}$ that separates simultaneously 
      the couple $(\nu_1,\nu_2)$ and the couple $(\nu_3,\nu_4)$.
     Then: 
  \begin{equation}\label{eqn:ABCD}
      \bra{\nu_1,\nu_2} \cdot \bra{\nu_3,\nu_4} \leq \bra{\nu_1,\nu_3} \cdot \bra{\nu_2,\nu_4} .
  \end{equation}
     Moreover, the equality in \eqref{eqn:ABCD} holds if and only if $\mu$ also separates 
     simultaneously the couple $(\nu_1,\nu_3)$ and the couple $(\nu_2, \nu_4)$.
\end{prop}

\begin{proof}
{\bf Suppose first that $\alpha(\mu)=+\infty$.} In this case, $\mu$ is necessarily an end of $\Val{X}$, 
  i.e., $\Val{X} \setminus \{\mu\}$ is connected.
It follows that, up to permuting the roles of $\nu_1, \nu_2$ and of $\nu_3, \nu_4$, 
we have either $\nu_1=\nu_3 = \mu$ or $\nu_1=\nu_4=\mu$.

In the first case, if either $\nu_2$ or $\nu_4$ coincides with $\mu$, 
then both sides of \eqref{eqn:ABCD} are $+\infty$, 
and we have equality, in agreement with the statement.
If both $\nu_2$ and $\nu_4$ differ from $\mu$, the left hand side of \eqref{eqn:ABCD} is finite, 
while the right hand side is $+\infty$, again in agreement with the statement, since $\mu$ 
does not separate $\nu_2$ and $\nu_4$.

In the second case, the left and right hand sides of \eqref{eqn:ABCD} coincide, 
and in fact $\mu$ separates also the couple $(\nu_1, \nu_3)$ and $(\nu_2, \nu_4)$.

\medskip
\noindent
{\bf Suppose now that $\alpha(\nu) < +\infty$.} 
By \refprop{positivity_bdivisors}, we have:
\begin{align}
\bra{\mu,\nu_1} \cdot \bra{\mu,\nu_3} &\leq \bra{\mu,\mu} \cdot \bra{\nu_1,\nu_3}, \label{eqn:ACEE}\\
\bra{\mu,\nu_2} \cdot \bra{\mu,\nu_4} &\leq \bra{\mu,\mu} \cdot \bra{\nu_2,\nu_4}. \label{eqn:BDEE}
\end{align}
We want to prove the inequality: 
\begin{equation}\label{eqn:ABCDEE}
    \bra{\nu_1,\nu_2} \cdot \bra{\nu_3,\nu_4} \cdot \bra{\mu,\mu} \leq \bra{\mu,\nu_2} \cdot \bra{\mu,\nu_4}    
     \cdot \bra{\nu_1,\nu_3},
\end{equation}
which implies the statement \eqref{eqn:ABCD} by applying \eqref{eqn:BDEE}.
Now, again by \refprop{positivity_bdivisors}, we have:
\begin{align}
\bra{\mu,\nu_1} \cdot \bra{\mu,\nu_2} &= \bra{\mu,\mu} \cdot \bra{\nu_1,\nu_2}, \label{eqn:ABEE}\\
\bra{\mu,\nu_3} \cdot \bra{\mu,\nu_4} &= \bra{\mu,\mu} \cdot \bra{\nu_3,\nu_4}, \label{eqn:CDEE}
\end{align}
where the equalities are given by the fact that $\mu$ separates both couples 
$(\nu_1,\nu_2)$ and $(\nu_3,\nu_4)$.
From these equalities, together with \eqref{eqn:ACEE}, we deduce that:
\begin{align*}
\bra{\nu_1,\nu_2} \cdot \bra{\nu_3,\nu_4} \cdot \bra{\mu,\mu}^2
&= \bra{\mu,\nu_1} \cdot \bra{\mu,\nu_3} \cdot \bra{\mu,\nu_2} \cdot \bra{\mu,\nu_4}\\
&\leq \bra{\mu,\mu} \cdot \bra{\nu_1,\nu_3} \cdot \bra{\mu,\nu_2} \cdot \bra{\mu,\nu_4},
\end{align*}
which gives the desired inequality \eqref{eqn:ABCDEE}.

Finally, by \refprop{positivity_bdivisors},  the inequalities \eqref{eqn:ACEE} and \eqref{eqn:BDEE} 
are equalities if and only if $\mu$ separates both the couple $(\nu_1, \nu_3)$ 
and the couple $(\nu_2, \nu_4)$.
This concludes the proof.
\end{proof}

\begin{proof}[{\bf Proof of \refthm{arborcase_val}}]
   By homogeneity of the bracket, we can assume that the semivaluation 
   $\lambda$ is normalized (see \refrmk{ud_val_properties}). 
   Clearly, (\ref{forany_val}) implies (\ref{exists_val}).

 \medskip
\noindent
{\bf Let us prove that (\ref{arbsingul_val}) $\Longrightarrow$ (\ref{forany}).} 
Let $\lambda \in \Val{X}$ be any normalized semivaluation. 
Since by construction $\udv{\lambda}$ is symmetric and vanishes only on the diagonal, 
it is enough to show that the ultrametric triangular inequality holds.

Let $\nu_1, \nu_2, \nu_3 \in \Val{X}$, and assume that  
$c := \bra{\lambda,\nu_1} \cdot \bra{\lambda,\nu_2} \cdot \bra{\lambda,\nu_3} \in [0, + \infty]$ 
is finite. This is guaranteed for example if the three semivaluations are taken in 
$\Val{X} \setminus \{\lambda\}$.
Let us define $I_1, I_2, I_3$ by:
\begin{align*}
\udv{\lambda}(\nu_1, \nu_2)=\frac{\bra{\lambda,\nu_1} \cdot \bra{\lambda,\nu_2}}{\bra{\nu_1, \nu_2}} = \frac{c}{\bra{\nu_1,\nu_2} \cdot \bra{\lambda,\nu_3}}=:\frac{c}{I_3},\\ 
\udv{\lambda}(\nu_1, \nu_3)=\frac{\bra{\lambda,\nu_1} \cdot \bra{\lambda,\nu_3}}{\bra{\nu_1, \nu_3}} = \frac{c}{\bra{\nu_1,\nu_3} \cdot \bra{\lambda,\nu_2}}=:\frac{c}{I_2},\\ 
\udv{\lambda}(\nu_2, \nu_3)=\frac{\bra{\lambda,\nu_2} \cdot \bra{\lambda,\nu_3}}{\bra{\nu_2, \nu_3}} = \frac{c}{\bra{\nu_2,\nu_3} \cdot \bra{\lambda,\nu_1}}=:\frac{c}{I_1}.
\end{align*}
We want to show that if $X$ is arborescent, then 
among the quantities $I_1, I_2, I_3$, at least two coincide, 
and they are smaller or equal than the third one.

Since $X$ is arborescent, 
the convex hull $\cvxh{\nu_1, \nu_2, \nu_3, \lambda}$ 
of $\{\nu_1, \nu_2, \nu_3, \lambda\}$ has one of the shapes represented in \reffig{fivetrees}. 
In this setting, the convex hull of a finite subset $S \subset \Val{X}$ may be defined as the 
union of the images of all injective continuous paths $\gamma \colon [0,1] \to \Val{X}$ 
(the latter considered with its weak topology) joining any two (distinct) points of $S$ 
(see Remark \ref{ch} below for an explicit description of this convex hull).

Possibly reordering the four semivaluations, we may assume that they are in counter-clockwise order, starting from the top right corner. In the case of the $Y$-shape, 
assume that the branch point is $\lambda$ (in other cases the argument is the same).
We study case by case, according to the shape of $\cvxh{\nu_1, \nu_2, \nu_3, \lambda}$:

\begin{itemize}[leftmargin=30pt, itemindent=40pt]
\item[{\bf $\bullet$ $H$-shaped.}] Let $\mu$ be any point in the horizontal segment. 
    It separates all couples, excepted at least one between $\nu_1, \lambda$ and $\nu_2, \nu_3$. 
     By \refprop{quadineg} we deduce that $I_3=I_2 < I_1$.
\item[{\bf $\bullet$ $X$-shaped.}] The branch point $\mu$ separates all couples, and $I_1=I_2=I_3$.
\item[{\bf $\bullet$ $Y$-shaped.}] The branch point $\mu=\lambda$ separates all couples, 
     and again $I_1=I_2=I_3$. 
\item[{\bf $\bullet$ $F$-shaped.}] Let $\mu$ be the branch point. It separates all couples, excepted  
      $\nu_1,\nu_2$. We get $I_1=I_2 < I_3$.
\item[{\bf $\bullet$ $C$-shaped.}] Let $\mu$ be any point in the vertical segment. 
      It separates all couples, excepted $\nu_1,\nu_2$ and $\nu_3, \lambda$. We get $I_1=I_2 < I_3$.
\end{itemize}

The case when some of the semivaluations $\nu_1, \nu_2, \nu_3, \lambda$ coincide is easier, 
and is left to the reader. 
We conclude that $\udv{\lambda}$ defines an ultrametric distance on 
$\Val{X} \setminus \{\lambda\}$ (and an extended ultrametric on $\Val{X}$).

\medskip
\noindent
{\bf We conclude the proof of \refthm{arborcase_val} by showing that (\ref{exists_val}) 
$\Longrightarrow$ (\ref{arbsingul_val}).} 
We proceed by contradiction, and assume that $X$ is not arborescent,  i.e., 
there exists a good model $\pi$ such that its dual graph $\Gamma_\pi$ has a loop. 
Denote by $E_1, \ldots, E_r$ the vertices of such a loop, where $E_j \in \primes{\pi}$ 
are exceptional primes satisfying  $E_j \cdot E_{j+1}=1$ for all $j=1, \ldots, r$ (with cyclic indices).
It follows that $\Val{X}$ has itself a loop $S$, given by the quasi-monomial valuations which 
are either the divisorial valuations $\nu_{E_j}$, or the quasi-monomial ones at $p_j = E_j \cap E_{j+1}$, 
for all $j \in \{1, \dots, r\}$. 
We have fixed a semivaluation $\lambda$ for which $\udv{\lambda}$ is an ultrametric distance.
We will show that there exist $\nu_1, \nu_2, \nu_3 \in \Val{X}$ satisfying
\begin{equation}\label{eqn:noud_val}
\bra{\nu_3,\lambda} \cdot \bra{\nu_1,\nu_2} < \bra{\nu_2,\lambda} \cdot \bra{\nu_1, \nu_3} < \bra{\nu_1, \lambda} \cdot \bra{\nu_2,\nu_3},
\end{equation}
or $I_3 < I_2 < I_1$, if we use the notations introduced in the previous part of the proof. 
This would contradict the hypothesis that $\udv{\lambda}$ is an ultrametric distance.

But this is the valuative counterpart of \refprop{noud}, which can be proved in this more general 
setting by using \refprop{positivity_bdivisors} instead of \refprop{crucial}.
The role of $a,b,m,p$ will be played by $\nu_3,\lambda,\nu_1,\nu_2$ respectively.
In particular, given $b$, it suffices to pick $\nu_3$ as any point in $S$ so that $\lambda$ 
is in the connected component of $\Val{X} \setminus \{\nu_3\}$ containing 
$S \setminus \{\nu_3 \}$. We may assume that $\nu_3$ is divisorial, associated 
to an exceptional prime divisor $E_a$.
Fix a model $X_\pi$ such that $\lambda$ and $\nu_3$ have different centers on it.  
Denote by $E_m$ and $E_p$ the exceptional prime divisors adjacent to $E_a$, whose 
associated valuations belong to $S$. Up to taking a higher model, we may also 
assume that the center of $\lambda$ is disjoint from $E_m$ and $E_p$, 
and that $\nu_3$ does not separate $\lambda$ from either $\nu_{E_m}$ or $\nu_{E_p}$.
\refprop{noud} gives two valuations $\nu_1$ and $\nu_2$, corresponding respectively 
to monomial valuations at the points $x_m$ and $x_p$ of \reffig{cycle}, which 
satisfy \eqref{eqn:noud_val}.
\end{proof}

\begin{rmk} \label{ch}
   The convex hull mentioned in  the previous proof 
   can be described in terms of the skeleton of a model.
   Fix a good resolution $\pi$, and for any closed point $P \in \pi^{-1}(x_0)$, 
   denote by $\mc{V}_P$ the topological closure of the set of semivaluations in $X_\pi$ 
   centered at $P$. This set $\mc{V}_P$ can be naturally identified with the valuative tree $\Val{}$ 
   of \cite{favre-jonsson:valtree}. 
   If $S$ is contained in $\mc{V}_P$ for some $P$, the convex hull $\cvxh{S}$ is taken in $\mc{V}_P$, 
   with respect to its tree structure inherited by $\Val{}$.
   If this is not the case, then there exist finitely many points 
   $P_1, \ldots, P_r$  (with $r \geq 2$) such that 
   $S \subset \bigcup_j \mc{V}_{P_j}$. In this situation,  one has to consider first for each 
   $j \in \{1, \dots, r\}$ the convex hull inside $\mc{V}_{P_j}$ of the union of $S \cap \mc{V}_{P_j}$ with 
   $r_\pi(S \cap \mc{V}_{P_j})$, as defined above, 
   where $r_\pi:\Val{X}\to \skel{\pi}$ is the retraction defined in \refssec{projlim}. 
   Then the convex hull $\cvxh{S}$ 
   is obtained as the union of those convex hulls with the convex hull of 
   $r_\pi(S)$ inside $\skel{\pi}$ (which is a tree, since $X$ is arborescent by hypothesis). 
   In fact, in this case $\mc{V}_X$ itself has a structure of $\nR$-tree 
   (see \refprop{allblocksaretrivial}).
\end{rmk}

\medskip
\subsection{$\nR$-trees and graphs of $\nR$-trees}\label{ssec:graphsofRtrees}
$\:$

\medskip

In  \refssec{blockvertextree}, we associated to any finite connected graph $\Gamma$ a tree $\bvt{\Gamma}$, called its \textit{brick-vertex tree}.  Then we applied 
this construction to the dual graph of the embedded resolution of 
the sum of a finite set $\mathcal{F}$ of branches on a normal surface singularity $X$, 
and we were able to describe using it a situation in which 
$\udb{L}$ defines an ultrametric distance on $\mathcal{F} \setminus \{L\}$ (see \refthm{ultramthm}).

In Section \ref{ssec:brickrtree} 
we construct an analog of the brick-vertex tree for the space $\Val{X}$.
With this scope in mind, we first recall the tree structure carried by the space of normalized 
semivaluations of a smooth surface singularity. Then we introduce the more general concept of 
\textit{graph of $\nR$-trees} (see \refdef{goRt}) and we explain how to associate to such 
a graph a topological space, called its \textit{realization} (see \refdef{realizgraph}). 
We conclude the section by introducing several operations on graphs of $\nR$-trees, \textit{regularizations} (see \refdef{regraph}) and \textit{refinements} (see \refdef{refin}), 
which will be used in the next section 
in the construction of the \textit{brick-vertex tree} of a graph of $\R$-trees.

\medskip

When $X$ is smooth, the space of normalized semivaluations $\boxed{\Val{}} :=\Val{X}$ has been 
deeply studied by Favre and Jonsson in \cite{favre-jonsson:valtree} 
(see also Jonsson's course \cite{jonsson:berkovich}). 
It is referred to as the \textit{valuative tree}, since it carries the structure of a \textit{$\nR$-tree} 
in the sense of \cite[Definition 2.2]{jonsson:berkovich}. Let us first recall the definition 
of this notion:

\begin{defi}\label{def:Rtree}
An \emph{interval structure} on a set $I$ is a partial order $\leq$ on $I$ under which 
$I$ becomes isomorphic as a poset to the real interval $[0,1]$ or to the trivial 
real interval $\{0\}$ (endowed with the standard total order of the real numbers).
A sub-interval $J \subseteq I$ is a subset of $I$ that becomes a subinterval of $[0,1]$ 
under such an isomorphism.
If $I$ is a set with an interval structure, we denote by $I^-$ the same set with the opposite interval structure.

An \emph{$\nR$-tree} is a set $W$ together with a family 
$\{\boxed{[x,y] }\subseteq W\ |\ x,y \in W\}$ of subsets endowed with interval structures, 
and satisfying the following properties:
\begin{enumerate}[label=(T\arabic*),leftmargin=0pt, itemindent=40pt]
  \item\label{item:T1} $[x,x]=\{x\}$;
   \item\label{item:T2} if $x \neq y$, then $[x,y] = [y,x]^-$ as posets; 
        moreover, $x= \min [x,y]$ and $y=\min [y,x]$;
    \item\label{item:T3} if $z \in [x,y]$, then $[x,z]$ and $[z,y]$ are subintervals 
       of $[x,y]$ such that $[x,z] \cup [z,y]=[x,y]$ and $[x,z]\cap[z,y] = \{z\}$;
    \item\label{item:T4} for any $x,y,z \in W$, there exists a unique element 
        $w= \boxed{x \wedge_z y} \in [x,y]$ such that $[z,x] \cap [y,x]=[w,x]$ and $[z,y] \cap [x,y] = [w,y]$;
    \item\label{item:T5} if $x \in W$ and $(y_\alpha)_{\alpha \in A}$ is a net in $W$ such that 
       the segments $[x,y_\alpha]$ increase with $\alpha$ 
       (relative to the inclusion partial order of the subsets of $W$), then there exists $y \in W$ 
        such that $\bigcup_\alpha [x,y_\alpha) = [x,y)$.
\end{enumerate}
Here we used the notation $\boxed{[x,y)} := [x,y] \setminus \{y\}$. 
We define analogously $\boxed{(x,y]}$ and $\boxed{(x,y)}$.
\end{defi}

Recall that a \emph{net} is a sequence indexed by a directed set, not necessarily countable. 

An $\nR$-tree structure on the set $W$ induces a natural topology, called \emph{weak topology}.
It is constructed as follows.
Fix any $z \in W$, and pick any two points $x,y \in W \setminus \{z\}$. 
We say that $x \sim_z y$ if $z \not \in [x,y]$ (a condition equivalent to 
$(z,x] \:  \cap \:  (z,y] \neq \emptyset$, found sometimes in the literature). 
An equivalence class is called a \emph{tangent direction} $\boxed{\vect{v}}$ at $z$, 
and the set of all such classes is denoted by $\boxed{T_z W}$ (see \refex{tangdir}).
Tangent directions need to be thought as \textit{branches} at a point $z$ of $W$, 
and in some way as infinitesimal objects (hence the name \textit{tangent direction}). 
For this reason we distinguish an element $\vect{v} \in T_z W$ from the set $\boxed{U_z(\vect{v})}$ 
of points $x \in W \setminus \{z\}$ representing $\vect{v}$, which is seen as a subset of $W$. 
We declare $U_z(\vect{v})$ to be open for any $z$ varying in $W$ and 
$\vect{v}$ varying among all tangent directions at $z$.
The weak topology is generated by such open sets (i.e., it is the weakest topology for which all 
the sets $U_z(\vect{v})$ are open).
When considering the $\nR$-tree structure of $\Val{}$, the weak topology defined here coincides with the weak topology defined in \refssec{valuation_spaces}.


\medskip

The structure of the space of normalized semivaluations $\Val{X}$ associated to a normal 
surface singularity $X$ has been investigated from a viewpoint similar to that of the present 
paper by Favre \cite{favre:holoselfmapssingratsurf}, and by Gignac and the last-named author in 
\cite{gignac-ruggiero:locdynnoninvnormsurfsing}.
It has also been investigated from somewhat different perspectives 
by Fantini \cite{fantini:normalizedlinks, fantini:normalizedberkovich}, 
Thuillier \cite{thuillier:homotopy} and 
de Felipe  \cite{defelipe:topspacesvalgeomsing}.
Roughly speaking, $\Val{X}$ is obtained patching together copies of the valuative tree $\Val{}$ 
along any skeleton $\mc{S}$ associated to a good resolution $\pi$ (see \refprop{valXisagraph}). 
As the name suggests, the space $\Val{X}$ admits an $\nR$-tree structure if and only 
if the singularity $X$ is arborescent (see \refprops{blocksinskeleta}{allblocksaretrivial}).

To cover the general case, we introduce the concept of \textit{graph of $\nR$-trees}, 
which combines the concepts of $\nR$-trees and finite graphs.

Seen combinatorially, a finite graph is given by a set of vertices $V$ and a set of edges $E$, 
both seen abstractly and related by incidence maps. One may then consider a topological 
realization of it: the edges can be seen as real segments $I_e=[0,1]$, and the incidences 
may be realized by maps $i_e\colon\{0,1\} \to V$, 
which give the identifications between the ends of the segment $I_e$ and some vertices of $V$.
We may assume that every vertex in $V$ is in the image of one such map $i_e$.
The graph can be then realized topologically as the disjoint union of all segments $I_e$ 
(and of the set $V$)
quotiented by the identification of the ends to vertices according to the maps $i_e$.
In order to define graphs of $\nR$-trees, we replace in this construction 
the segments with $\nR$-trees:

\begin{defi}\label{def:goRt}
   A \emph{graph of $\nR$-trees of finite type} is defined by the following data: 
\begin{enumerate}[label=(G\arabic*),leftmargin=0pt, itemindent=40pt]
  \item\label{item:G1} Three sets $V,E,D$, with $V$ and $E$ finite. 
   \item A family $(W_e)_{e \in E}$ of $\nR$-trees with two distinct marked points 
       $x_e, y_e \in W_e$, together with a map $i_e \colon V_e:=\{x_e,y_e\} \to V$. 
  \item A family $(W_d)_{d \in D}$ of $\nR$-trees with a marked point $x_d \in W_d$, 
      together with a map $i_d \colon V_d:=\{x_d\} \to V$.
\end{enumerate}
   We denote such a structure by $\boxed{(V,W)}$, where $W:=(W_a)_{a \in A}$ 
    is a family of $\nR$-trees as described above, with $A: =E \sqcup D$.
     An element $W_a$ is called a \emph{tree element} of $(V,W)$. If $a \in E$, $W_a$ is called an    
     \emph{edge element}, while if $a \in D$, $W_a$ is called a \emph{decoration element} of $(V,W)$.
      The maps $i_a$ are called \emph{identification maps}.
 \end{defi}

The previous definition has both topological aspects (as we consider $\R$-trees as building blocks) 
and combinatorial ones (as one has incidence maps). 
As for finite graphs, this definition allows to get a topological space:

\begin{defi}\label{def:realizgraph}
Given a graph of $\nR$-trees $(V,W)$, its \emph{realization} $Z$ is the set defined as
\[
      \boxed{Z(V,W)} :=\left.\bigsqcup_{a \in A} W_a\right/ \sim,
\]
    where $W_a \ni x \sim x' \in W_{a'}$ if and only if $x \in V_a, x' \in V_{a'}$ and $i_a(x)=i_{a'}(x')$.
\end{defi}

\begin{rmk} \label{rmk:notcorrectop}
Notice that we defined the realization $Z$ of a graph of $\nR$-trees $(V,W)$ merely as a set, 
and not as a topological space, even though it is endowed naturally with the topology induced 
by the one on the tree elements through the quotient by the equivalence relation $\sim$.
This topology, to which we will refer as the \emph{quotient topology}, is not well adapted to  
our purposes (see \refrmk{graphsofRtrees-topologies}). We will introduce a second topology, 
called the \emph{weak topology} (see \refdef{goRt_weaktop}), 
and we will consider a realization of $Z$ as a topological space with respect to the weak topology.
\end{rmk}

Up to restricting $V$ if necessary, we will always assume that for any $v \in V$, 
there exists an $a \in A$ such that $v \in i_a(V_a)$.
In this case, we can identify $v$ with the class of elements of the form $i_a(x)$ that satisfy $i_a(x)=v$.

Denote by $\boxed{\on{pr}}$ the natural projection from $\bigsqcup_{a \in A} W_a$ to $Z$. 
Let $x,y \in Z$ be two points, and suppose that there exists $a \in A$ such that $x,y \in \pr(W_a)$.
If $W_a$ is an edge element (i.e., $a \in E$), and $x=y=\pr(v)$ with $v \in V$, we denote by $[x,y]$ 
the singleton $\{\pr(v)\}$, and by $[x,y]_a$ the projection of the segment 
$[x_a,y_a]_a \subseteq W_a$ given by the $\nR$-tree structure of $W_a$, 
where $x_a, y_a$ are the marked points of $W_a$.

If all other situations, there exists unique $\tilde{x}$ and $\tilde{y}$ in $W_a$ 
so that $\pr(\tilde{x})=x$ and $\pr(\tilde{y})=y$.
In this case we denote by $[x,y]_a$ the projection of the unique segment $[\tilde{x},\tilde{y}]_a$ 
in $W_a$.

To ease notation, if clear from the context, we will omit the projection map and denote $\pr(W_a)\subseteq Z$ simply by $W_a$.

\begin{rmk}
We say that the graph in \refdef{goRt} is of finite type because we impose 
both the set of vertices $V$ and the set $E$ parametrizing the edge elements  to be finite.
One can remove these conditions in \ref{item:G1} and get more general objects.
Since our interest in graphs of $\nR$-trees lies solely in the description of valuation spaces, 
we will only need to work with graphs of $\nR$-trees of finite type.
We will hence assume all graphs of $\nR$-trees to be of finite type, without further mention.

Nevertheless, most of the results in this section will apply for general graphs of $\nR$-trees. 
We will use the finiteness of $V$ and $E$ in the next sections, to deduce the finiteness 
of the number of bricks (see \refssec{brickrtree}).

Moreover, the definition of graphs of $\nR$-trees can be easily adapted to other situations, 
for example to $\nQ$-trees, or trees of spheres, etc.
\end{rmk}

From a graph of $\nR$-trees, we can easily extract a finite graph (in the sense of 
\refdef{graphcell}), which encodes its geometric complexity:

\begin{defi}
     Let $(V,W)$ be a graph of $\nR$-trees, with realization $Z(V,W)$. Its \emph{skeleton} 
     $\boxed{S(V,W)}$ is the subset of $Z(V,W)$ obtained as the union 
     of the projected segments $[x_e,y_e]_e$, while $e$ varies in $E$.
\end{defi}

\begin{ex}\label{ex:goRt0}
The top left part of \reffig{goRt1} depicts an example of graph of $\nR$-trees $(V,W)$, 
where $V$ consists of two points $\{v_r, v_g\}$ (depicted in red and green), and $W$ 
consists of four tree elements: one decoration element and three edge elements. 
Marked points are colored red or green according to the identification maps.
On the right part, we can see its realization, obtained by gluing together the tree elements 
along the marked points according to the identification maps.
Its skeleton $S(V,W)$, represented by thick lines, consists of the projection to $Z$ of the 
three segments between the marked points of the three edge elements. 
The lower left part of \reffig{goRt1} depicts the regularization of $(V,W)$, 
a notion introduced below in 
\refdef{regraph}.
\end{ex}

\begin{figure}[ht]
\centering
\begin{minipage}{0.5 \columnwidth}
\def\svgwidth{0.97 \columnwidth}
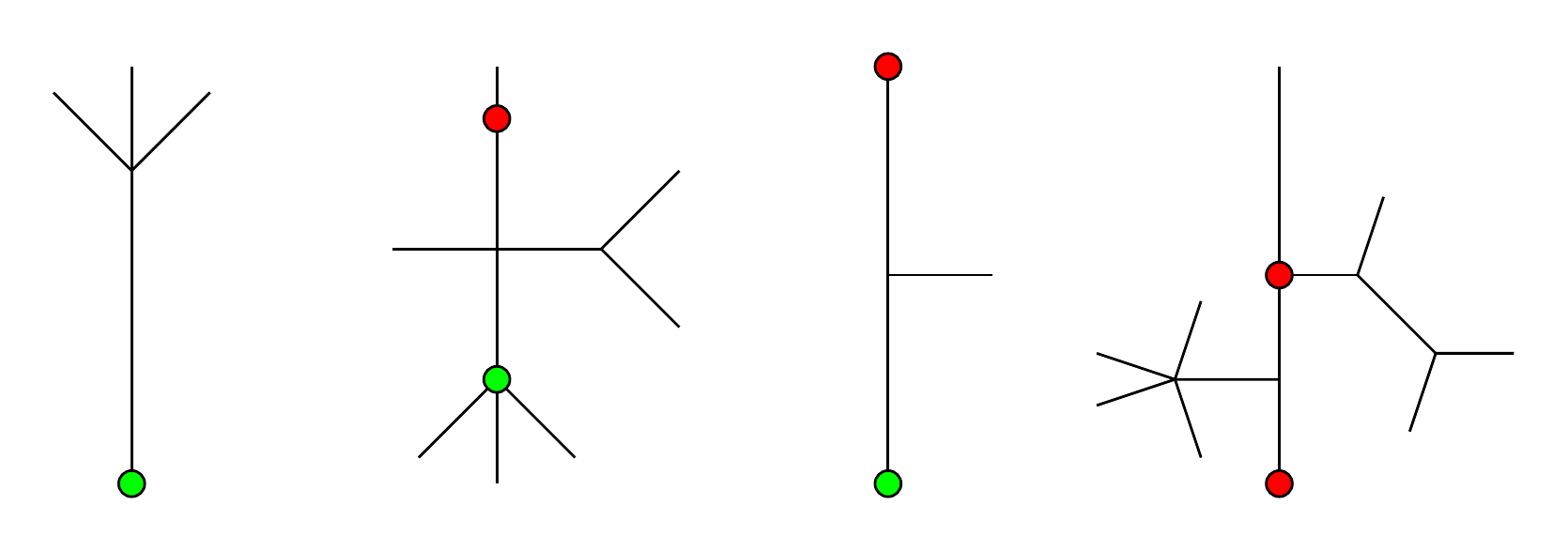

\vspace{8mm}

\def\svgwidth{1.0 \columnwidth}
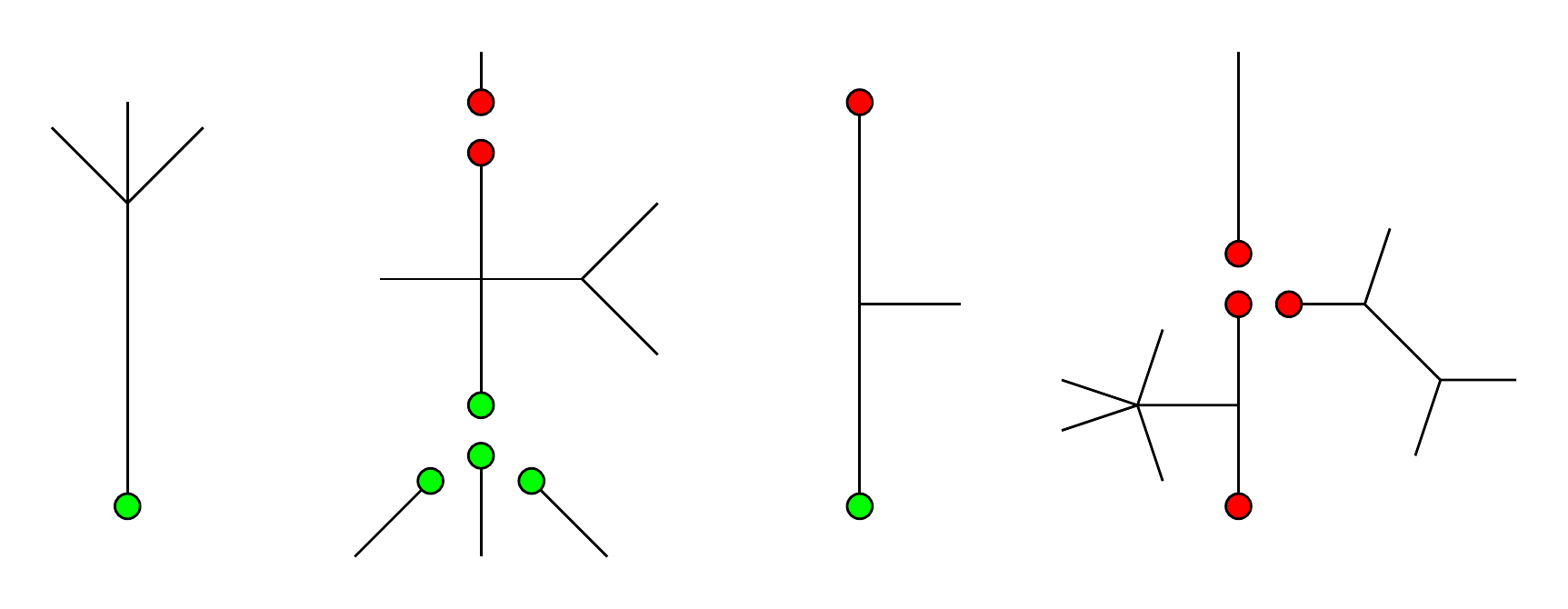
\end{minipage}
\begin{minipage}{0.45 \columnwidth}
\def\svgwidth{1.0 \columnwidth}
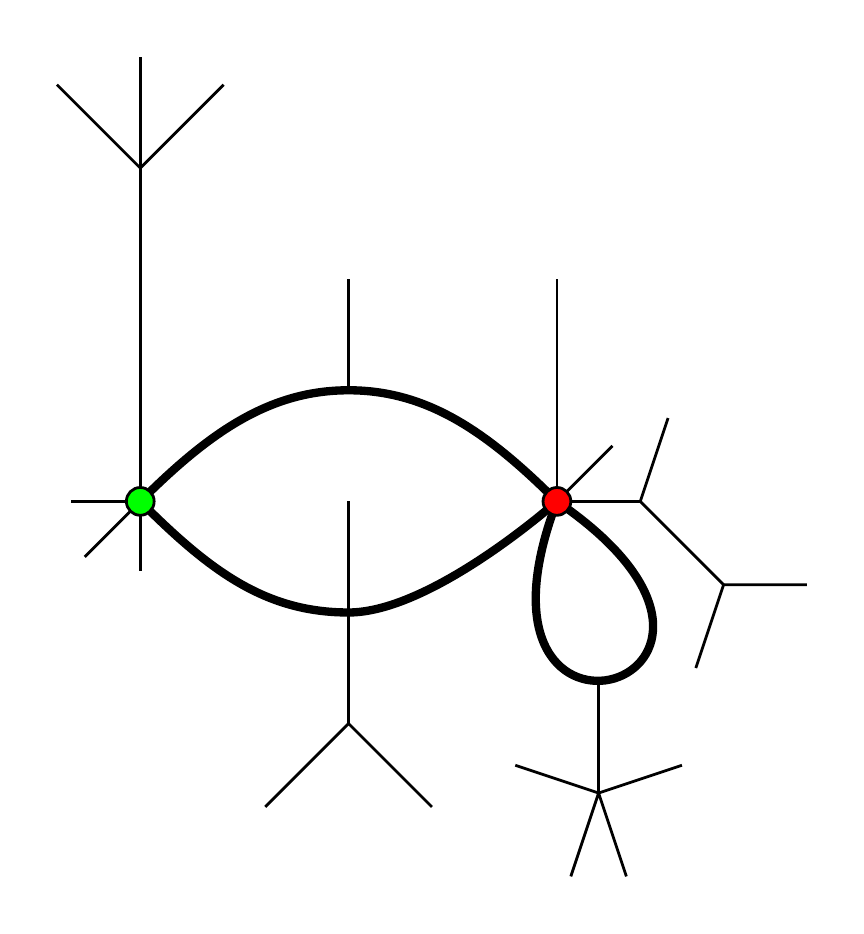
\end{minipage}
\caption{A graph of $\nR$-trees, its regularization, their realization and the 
   corresponding skeleton.}\label{fig:goRt1}
\end{figure}

As indicated in \refrmk{notcorrectop}, the quotient topology on the realization 
of a graph of $\nR$-trees is not well adapted. Another topology can be introduced, 
using the notion of \textit{arc}  between two points of the realization:

\begin{defi}
Let $(V,W)$ be a graph of $\nR$-trees, with realization $Z$. Let $x,y$ be two points in 
$Z$.
An \emph{arc} $\gamma$ between $x$ and $y$ is a subset of $Z$ obtained 
as a finite concatenation of segments $[s_j, s_{j+1}]_{a_j}$, $j=0, \ldots, n$, where 
\begin{itemize}
\item $s_0=x$, $s_{n+1}=y$, and $s_j \in V$ for all $j=1, \ldots n$;
\item $s_j, s_{j+1} \in W_{a_j}$ for all $j=0, \ldots, n$;
\item any two segments in the concatenation intersect in at most finitely many points.
\end{itemize}
\end{defi}

Here comes the definition of the topology on the realization:

\begin{defi}\label{def:goRt_weaktop}
Let $(V,W)$ be a graph of $\nR$-trees, with realization $Z$.
For any $z \in Z$, and any $x,y \in Z \setminus \{z\}$, we say that $x \sim_z y$ if there 
exists an arc between $x$ and $y$, which does not contain $z$.
The \emph{weak topology} on $Z$ is the weakest topology for which any 
subset $U$ of $Z$ representing an equivalence 
class for $\sim_z$, for any $z \in Z$, is a open set.
\end{defi}

Notice that, in contrast with the situation for $\nR$-trees, the equivalence classes 
for $\sim_z$ do not correspond directly with tangent vectors at $z$.
In fact, one can define tangent vectors at a point $z \in Z$ as the union of tangent 
vectors at $z \in W_a$ for all $a \in A$.
When $Z$ admits cycles, the spaces associated to two tangent 
vectors at a point $z$ of the cycle could belong to the same equivalence class 
with respect to $\sim_z$.
See \cite[Section 2.4]{gignac-ruggiero:locdynnoninvnormsurfsing} for a description 
of this phenomenon for normalized semivaluation spaces attached to normal surface singularities.

\begin{ex}  \label{ex:tangdir}
  Consider again the graph of $\nR$-trees $(V,W)$ described in \refex{goRt0}, 
  and its realization $Z$, depicted on the top left and right part of \reffig{goRt1} respectively.
  The tangent space at the green point $v_g$ consists of $6$ tangent vectors, 
  associated to the $1+4+1$ tangent vectors appearing on the first $3$ tree elements.
  By contrast, $Z \setminus \{v_g\}$ has $5$ connected components. The discrepancy 
  is due to the fact that $v_g$ belongs to a cycle of the realization $Z$ of $(V, W)$.
   Similarly, the red point $v_r$ has $7$ tangent directions, while $Z \setminus \{v_r\}$ 
  has $5$ connected components. 
\end{ex}

$\nR$-trees and more generally graphs of $\nR$-trees should not be thought only as 
topological spaces.
In fact for applications to semivaluation spaces, one usually needs to go back and forth 
from the weak topology to the strong topology induced by $\logfunc$ 
(see \cite{favre-jonsson:valtree, favre-jonsson:eigenval, jonsson:berkovich, gignac-ruggiero:attractionrates, favre-jonsson:dynamicalcompactifications, gignac-ruggiero:locdynnoninvnormsurfsing}).
Nevertheless, the weak topology will be very handy, for example in order to 
be able to talk about connected components of cofinite subsets of $Z(V,W)$ and to define bricks.

\begin{rmk}\label{rmk:graphsofRtrees-topologies}
Let us compare the two topologies introduced for the realization $Z$ of a graph of $\nR$-trees: the quotient topology and the weak topology.
On the one hand, it is easy to see that the topology induced on $W_a$ by the weak topology on 
$Z$ does coincide with the weak topology on $W_a$ given by 
its $\nR$-tree structure.
On the other hand, the weak topology on $Z$ does not coincide in general  
with the quotient topology.

Consider for example the graph $(V,W)$ where $V$ consists of just one element $V=\{p\}$, 
and the family $W=(W_d)_{d \in D}$ is an infinite family of decoration elements 
(not reduced to a point).
In this case, the realization $Z$ admits a structure of $\nR$-tree, and the topology 
induced by this $\nR$-tree structure coincides with the weak topology of its graph 
of $\nR$-tree structure.
In particular, an open connected neighborhood of $p$ would contain all decoration elements 
$W_d$, but for a finite number of $d \in D$.
In contrast, an open connected neighborhood of $p$ for the quotient topology is the union of open connected neighborhoods of $p$ in any 
decoration element $W_d$, and in particular it need not contain any $W_d$. 
\end{rmk}

\medskip

Since it is not the aim of this paper to develop a complete theory of graphs of $\nR$-trees, 
we will not give a definition of morphisms of graphs of $\nR$-trees, nor of isomorphic 
graphs of $\nR$-trees. 
Nevertheless, we will consider in this subsection a few operations on graphs of $\nR$-trees, 
which will change the graph structure without changing the underlying realization 
(seen as a topological space).
With this in mind, we will say that two graphs of $\nR$-trees are \emph{equivalent} 
if their realizations are homeomorphic with respect to the weak topologies.

The first operation is related to the choice of the marked points in the tree elements.
In fact, following the parallel with classical graphs, we consider the additional condition:
\begin{enumerate}[label=(G\arabic*),leftmargin=0pt, itemindent=40pt]\setcounter{enumi}{3}
\item \label{item:G4} the marked points $V_a$ of a tree element $W_a$ are \emph{ends} 
of $W_a$ (i.e., elements that do not disconnect $W_a$).
\end{enumerate}

\begin{defi} 
   The graphs of $\nR$-trees satisfying the additional condition \ref{item:G4} are called \emph{regular}.
 \end{defi}

Given any graph of $\nR$-trees $(V,W)$, one can consider the following construction.
For any $d \in D$, the tree $W_d$ has a marked point $x=x_d$.
For any tangent vector $\vect{v} \in T_{x}W_d$, set 
$W_{d, \vect{v}} := U_x(\vect{v}) \cup\{x\}$. 
The set $W_{d, \vect{v}}$ is an $\nR$-tree, with marked point $x$.
Set $i_{d,\vect{v}}(x) :=i_d(x)$.
We replace $W_d$ by the family $(W_{d,\vect{v}})_{\vect{v} \in T_{x}W_d}$.

Analogously, for any $e \in E$, the tree $W_e$ has two marked points $x=x_e$ and $y=y_e$.
Consider the set of connected components of $W_e \setminus V_e$. 
For any such component $U$, set $W_{e,U} :=\overline{U}$.
Notice that there is a unique component $U$ such that $W_{e,U}$ contains $V_e$, 
namely, the one containing the open segment $(x,y)$. 
We set $V_{e,U} := W_{e,U} \cap V_e$, and $i_{e,U} \colon V_{e,U} \to V$ 
so that it coincides with $i_e$ on its domain of definition.
We replace $W_e$ with the family $(W_{e,U})_{e, U}$.

Clearly $(V,(W_{d,\vect{v}}, W_{e,U})_{d, \vect{V}, e, U})$ defines a graph of $\nR$-trees 
equivalent to $(V,W)$, and satisfying property \ref{item:G4}. Therefore it is regular.

\begin{defi}   \label{def:regraph}
   The graph of $\nR$-trees $(V,(W_{d,\vect{v}}, W_{e,U})_{d, \vect{V}, e, U})$ constructed 
    above is called the \emph{regularization} of $(V,W)$.
\end{defi}

\begin{ex}\label{ex:goRt1}
On the bottom left part of \reffig{goRt1}, we can see the regularization 
$(V,W')$ of $(V,W)$ considered in \refex{goRt0}. In this case, $W'$ 
consists of ten tree elements. Notice that the number of edge elements remains unchanged.
\end{ex}

Given a graph of $\nR$-trees $(V,W)$, one can define refinements of its structure 
by adding new vertices.
Assume for simplicity that $(V,W)$ is regular (analogous constructions can be done 
in the non-regular case).
Denote by $Z$ the realization of $(V,W)$, and let $p \in Z\setminus V$ be any point.
Since $p$ is not a vertex, it belongs to a unique tree element $W_a$.

If $W_a$ is a decoration element with marked point $x$, we consider the $\nR$-tree 
$W'_a=W_a$ with marked points $x$ and $p$.
Set $V'=V \cup \{p\}$, then $i'_a(x)=i_a(x)$ and $i'_a(p)=p$.
Taking $V'$ as set of vertices, and the family $W'$ obtained from $W$ by replacing 
$W_a$ with $W'_a$, we get a new (in general non-regular) graph of $\nR$-trees, 
equivalent to $(V,W)$. Notice that in this case the number of vertices and edges increases by one.
Moreover, the skeleton $S(V',W')$ strictly contains $S(V,W)$.

If $W_a$ is an edge element with marked points $x$ and $y$, set $z=x \wedge_p y$ 
and $V'=V \cup \{p,z\}$. For any tangent vector $\vect{v} \in T_z W_a$, define 
$W'_a(\vect{v})$ as the closure of $U_x(\vect{v})$ in $W_a$. 
Set $V'_a(\vect{v}) := W'_a(\vect{v}) \cap V'$. Notice that $V'_a(\vect{v})$ 
always contains $z$, and contains another point in $V'$ in at most three cases 
(associated to the tangent vectors towards the elements $p,x,y$).
We define $i'_{a,\vect{v}} \colon V'_a(\vect{v}) \to V'$ 
similarly to the previous case.
The couple $(V',W')$, where $W'$ is the family obtained from $W$ 
by replacing $W_a$ with the family $W'_a(\vect{v})$, defines again a graph of 
$\nR$-trees equivalent to $(V,W)$.
In this case the numbers of vertices and of edges increase either by $1$ or by $2$,  according to 
the cases $p \in (x, y)$ or $p \notin (x,y)$.
Finally, also in this case $S(V',W') \supseteq S(V,W)$, with equality if and only if $p \in S(V,W)$.

\begin{defi}\label{def:refin}
Any finite composition of the operation described above and regularizations will be called a \emph{refinement} of the graph structure $(V,W)$.
\end{defi}

\begin{figure}[ht]
\centering
\begin{minipage}{0.45 \columnwidth}
\def\svgwidth{0.97 \columnwidth}
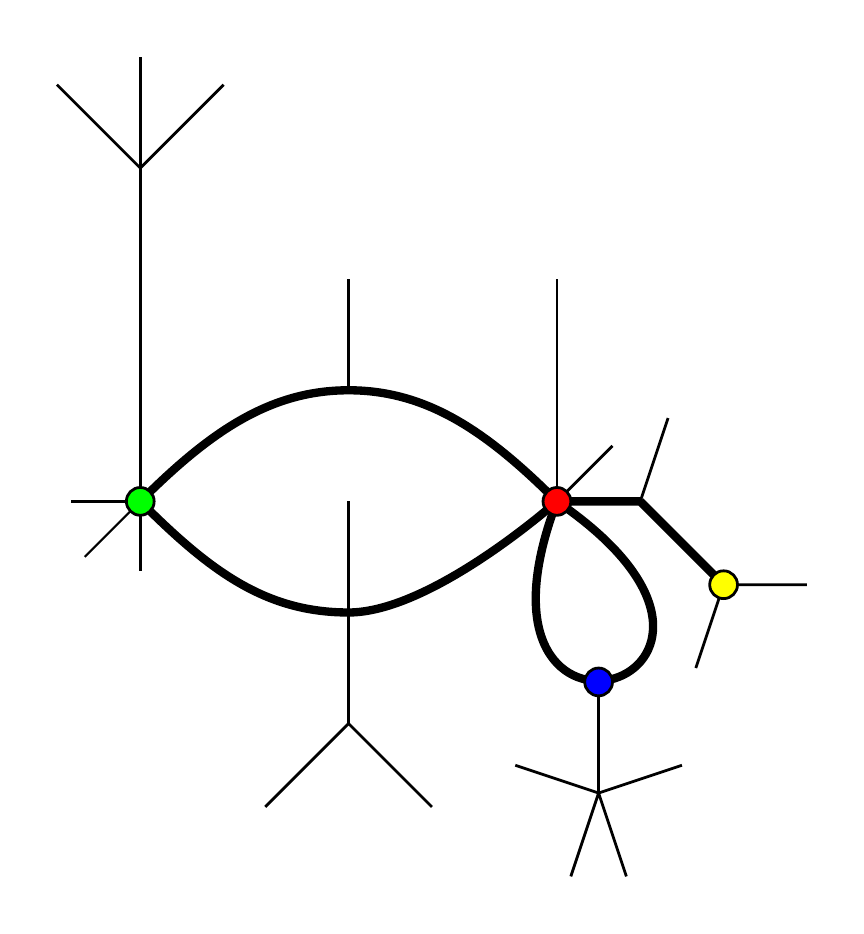
\end{minipage}
\begin{minipage}{0.5 \columnwidth}
\def\svgwidth{1.0 \columnwidth}
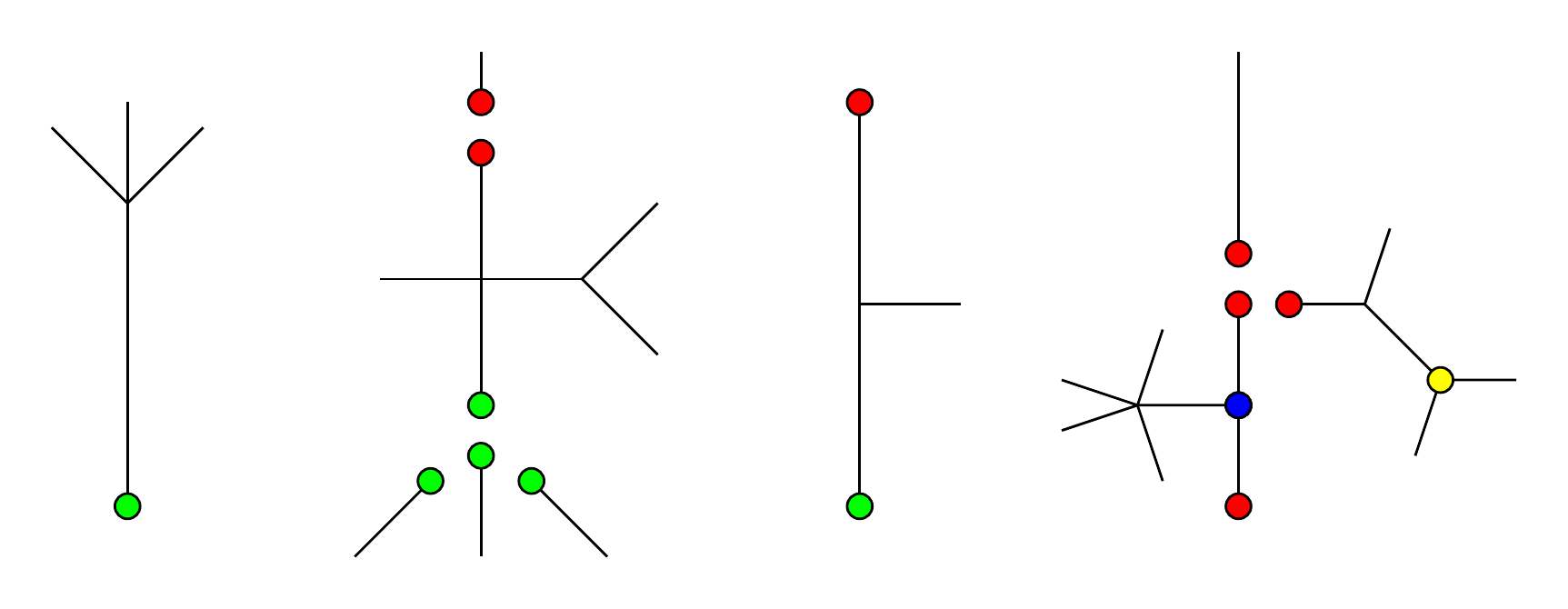

\vspace{8mm}

\def\svgwidth{1.0 \columnwidth}
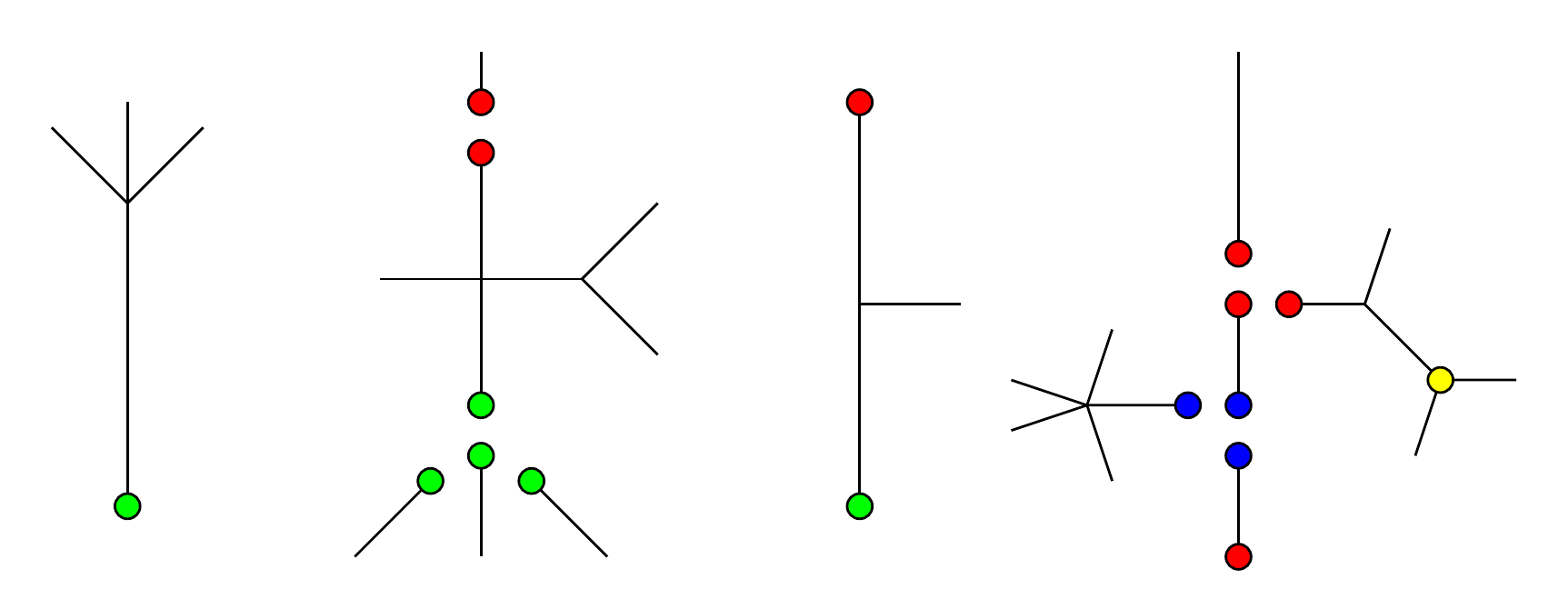
\end{minipage}
\caption{Refinement of a graph of $\nR$-trees.}\label{fig:goRt2}
\end{figure}

\begin{ex}
Consider again the regular graph $(V,W')$ described by \refex{goRt1}, 
with realization $Z$, depicted in \reffig{goRt1}. In the left part of 
\reffig{goRt2} we added two vertices, depicted in blue and yellow, 
obtaining four vertices $V'=\{r_e,r_g, r_b, r_y\}$.
The two new vertices belong to unique tree elements, that one can see in the 
top right part of the picture.
In the bottom right, we describe the (double) refinement $(V',W'')$ of $(V,W')$ 
with respect to these two new vertices.
The yellow vertex belongs to a decoration element. In this case the new 
element associated becomes an edge element, and we add a segment to the 
skeleton (represented by thick segments).
The blue vertex belongs to an edge element, and to the skeleton $S(V,W')$. 
In this case, this edge element splits in two edge elements, plus a decoration element.
\end{ex}

\begin{rmk}\label{rmk:transversetree}
Let $W$ be an edge element of some graph of $\nR$-trees, with marked points $x,y$.
For any point $z \in [x,y]$, define $\boxed{N_z}$ as  
${\displaystyle \bigcup}_{\vect{v}} U_z(\vect{v}) \cup \{z\}$, 
where $\vect{v}$ varies among the tangent vectors at $z$ not represented 
by either $x$ nor $y$.
It can be also described as the set of points $w \in W$ such that $[w,z] \cap [x,y] = \{z\}$.
The set $N_z$ admits a natural $\nR$-tree structure, as a subtree of the tree element $W$. 
It can be also seen as an $\nR$-tree \emph{rooted} at $z$, or again as a graph of $\nR$-trees 
with a single vertex $z$ and a single decoration tree. 
We will refer to $N_z$ as the \emph{tree at $z$ transverse to $[x,y]$}. 
It will be used below to define \textit{implosions} of graphs of 
$\R$-trees (see \refdef{blockimplosion}). 
\end{rmk}

\medskip
\subsection{Bricks and the brick-vertex tree of a graph of $\R$-trees} 
\label{ssec:brickrtree}
$\:$

\medskip

In this section we extend the notions of \emph{brick} and of \emph{brick-vertex tree} 
to graphs of $\R$-trees (see \refdef{genbvtree}).
In the next section, we will apply this extended notion of brick-vertex tree to 
the semivaluation space $\Val{X}$ of a normal surface singularity $X$, proving first 
that it has a structure of graph of $\nR$-trees, and getting then \refthm{udbv_val}, 
which is the counterpart of \refthm{ultramthm} for semivaluation spaces.
\medskip

The following is an analog of \refdef{separ}:

\begin{defi}
Let $Z$ be the realization of a graph of $\nR$-trees, and $x,y,z$ three points of $Z$.
We say that $z$ \emph{separates} $x$ and $y$ if 
either $z \in \{x,y\}$ or $x$ and $y$ belong 
to different connected components of $Z\setminus \{z\}$.
\end{defi}

Notice that $z$ separates $x$ and $y$ if and only if all arcs between $x$ and $y$ contain $z$.

In this section, unless it is specified differently, we will assume that the point 
$z$ separating $x$ and $y$ never belongs to $\{x,y\}$.

Let us formulate now an analog of \refdef{blocks}:

\begin{defi}  \label{def:blocksbis}
Let $Z$ be the realization of a graph of $\nR$-trees. A subset $C \subseteq Z$ is called 
\emph{cyclic} if for every couple $(x,y)$ of distinct points of $C$,  
no point $z \in C \setminus \{x,y\}$ separates them.
A \emph{cyclic element} of $Z$ is a cyclic subset which is maximal with respect to inclusion.
A cyclic element is called a \emph{brick} if it does not consist of a single point.
\end{defi}

Notice that if $C= \{x\}$, then $C$ is a cyclic element if and only if for all $y \in Z  \setminus \{x\}$ 
there exists $z \in Z  \setminus \{x,y\}$ such that $z$ separates $x$ and $y$ in $Z$.

\begin{prop}\label{prop:blocksinskeleta}
Let $Z$ be the realization of a graph $(V,W)$ of $\nR$-trees.
Then any brick of $Z$ is contained in the skeleton $S(V,W)$.
\end{prop}

\begin{proof}
Let $x$ be any point in $Z \setminus S(V,W)$. We want to prove that $\{x\}$ is a cyclic element of $Z$.
This is equivalent to showing that for any point $y \in Z \setminus \{x\}$, there exists a third 
point $z$ that separates $x$ and $y$.

Since $x \not \in S(V,W)$, there exists a unique $a \in A$ so that $x \in W_a$.
We first assume that $W_a$ is a decoration element, and denote by $z$ 
the unique point marked point of $W_a$.
Then $z$ separates $x$ and any point $y$ in $Z \setminus W_a$.
Let now $y$ be any point in $W_a \setminus \{x\}$. In this case, any point in $(x,y)$ 
separates $x$ and $y$.

Suppose now that $W_a$ is an edge element, say with ends $x_a,y_a$. By definition we have 
$W_a \cap S(V,W)=[x_a,y_a]$.
Set $z :=x_a \wedge_x y_a$. It belongs to $[x_a,y_a]$, and by our assumption it is different from $x$.
In this case, $z$ separates $x$ and any point outside the connected component $U$ of 
$W_a \setminus [x_a,y_a]$ containing $x$ (i.e. any point representing the tangent 
vector at $z$ towards $x$).
Finally, let $y$ be any point in $\overline{U} \setminus \{x\}$, 
where $\overline{U}= U \cup \{z\}$. Then the segment $[x,y]$ is contained in 
$\overline{U} \subseteq W_a$, and any point in $(x,y)$ separates $x$ and $y$.
\end{proof}

We deduce that the bricks of $Z$ may be identified with  
the bricks of the skeleton $S(V,W)$ with respect to its finite graph structure.

As an immediate consequence of \refprop{blocksinskeleta}, we get the following property 
of graphs of $\R$-trees, assumed as usual to be of finite type:

\begin{cor}\label{cor:finiteblocks}
Let $Z$ be the realization of a graph of $\nR$-trees.
     Then $Z$ has a finite number of bricks.
\end{cor}
\begin{proof}
Pick any graph structure $(V,W)$ whose realization is $Z$, and denote by 
$S=S(V,W)$ the skeleton associated to it, with its structure of finite graph.
Let $E=[x,y]$ be an edge of $S$. Then either $E$ is a bridge of $S$, 
in which case every point in $(x,y)$ is a cyclic element, or $E$ is not a bridge, 
and in this case $E$ belongs to a brick.
Since the number of edges is finite, so is the number of bricks.
\end{proof}

The absence of bricks characterizes the graphs of $\R$-trees whose 
realizations have again a  structure of $\R$-tree:

\begin{prop}\label{prop:allblocksaretrivial}
    Let $Z$ be the realization of a graph of $\nR$-trees. 
    Suppose that no cyclic element of $Z$ is a brick. Then $Z$ admits a structure of $\nR$-tree.
\end{prop}

\begin{proof}  Let us introduce an $\nR$-tree structure on $Z$ satisfying the conditions of \refdef{Rtree}.

Since all cyclic elements of $Z$ are points, we infer that for every couple of points $(x,y)$ in $Z$, 
there exists a unique arc $\gamma=\gamma(x,y)$ between $x$ and $y$.
To show this, suppose by contradiction that there are two such arcs that do not coincide. 
Then in the union of the two we have a cycle, which would be contained in a brick, against the assumption.

Fix any regular structure $(V,W)$ of graph of $\R$-trees, whose realization is $Z$.
Then $\gamma$ is a finite concatenation of segments $I_j=[s_j,s_{j+1}]$ 
contained in tree elements $W_{a_j}$.
We set $[x,y] = \gamma$, with the segment structure obtained by taking a concatenation 
of the orders given by the segment structures on $I_j$.
It is easy to see that \ref{item:T2} is satisfied for this family of intervals, while property 
\ref{item:T3} holds directly by construction.

To verify property \ref{item:T4}, we have to show that for any triple $x,y,z$ of points in $Z$, 
there exists a unique element $w=x \wedge_z y$ so that $[z,x] \: \cap\:  [y,x] = [w,x]$ and 
$[z,y] \: \cap \: [x,y] = [w,y]$.
The uniqueness of such $w$ is trivial, hence we only need to show its existence.
Consider the set $I=[z,x] \cap [z,y]$, with the partial order induced by the one in $[z,x]$. 
By uniqueness of arcs between two points, we infer that $I$ is itself a (possibly not closed) interval. 
Decompose $[z,x]=\bigcup_{j} [s_j,s_{j+1}]_{a_j}$ where $[s_j,s_{j+1}]_{a_j}$ belongs 
to $W_{a_j}$.
Let $k$ be the highest index for which $W_{a_k} \cap I \neq \emptyset$.
Notice that if $y \not\in W_{a_k}$, then $[z,y]$ intersects $W_{a_k} \: \cap \: V$ 
in a point $\tilde{s}$ different from $s_{k}$. 
Set:
\begin{itemize}
\item $x_k = x$ if $x \in W_{a_k}$, and $x_k=s_{k+1}$ otherwise;
\item $y_k = y$ if $y \in W_{a_k}$, and $y_k=\tilde{s}$ otherwise;
\item $z_k = z$ if $z \in W_{a_k}$, and $z_k=s_k$ otherwise.
\end{itemize}
Set now $w=x_k \wedge_{z_k} y_k$, the wedge being taken with respect  
to the tree structure on $W_{a_k}$. Clearly, $w$ satisfies property \ref{item:T4}.

Finally, property \ref{item:T5} clearly holds for $Z$. In fact, for any sequence of segments 
$[x,y_\alpha)$ in $Z$, there exists $z \in Z$ so that $[z,y_\alpha]$ belongs to a certain tree element 
$W_a$ for $\alpha$ big enough. Then property \ref{item:T5} derives directly from the analogous 
property for $W_a$.
\end{proof}

We want now to generalize the brick-vertex trees we defined for finite graphs to the case 
of graphs of $\nR$-trees. In order to get such a definition, we need first to introduce a few 
more constructions.

There is a natural way to associate an $\R$-tree to any non-empty set:

\begin{defi}
  Let $B$ be any non-empty set.
   Let $\sim$ be the equivalence relation on $B \times [0,1]$ defined by by $(x,s) \sim (y,t)$ 
   if and only if $(x,s)=(y,t)$ or $t=s=0$.
   The quotient
   $$\etoile{B}=\left. B \times [0,1]\right/ \sim$$
   is called the \emph{star} over $B$.
   We will denote by $\boxed{x_t}$ the class in $\etoile{B}$ corresponding to the point $(x,t)$, 
   and by $\boxed{\apex{B}}$ the \emph{apex} of $\etoile{B}$, 
   which is represented by $(x,0)$ for any $x \in B$.
\end{defi}

Each star $\etoile{B}$ is endowed with a natural structure of $\nR$-tree, whose definition 
we leave to the reader.

Let $(V,W)$ be a regular graph of $\nR$-trees, $Z$ be its realization, and $B$ be a 
brick of $Z$.
For any point $z \in B \setminus V$, there exists a unique edge element $W_{e(z)}$ 
containing $z$. We denote by $N_z$ the $\nR$-subtree at $z$ transverse to $e$ 
as defined in \refrmk{transversetree}.
Then, we consider the graph of $\nR$-trees $\boxed{N'_z}$ which has one vertex $\{z\}$, 
and two decorative elements:
\begin{itemize}
\item $N_z$, with marked point $\{z\}$, 
\item the segment $[\apex{B}, z_1] \subset \etoile{B}$, with marked point 
$z_1 =(z,1)$,
\end{itemize}
with natural identification maps. It is easy to see that $N'_z$ has no bricks.
In  \refdef{blockimplosion}, $N'_z$ will be considered just as an $\nR$-tree, 
with its structure given by \refprop{allblocksaretrivial}.

Given a brick $B$, let us denote by $\boxed{E(B)}$ the set of indices $e \in E$ 
such that the edge $[x_e,y_e]$ between the two marked points of an edge 
element $W_e$ is contained in $B$.

\begin{defi}\label{def:blockimplosion}
Let $(V,W)$ be a regular graph of $\nR$-trees, $Z$ be its realization, $B$ be a brick of $Z$.
For any $z \in B \setminus V$, consider the $\nR$-tree $N'_z$ as defined above.
Set $V'= V \cup \{\apex{B}\}$, and consider the family $W'$ of $\nR$-trees given by:
\begin{itemize}
\item the decorative elements $W_d$, $d \in D$, of $W$, with same marked point and 
     same identification map;
\item the edge elements $W_e$ with $e \in E \setminus E(B)$, with same marked points 
     and same identification map;
\item the decorative elements $N'_z$ for $z \in B \setminus V$, with marked point $\{v_B\}$ 
     and natural identification map;
\item the edge elements $[\apex{B},v_1] \subset \etoile{B}$, for any $v\in B \cap V$, 
     with marked points $\apex{B}$ and $v_1$, and identifications $i(\apex{B})=\apex{B}$ 
     and $i(v_1)=v$.
\end{itemize}
   Then $(V',W')$ is a graph of $\nR$-trees, which we call the \emph{implosion} 
    of $(V,W)$ along the brick $B$. 
    We denote by $Z'$ the realization of the graph $(V',W')$ 
    and by $i_B:Z \to Z'$ the associated natural injection.
\end{defi}

Note that the  injection $i_B:Z \to Z'$  is not continuous with respect to the weak topologies 
in $Z$ and $Z'$. 
This is due to the fact that the topology induced on $i_B(B)$ by the topology on $Z'$ 
is the discrete topology, which does not coincide with the topology induced on $B$ by 
the weak topology of $Z$ (which is the standard topology defined on a graph, see \refprop{blocksinskeleta}).
In other terms, we replaced the brick $B$ with its star $\etoile{B}$, 
and not with the \textit{cone} with base $B$, which corresponds to the analogous 
construction done by replacing the discrete topology on $B$ 
with the standard topology of its finite graph structure. 

\begin{prop}\label{prop:blockimplosion}
Let $(V,W)$ be a regular graph of $\nR$-trees, and $Z$ be its realization. 
Assume that $Z$ has $n\geq 1$ bricks, and let $B$ be any one of them.
Let $(V',W')$ be the implosion of $(V,W)$ along the brick $B$, and $Z'$ its realization.
Then $Z'$ has exactly $n-1$ bricks, given by the images through the natural 
injection $i_B$ of the bricks of $Z$ different from $B$.
\end{prop}

\begin{proof}
We only need to check that all points in $\etoile{B} \setminus i_B(B)$ form singleton
cyclic elements of $Z'$. 
By \refprop{blocksinskeleta}, the bricks of $Z'$ are contained in the skeleton $S(V',W')$, 
which intersects $\etoile{B}$ exactly in the edge elements $[\apex{B},v_1]$ with 
$v \in V \cap B$ (see \refdef{blockimplosion}).
Let $w$ be any point in $\etoile{B} \setminus i_B(B)$, and assume by contradiction 
that $w$ is contained in a brick $B'$.
Since $\etoile{B}$ is a tree, we get that $B' \cap (Z' \setminus \etoile{B}) =:C \neq \emptyset$.
But then, $B \cup i_B^{-1}(C)$ would be a cyclic subset of $Z$ strictly containing $B$, 
which is in contradiction with the maximality of $B$ with respect to inclusion.
\end{proof}

Given any graph of $\nR$-trees, we can apply recursively regularizations and brick implosions, 
in order to \textit{kill} all bricks. In fact, by \refcor{finiteblocks}, the number of bricks is finite, 
and by \refprop{blockimplosion}, the number of bricks strictly decreases under brick implosion.
The final product of this process will be a graph of $\nR$-trees $(V',W')$, in which 
all cyclic elements are  singletons. 
By \refprop{allblocksaretrivial}, its realization $Z'$ admits a structure of $\nR$-tree.
It is the brick-vertex tree of the starting graph of $\R$-trees:

\begin{defi}   \label{def:genbvtree}
     Let $Z$ be the realization of a graph of $\nR$-trees $(V,W)$, and $Z'$ be the $\nR$-tree 
     described above, obtained by recursive regularizations and brick implosions of all bricks of $Z$. 
     Then $Z'$ is called the \emph{brick-vertex tree} of $Z$, and denoted by 
     $\boxed{\Rbvt{Z}}$.
     The points of $Z'$ corresponding to apices of bricks of $Z$ are called \emph{brick points} 
     of the brick-vertex tree.
    We denote by $\Rbv \colon Z \to Z'$ the natural injection obtained by the composition 
    of the natural injections $i_B$ described above for brick implosions.
\end{defi}

Note that if $B, B'$ are two bricks of a graph of $\nR$-trees $Z$, and $Z'$ is the implosion of $B$, 
then $i_B(B')$ is a brick in $Z'$. It follows that the brick-vertex tree of $Z$ does not depend 
on the order in which we perform the brick implosions.

We end this section with a remark about the notion of cyclic element from a topological perspective.

\begin{rmk}  \label{rmk:cyclelemtheory}
The term \textit{cyclic element} is standard in general topology, while that of  \textit{brick} 
was introduced by us in order to get a common denomination for the graph-theoretic blocks 
which are not bridges and for the cyclic elements which are not points.
Indeed, while the notion of block is combinatorial and that of cyclic element is topological, 
the underlying topological space of a brick of a finite graph is a brick of its underlying 
topological space (see \refprop{blocksinskeleta}).

Cyclic elements can be defined for much more general topological spaces 
than for finite graphs or realization spaces of graphs of $\nR$-trees. 
This notion was introduced by Whyburn in his 1927 paper \cite{whyburn:cyclconncontcurves},  
as a mean to describe the overall structure of \textit{Peano continua}, i.e., 
the compact connected metric spaces which may be obtained as continuous 
images of the real interval $[0,1]$ inside some Euclidean space $\nR^n$. He defined the 
\textit{cyclic elements} of such a topological space as its maximal subsets $C$ 
such that any two distinct points of them are contained in a circle topologically embedded in $C$. 
In fact, he initially studied only \textit{plane Peano continua}, and he extended in later papers 
the theory to arbitrary ones using ingredients from Ayres' 1929 paper \cite{ayres:1929}. 
Later on, in the 1930 paper \cite{kuratowski-whyburn:elemcyclappl},  
Kuratowski and Whyburn simplified the theory of cyclic elements by defining them as 
in \refdef{blocksbis} above. 

The main point of this theory was to explain that the cyclic elements of a Peano continuum 
are organized in a tree-like manner. For instance, given any two cyclic elements, there is 
a unique connected union of cyclic elements which contains them and is minimal for inclusion -- 
this is an analog of the uniqueness of path joining two points of a tree. 

Later, the theory of cyclic elements was extended to more general settings (see e.g. \cite{whyburn:cutpointsgentopspaces,lehman:cyclicelemtheory, nikiel-tunkali-tymchatyn:contimagesarcs} 
as well as the references in McAllister's surveys \cite{mcallister:1966}, \cite{mcallister:1981}
 of the theory up to 1966 and in the interval 1966--81 respectively). 
In fact, as pointed out by Rado and Reichelderfer in \cite{rado-reichelderfer:cyclictrans}, 
most of the results 
of the theory can be obtained in the very general situation of a set endowed with a 
``cyclic transitive relation'' (a cyclic transitive relation $\mc{R}$ on a set $S$ is a binary relation 
which is reflexive, symmetric, and such that if 
$x_1\ \mc{R}\ x_2\ \mc{R} \ldots \mc{R}\ x_n\ \mc{R}\ x_1$, then 
$x_i\ \mc{R}\ x_j$ for all $i,j = 1\ldots, n$). 
In particular, in this generality one does not need topological spaces in order to talk about 
cyclic elements.  This last aspect is very interesting in our setting, since as already pointed out, 
valuative spaces carry two natural topologies, with quite different properties 
(the weak topology is non-metrizable, and the space is compact and locally compact, 
while the strong topology is metrizable, but the space is not locally compact).

Let us mention that the Peano spaces in which all the cyclic elements 
are points are called \textit{dendrites} 
(see \cite{whyburn:whatisacurve}). Wa\.zewski proved in \cite{wazewski:1923} the existence of a 
\textit{universal dendrite}, in which embed all other dendrites. Recently, 
Hrushovski, Loeser and Poonen found in 
\cite[Corollary 8.2]{hrushovski-Loeser-Poonen:berkspacesembed} 
a representation of it as a special type of valuation space, under a countability hypothesis on the base field. 

In what concerns the relation between cyclic element theory of topological spaces and block theory 
of graphs, it is interesting to note that in the paper \cite{whitney:1932}, 
in which Whitney introduced the notion of \textit{nonseparable graph} (see \refdef{sepgraph}), 
he quoted an article of Whyburn on cyclic element theory, 
but after that date the two fields seem to have evolved quite independently of each other.
\end{rmk}

\medskip
\subsection{Valuation spaces as graphs of $\nR$-trees}
\label{ssec:semivalgraphRtree}
$\:$

\medskip

In this section we apply the constructions of the previous section to the space 
of normalized semivaluations associated to a normal surface singularity. 
We first prove that it 
admits a structure of connected graph of $\R$-trees (see \refprop{valXisagraph}). 
Then we prove the valuative analog of \refthm{ultramthm}, stating that the functions 
$\udv{\lambda}$ 
are ultrametrics on special types of subspaces of the space of normalized 
semivaluations (see \refthm{udbv_val}). We conclude the paper with several 
examples which show that the hypotheses of the theorem are not necessary in order 
to get ultrametrics.

\begin{prop}\label{prop:valXisagraph}
      Let $X$ be a normal surface singularity, and $\Val{X}$ its associated space of 
      normalized semivaluations. Then $\Val{X}$ admits a structure of 
      connected graph of $\nR$-trees, that is, it is a connected realization space 
      of a graph of $\nR$-trees. More precisely, any good resolution defines 
      canonically such a structure.
\end{prop}

\begin{proof}
Let $\pi \colon X_\pi \to X$ be any good resolution. We set $V$ as the set of 
divisorial valuations associated to the primes of $\pi$.
For any point $p \in \pi^{-1}(x_0)$, we set $W_p = \overline{U_\pi(p)}$, which consists 
in the set $U_\pi(p)$ of all semivaluations whose center in $X_\pi$ is $p$, plus the 
divisorial valuations of the form $\nu_E$ with $E \ni p$ (which belong to $V$).
Since $\pi^{-1}(x_0)$ has simple normal crossings, either $p$ belongs to a unique 
prime $E$ of $\pi$, in which case we declare $W_p$ a decoration element, 
with marked point $\nu_E$, or $p$ belongs to exactly two exceptional primes $E$ and $F$, 
in which case we declare $W_p$ an edge element, with marked points $\nu_E$ and $\nu_F$.
Since for any such $p$, the germ $(X_\pi,p)$ is  smooth, the set $W_p$ is isomorphic to 
the valuative tree, hence it is an $\nR$-tree.
The couple $(V,(W_p)_{p \in \pi^{-1}(x_0)})$ defines a structure of graph of $\nR$-trees on $\Val{X}$.
\end{proof}

\begin{ex}
In \reffig{example01_val}, we may see on the left the dual graph $\Gamma_\pi$ of 
a good resolution $\pi$ of some normal surface singularity $X$.
In this example, there are $3$ bricks, depicted in orange, blue and yellow.
On the right side, we may see a depiction of the semivaluation space $\Val{X}$. The structure 
of a graph of $\nR$-trees induced by $\pi$ in this case has as vertices the vertices 
of $\Gamma_\pi$ under identification with the corresponding valuations (we denoted 
them as $\skeldiv{\pi}$), edge elements correspond to the trees along the edges of 
$\Gamma_\pi$, and all other tree elements are decorations. The thick colored segments 
correspond to bricks of $\Val{X}$ with respect to its structure of graph of $\nR$-trees.
\end{ex}

\begin{figure}[ht]
\centering
\hspace{-2cm}
\def\svgwidth{0.6 \columnwidth}
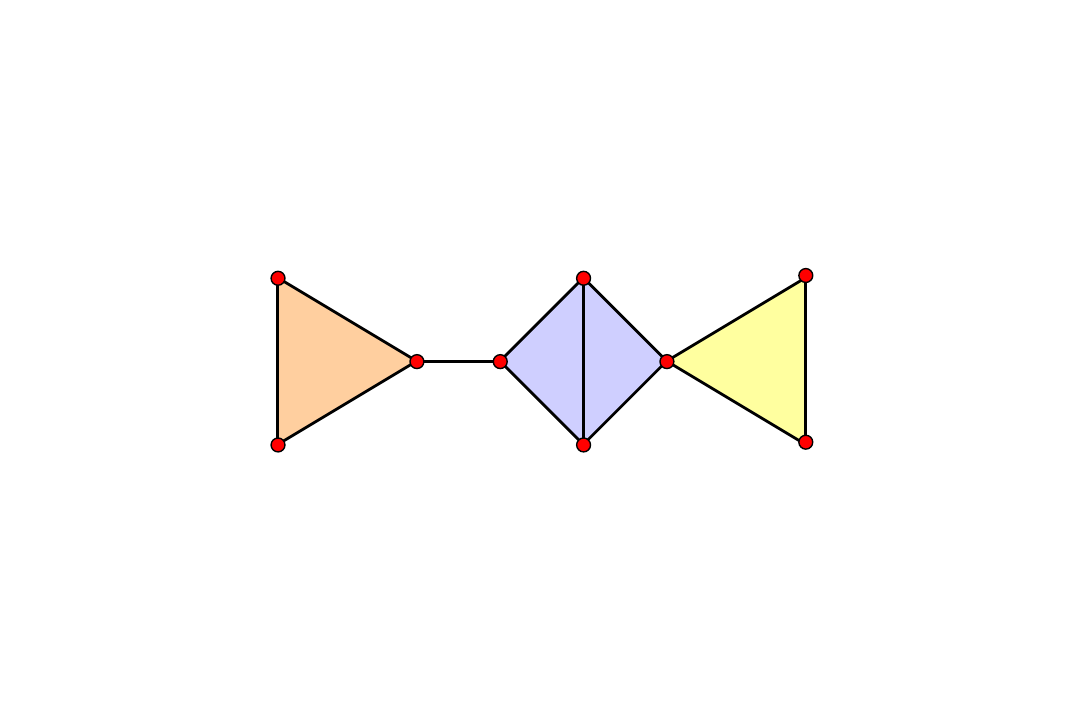
\hspace{-2cm}
\def\svgwidth{0.6 \columnwidth}
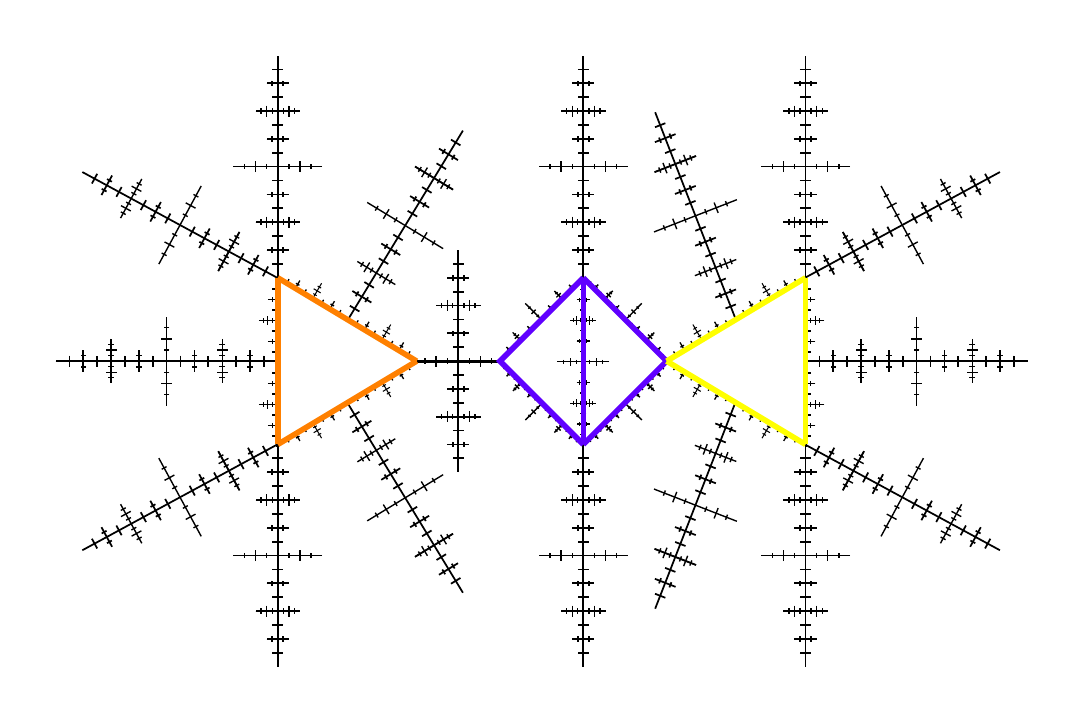
\caption{The dual graph associated to a good resolution $\pi$ of a normal 
surface singularity $X$, with bricks shaded, and its associated 
space $\Val{X}$ of normalized semivaluations.}\label{fig:example01_val}
\end{figure}

We are now able to state and prove the following theorem, which is an analog of \refthm{ultramthm} for valuation spaces: 

\begin{figure}[ht]
\centering
\def\svgwidth{1 \columnwidth}
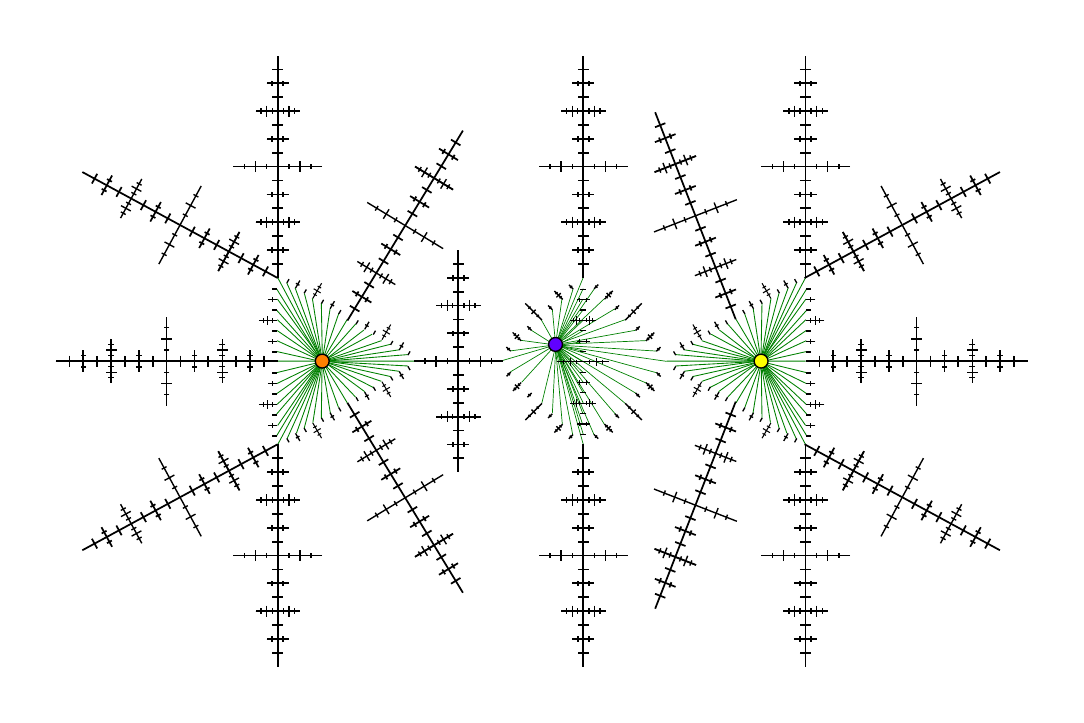
\caption{The brick-vertex tree $\Rbvt{\Val{X}}$ for the example of \reffig{example01_val}.}\label{fig:example01_bvt}
\end{figure}

\begin{figure}[ht]
\centering
\def\svgwidth{1 \columnwidth}
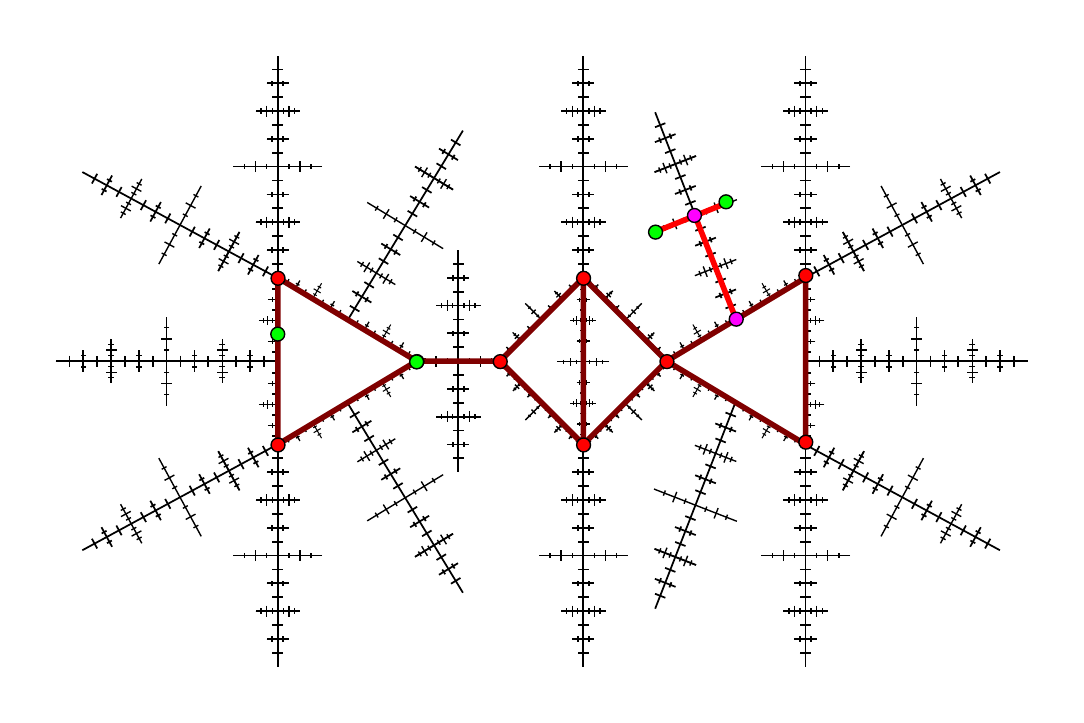
\caption{Graphs embedded in $\Val{X}$, illustrating 
     the proof of \refthm{udbv_val}.}\label{fig:example01_4points}
\end{figure}

\begin{thm}\label{thm:udbv_val}
    Let $X$ be a normal surface singularity, $\Val{X}$ the associated space 
    of normalized semivaluations, and $\mathcal{J} \subseteq \Val{X}$ any subset of it.
   Let $\Rbvt{\Val{X}}$ be the brick-vertex tree of $\Val{X}$, and consider its subtree $W=   
   \cvxh{\Rbv(\mathcal{J})}$.
   If $T_{\apex{B}} W$ consists of at most $3$ points for every brick point $\apex{B} \in W$, 
    then $\udv{\lambda}$ defines an extended ultrametric distance on $\mathcal{J}$, 
    for any $\lambda \in \mathcal{J}$.
\end{thm}

\begin{proof}
Fixed any $\lambda \in \mathcal{J}$, we need to prove that
\begin{equation}\label{eqn:ineq_udv}
\udv{\lambda}(\nu_1, \nu_3) \leq \max\{\udv{\lambda}(\nu_1, \nu_2), \udv{\lambda}(\nu_2, \nu_3)\}
\end{equation}
for any triple $\nu_1, \nu_2, \nu_3 \in \mathcal{J}$.
Notice that \eqref{eqn:ineq_udv} is satisfied if either $\nu_1$, $\nu_2$ or $\nu_3$ 
coincide with $\lambda$. 
Say for example that $\nu_2=\lambda$. Then $\udv{\lambda}(\nu_1, \nu_2)=\udv{\lambda}(\nu_2, \nu_3)=\bra{\lambda, \lambda} = \alpha(\lambda)$, while $\udv{\lambda}(\nu_1, \nu_3)=\frac{\bra{\lambda, \nu_1} \bra{\lambda, \nu_3}}{\bra{\nu_1, \nu_3}} \leq \bra{\lambda, \lambda}$ by \refprop{positivity_bdivisors}. 
We may hence assume that $\lambda \not \in \{\nu_1, \nu_2, \nu_3\}$. In particular, the three values in \eqref{eqn:ineq_udv} are finite.

By proceeding as in \refprop{reformultra} and \refcor{equivultr}, we get that \eqref{eqn:ineq_udv} 
is equivalent to showing that $\logfunc$ is tree-like, i.e., it satisfies the $4$-point condition 
\eqref{eqn:4pointmax}.  Set $J :=\{\nu_1, \nu_2, \nu_3, \nu_4\} \subset  \mathcal{J}$.

Take any good resolution $\pi \colon X_\pi \to X$. Any semivaluation $\nu \in J$ either 
belongs to $\skeldiv{\pi}$, or it belongs to the weakly open set $U_\pi(p)$ associated 
to the center $p=p(\nu) \in \pi^{-1}(x_0)$ of $\nu$ in $X_\pi$.
Let $\skel{\pi}$ denote the skeleton associated to $\pi$, and let $\Gamma$ be the 
subset of $\Val{X}$ given by the union of $\skel{\pi}$ and the segments 
$[\nu_E, \nu] \subset \ol{U_\pi(p)}$, where $p=p(\nu)$ is as above, $E$ is any 
exceptional prime of $\pi$ containing $p$, and $\nu$ varies in $J$.
The set $\Gamma$ admits a structure of finite graph. In fact, up to taking higher 
good resolutions, we may assume that for any distinct $\nu,\nu' \in J$, their centers 
in $X_\pi$ are also distinct. We may also assume that any semivaluation in $J$ either 
belongs to $\skel{\pi}$, or its center in $X_\pi$ is a smooth point of $\pi^{-1}(x_0)$.
In this case, the structure of finite graph on $\Gamma$ has as vertices $\skeldiv{\pi} \: \cup \: J$, 
and as edges all the edges in $\skel{\pi}$, eventually cut by elements in $J \: \cap \: \skel{\pi}$, 
plus all the edges associated to the segments $[\nu_E, \nu]$ with $\nu \in J$ 
as described above (see \reffig{example01_4points}).

The function $\logfunc$ defines a distance on the set of vertices of $\Gamma$, 
satisfying the condition \eqref{eqn:additive_graph_metric}.
This is a consequence of \refprop{positivity_bdivisors} applied to the reformulations given in \refprop{MW}.

Consider now the brick-vertex tree $\bvt{\Gamma}$ associated to $\Gamma$. 
The embedding of $\Gamma$ in $\Val{X}$ induces an embedding of $\bvt{\Gamma}$ inside $\Rbvt{\Val{X}}$.
Since the tangent space of $W$ at any brick point consists at most of $3$ points, 
the $\cvxh{J}$-valency of any brick point of $\Gamma$ is at most $3$.
We can apply \refthm{valblocks}, and deduce that $\logfunc$ is tree-like on the set $J$, 
and we are done.
\end{proof}

Notice that, as in the case of finite graphs, we get again the proof of the implication 
(\ref{arbsingul_val}) $\Longrightarrow$ (\ref{forany_val}) of \refthm{arborcase_val} 
as a direct corollary of \refthm{udbv_val}.

\begin{ex}
\reffig{example01_bvt} depicts the brick-vertex tree associated to the semivaluation space $\Val{X}$ represented in \reffig{example01_val}.
The thick vertices in orange, blue and yellow 
denote the brick-vertices of $\Rbvt{\Val{X}}$, while the dark green segments belong 
to the stars on them. The image needs to be thought with the green part not intersecting 
the rest of the space.

In \reffig{example01_4points} consider a set $J$ of four semivaluations in $\Val{X}$ as in 
the proof of \refthm{udbv_val}, that are depicted in light green. The dark red area denotes 
the skeleton associated to the minimal good resolution of $X$, while the light red 
part corresponds to the part added to $\skel{\pi}$ to obtain $\Gamma$.
The thick red dots correspond to the divisorial valuations in $\skeldiv{\pi}$ (not belonging to $J$), 
while the pink-purple dots are the rest of divisorial valuations added for describing 
the graph structure on $\Gamma$.
\end{ex}

\begin{ex}\label{ex:tetrahedron_val}
As for its counterpart for finite sets of branches formulated in \refthm{ultramthm}, 
the condition on the valency of brick-points in \refthm{udbv_val} is not necessary in general.
Consider again the singularity studied in \refex{tetrahedron}, whose minimal good 
model $X_\pi$ has four exceptional primes $E_1, \ldots, E_4$ of self-intersection $-4$,  
which intersect transversely each another.
The skeleton associated to it is the $1$-skeleton of a tetrahedron.
Denote by $\nu_j$ the prime divisorial valuation associated to $E_j$ for all $j=1 \ldots, 4$, 
and denote by $\mu_t$ the monomial valuation at the intersection point $p$ between 
$E_1$ and $E_2$, so that $Z_\pi(\mu_t) = (1-t)\check{E_1} + t \check{E_2}$.

Since all these valuations belong to the skeleton $\skel{\pi}$, which is included in a unique brick, any choice of $4$ valuations $a,b,c,d$ among 
$\nu_1, \nu_2, \nu_3, \nu_4, \mu_t$ for $0 < t < 1$ would not satisfy 
the hypotheses of \refthm{udbv_val}.
By computing the bracket between $\mu_t$ and $\nu_j$, we get
\[
   5\bra{\nu_1,\mu_t} = 2-t, \qquad 5\bra{\nu_2,\mu_t} = 1+t, 
     \qquad 5\bra{\nu_3,\mu_t} = 5\bra{\nu_4,\mu_t} = 1.
\]

For any choice of $4$ valuations $a,b,c,d$, we consider now the values 
$I_1=25\bra{a,b}\bra{c,d}$, $I_2=25\bra{a,c}\bra{b,d}$ and $I_3=25\bra{a,d}\bra{b,c}$.
We recall that $a,b,c,d$ satisfy the $4$-point condition if and only if two out of 
these three values coincide and the third is greater or equal to the other two.
First, pick the quadruple $\nu_1, \mu_t, \nu_3, \nu_4$: we get $I_1= 2-t$, $I_2=I_3=1$. 
In this case the $4$-point condition is satisfied.
Then, pick the quadruple $\nu_1, \mu_t, \nu_2, \nu_3$: we get $I_1 = 2-t$, $I_2= 1$, $I_3 = 1+t$. 
In this case the $4$-point condition is never satisfied.
\end{ex}

\begin{ex}\label{ex:tetrahedron_val_modif}
We saw in \refex{tetrahedron_val} how the validity of the $4$-point condition may 
depend on the valuation when it varies inside the same brick. 
We now investigate how it varies when changing the self-intersections of prime 
divisors in some model.
To this end, consider again the singularity $X$ defined in \refex{tetrahedron_val}, 
and the point $p$ of intersection of $E_1$ and $E_2$. 
Denote by $E_5$ the exceptional prime divisor corresponding to the blow-up of $p$.
In this new model $X_{\pi'}$, the self intersections of the strict transforms of 
$E_j$, $j=1, \ldots, 4$, and of $E_5$, are respectively $-5, -5, -4, -4, -1$.

Consider now the normal surface singularity $Y$ whose minimal resolution has the 
same dual graph as of $X_{\pi'}$, but satisfying $E_5^{2} = -2$ instead of $-1$.
Denote by $\nu_j$ the prime divisorial valuation associated to $E_j$ for all 
$j=1 \ldots, 4$ and by $\nu_5$ the one associated to $E_5$. Let  $\mu'_t$ be 
the monomial valuation at the intersection between the strict transform of 
$E_2$ and $E_5$, so that $Z_{\pi'}(\mu'_t) = (1-t)\check{E_2} + t \check{E_5}$.
In this case, we get
\[
      80\bra{\nu_1,\mu'_t} = 7+8t, \qquad 80\bra{\nu_3,\mu'_t} = 80\bra{\nu_4,\mu'_t} = 10.
\]
For the choice of valuations $a,b,c,d$ given by $\nu_1, \mu'_t, \nu_3, \nu_4$, 
we consider $I_1=80^2\bra{a,b}\bra{c,d}$, $I_2=80^2\bra{a,c}\bra{b,d}$ and 
$I_3=80^2\bra{a,d}\bra{b,c}$.
In this case we get   $I_2=I_3 = 100$ and $I_1 =12 (7+8t)$.

In particular, we notice that the $4$-point condition is satisfied for this quadruple 
if and only if $t \geq \frac{1}{6}$.
Notice also that $\mu'_t$ parametrizes the segment $[\nu_2, \nu_5]$, 
which is contained in the segment $[\nu_2, \nu_1]$.
The situation here is quite different from the one described in \refex{tetrahedron_val}, 
where the $4$-point condition of the quadruple $\nu_1, \mu_t, \nu_3, \nu_4$ 
was satisfied for any choice of $\mu_t$.
In particular, the valuations $\nu_1, \nu_2, \nu_3, \nu_4$ satisfy the $4$-point 
condition for $X$, but they do not satisfy the $4$-point condition for $Y$.
\end{ex}

\bibliographystyle{abbrv}
\bibliography{biblio}

\end{document}

%% file: graph1_blocks.pdf_tex
\begingroup%
  \makeatletter%
  \providecommand\color[2][]{%
    \errmessage{(Inkscape) Color is used for the text in Inkscape, but the package 'color.sty' is not loaded}%
    \renewcommand\color[2][]{}%
  }%
  \providecommand\transparent[1]{%
    \errmessage{(Inkscape) Transparency is used (non-zero) for the text in Inkscape, but the package 'transparent.sty' is not loaded}%
    \renewcommand\transparent[1]{}%
  }%
  \providecommand\rotatebox[2]{#2}%
  \ifx\svgwidth\undefined%
    \setlength{\unitlength}{184.8bp}%
    \ifx\svgscale\undefined%
      \relax%
    \else%
      \setlength{\unitlength}{\unitlength * \real{\svgscale}}%
    \fi%
  \else%
    \setlength{\unitlength}{\svgwidth}%
  \fi%
  \global\let\svgwidth\undefined%
  \global\let\svgscale\undefined%
  \makeatother%
  \begin{picture}(1,0.85521999)%
    \put(0,0){\includegraphics[width=\unitlength]{graph1_blocks.pdf}}%
    \put(0.16040614,0.74232765){\color[rgb]{0,0,0}\makebox(0,0)[lb]{\smash{$\Gamma$}}}%
  \end{picture}%
\endgroup%

%% file: graph1_bvt.pdf_tex
\begingroup%
  \makeatletter%
  \providecommand\color[2][]{%
    \errmessage{(Inkscape) Color is used for the text in Inkscape, but the package 'color.sty' is not loaded}%
    \renewcommand\color[2][]{}%
  }%
  \providecommand\transparent[1]{%
    \errmessage{(Inkscape) Transparency is used (non-zero) for the text in Inkscape, but the package 'transparent.sty' is not loaded}%
    \renewcommand\transparent[1]{}%
  }%
  \providecommand\rotatebox[2]{#2}%
  \ifx\svgwidth\undefined%
    \setlength{\unitlength}{184.8bp}%
    \ifx\svgscale\undefined%
      \relax%
    \else%
      \setlength{\unitlength}{\unitlength * \real{\svgscale}}%
    \fi%
  \else%
    \setlength{\unitlength}{\svgwidth}%
  \fi%
  \global\let\svgwidth\undefined%
  \global\let\svgscale\undefined%
  \makeatother%
  \begin{picture}(1,0.85521999)%
    \put(0,0){\includegraphics[width=\unitlength]{graph1_bvt.pdf}}%
    \put(0.16040614,0.74232765){\color[rgb]{0,0,0}\makebox(0,0)[lb]{\smash{$\bvt{\Gamma}$}}}%
  \end{picture}%
\endgroup%

%% file: biggraph1dual.pdf_tex
\begingroup%
  \makeatletter%
  \providecommand\color[2][]{%
    \errmessage{(Inkscape) Color is used for the text in Inkscape, but the package 'color.sty' is not loaded}%
    \renewcommand\color[2][]{}%
  }%
  \providecommand\transparent[1]{%
    \errmessage{(Inkscape) Transparency is used (non-zero) for the text in Inkscape, but the package 'transparent.sty' is not loaded}%
    \renewcommand\transparent[1]{}%
  }%
  \providecommand\rotatebox[2]{#2}%
  \ifx\svgwidth\undefined%
    \setlength{\unitlength}{288.8bp}%
    \ifx\svgscale\undefined%
      \relax%
    \else%
      \setlength{\unitlength}{\unitlength * \real{\svgscale}}%
    \fi%
  \else%
    \setlength{\unitlength}{\svgwidth}%
  \fi%
  \global\let\svgwidth\undefined%
  \global\let\svgscale\undefined%
  \makeatother%
  \begin{picture}(1,0.97229917)%
    \put(0,0){\includegraphics[width=\unitlength]{biggraph1dual.pdf}}%
    \put(0.95807671,0.4344766){\color[rgb]{0,0,0}\makebox(0,0)[lb]{\smash{$a_1$}}}%
    \put(0.95696261,0.55256513){\color[rgb]{0,0,0}\makebox(0,0)[lb]{\smash{$a_2$}}}%
    \put(0.62274981,0.64391665){\color[rgb]{0,0,0}\makebox(0,0)[lb]{\smash{$a_3$}}}%
    \put(0.55169149,0.88891589){\color[rgb]{0,0,0}\makebox(0,0)[lb]{\smash{$a_4$}}}%
    \put(0.19408379,0.93784362){\color[rgb]{0,0,0}\makebox(0,0)[lb]{\smash{$a_5$}}}%
    \put(0.05013192,0.87897964){\color[rgb]{0,0,0}\makebox(0,0)[lb]{\smash{$a_6$}}}%
    \put(0.02228085,0.29410724){\color[rgb]{0,0,0}\makebox(0,0)[lb]{\smash{$a_7$}}}%
    \put(0.13034299,0.08165566){\color[rgb]{0,0,0}\makebox(0,0)[lb]{\smash{$a_8$}}}%
    \put(0.46789792,0.02005275){\color[rgb]{0,0,0}\makebox(0,0)[lb]{\smash{$a_9$}}}%
    \put(0.58424141,0.02005275){\color[rgb]{0,0,0}\makebox(0,0)[lb]{\smash{$a_{10}$}}}%
    \put(0.67176768,0.0802111){\color[rgb]{0,0,0}\makebox(0,0)[lb]{\smash{$a_{11}$}}}%
    \put(0.80990896,0.14259748){\color[rgb]{0,0,0}\makebox(0,0)[lb]{\smash{$a_{12}$}}}%
    \put(0.91017282,0.3030196){\color[rgb]{0,0,0}\makebox(0,0)[lb]{\smash{$a_{13}$}}}%
  \end{picture}%
\endgroup%

%% file: biggraph1bv.pdf_tex
\begingroup%
  \makeatletter%
  \providecommand\color[2][]{%
    \errmessage{(Inkscape) Color is used for the text in Inkscape, but the package 'color.sty' is not loaded}%
    \renewcommand\color[2][]{}%
  }%
  \providecommand\transparent[1]{%
    \errmessage{(Inkscape) Transparency is used (non-zero) for the text in Inkscape, but the package 'transparent.sty' is not loaded}%
    \renewcommand\transparent[1]{}%
  }%
  \providecommand\rotatebox[2]{#2}%
  \ifx\svgwidth\undefined%
    \setlength{\unitlength}{288.8bp}%
    \ifx\svgscale\undefined%
      \relax%
    \else%
      \setlength{\unitlength}{\unitlength * \real{\svgscale}}%
    \fi%
  \else%
    \setlength{\unitlength}{\svgwidth}%
  \fi%
  \global\let\svgwidth\undefined%
  \global\let\svgscale\undefined%
  \makeatother%
  \begin{picture}(1,0.97229917)%
    \put(0,0){\includegraphics[width=\unitlength]{biggraph1bv.pdf}}%
    \put(0.95807671,0.4344766){\color[rgb]{0,0,0}\makebox(0,0)[lb]{\smash{$a_1$}}}%
    \put(0.95696261,0.55256513){\color[rgb]{0,0,0}\makebox(0,0)[lb]{\smash{$a_2$}}}%
    \put(0.62274981,0.64391665){\color[rgb]{0,0,0}\makebox(0,0)[lb]{\smash{$a_3$}}}%
    \put(0.55169149,0.88891589){\color[rgb]{0,0,0}\makebox(0,0)[lb]{\smash{$a_4$}}}%
    \put(0.19408379,0.93784362){\color[rgb]{0,0,0}\makebox(0,0)[lb]{\smash{$a_5$}}}%
    \put(0.05013192,0.87897964){\color[rgb]{0,0,0}\makebox(0,0)[lb]{\smash{$a_6$}}}%
    \put(0.02228085,0.29410724){\color[rgb]{0,0,0}\makebox(0,0)[lb]{\smash{$a_7$}}}%
    \put(0.13034299,0.08165566){\color[rgb]{0,0,0}\makebox(0,0)[lb]{\smash{$a_8$}}}%
    \put(0.46789792,0.02005275){\color[rgb]{0,0,0}\makebox(0,0)[lb]{\smash{$a_9$}}}%
    \put(0.58424141,0.02005275){\color[rgb]{0,0,0}\makebox(0,0)[lb]{\smash{$a_{10}$}}}%
    \put(0.67176768,0.0802111){\color[rgb]{0,0,0}\makebox(0,0)[lb]{\smash{$a_{11}$}}}%
    \put(0.80990896,0.14259748){\color[rgb]{0,0,0}\makebox(0,0)[lb]{\smash{$a_{12}$}}}%
    \put(0.91017282,0.3030196){\color[rgb]{0,0,0}\makebox(0,0)[lb]{\smash{$a_{13}$}}}%
  \end{picture}%
\endgroup%

%% file: cycleval.pdf_tex
\begingroup%
  \makeatletter%
  \providecommand\color[2][]{%
    \errmessage{(Inkscape) Color is used for the text in Inkscape, but the package 'color.sty' is not loaded}%
    \renewcommand\color[2][]{}%
  }%
  \providecommand\transparent[1]{%
    \errmessage{(Inkscape) Transparency is used (non-zero) for the text in Inkscape, but the package 'transparent.sty' is not loaded}%
    \renewcommand\transparent[1]{}%
  }%
  \providecommand\rotatebox[2]{#2}%
  \ifx\svgwidth\undefined%
    \setlength{\unitlength}{267.6244873bp}%
    \ifx\svgscale\undefined%
      \relax%
    \else%
      \setlength{\unitlength}{\unitlength * \real{\svgscale}}%
    \fi%
  \else%
    \setlength{\unitlength}{\svgwidth}%
  \fi%
  \global\let\svgwidth\undefined%
  \global\let\svgscale\undefined%
  \makeatother%
  \begin{picture}(1,0.87999078)%
    \put(0,0){\includegraphics[width=\unitlength,page=1]{cycleval.pdf}}%
    \put(0.10607575,0.28518369){\color[rgb]{0,0,0}\makebox(0,0)[lb]{\smash{$E_p$}}}%
    \put(0.82947745,0.28518369){\color[rgb]{0,0,0}\makebox(0,0)[lb]{\smash{$E_m$}}}%
    \put(0.20173215,0.11180644){\color[rgb]{0,0,0}\makebox(0,0)[lb]{\smash{$x_p$}}}%
    \put(0.73979956,0.11180644){\color[rgb]{0,0,0}\makebox(0,0)[lb]{\smash{$x_m$}}}%
    \put(0.29141005,0.28518369){\color[rgb]{0,0,0}\makebox(0,0)[lb]{\smash{$(C_p)_\pi$}}}%
    \put(0.4408732,0.15365612){\color[rgb]{0,0,0}\makebox(0,0)[lb]{\smash{$E_a$}}}%
    \put(0.54250815,0.28518369){\color[rgb]{0,0,0}\makebox(0,0)[lb]{\smash{$(C_m)_\pi$}}}%
    \put(0.42943809,0.72959719){\color[rgb]{0,0,0}\makebox(0,0)[lb]{\smash{$E_b$}}}%
    \put(0,0){\includegraphics[width=\unitlength,page=2]{cycleval.pdf}}%
    \put(0.56044372,0.06397823){\color[rgb]{0,0,0}\makebox(0,0)[lb]{\smash{$A_\pi$}}}%
    \put(0.56096567,0.81329657){\color[rgb]{0,0,0}\makebox(0,0)[lb]{\smash{$B_\pi$}}}%
  \end{picture}%
\endgroup%

%% file: graphofRtrees1a.pdf_tex
\begingroup%
  \makeatletter%
  \providecommand\color[2][]{%
    \errmessage{(Inkscape) Color is used for the text in Inkscape, but the package 'color.sty' is not loaded}%
    \renewcommand\color[2][]{}%
  }%
  \providecommand\transparent[1]{%
    \errmessage{(Inkscape) Transparency is used (non-zero) for the text in Inkscape, but the package 'transparent.sty' is not loaded}%
    \renewcommand\transparent[1]{}%
  }%
  \providecommand\rotatebox[2]{#2}%
  \ifx\svgwidth\undefined%
    \setlength{\unitlength}{480.8bp}%
    \ifx\svgscale\undefined%
      \relax%
    \else%
      \setlength{\unitlength}{\unitlength * \real{\svgscale}}%
    \fi%
  \else%
    \setlength{\unitlength}{\svgwidth}%
  \fi%
  \global\let\svgwidth\undefined%
  \global\let\svgscale\undefined%
  \makeatother%
  \begin{picture}(1,0.35108153)%
    \put(0,0){\includegraphics[width=\unitlength]{graphofRtrees1a.pdf}}%
  \end{picture}%
\endgroup%

%% file: graphofRtrees1c.pdf_tex
\begingroup%
  \makeatletter%
  \providecommand\color[2][]{%
    \errmessage{(Inkscape) Color is used for the text in Inkscape, but the package 'color.sty' is not loaded}%
    \renewcommand\color[2][]{}%
  }%
  \providecommand\transparent[1]{%
    \errmessage{(Inkscape) Transparency is used (non-zero) for the text in Inkscape, but the package 'transparent.sty' is not loaded}%
    \renewcommand\transparent[1]{}%
  }%
  \providecommand\rotatebox[2]{#2}%
  \ifx\svgwidth\undefined%
    \setlength{\unitlength}{496.8bp}%
    \ifx\svgscale\undefined%
      \relax%
    \else%
      \setlength{\unitlength}{\unitlength * \real{\svgscale}}%
    \fi%
  \else%
    \setlength{\unitlength}{\svgwidth}%
  \fi%
  \global\let\svgwidth\undefined%
  \global\let\svgscale\undefined%
  \makeatother%
  \begin{picture}(1,0.38808374)%
    \put(0,0){\includegraphics[width=\unitlength]{graphofRtrees1c.pdf}}%
  \end{picture}%
\endgroup%

%% file: graphofRtrees1b.pdf_tex
\begingroup%
  \makeatletter%
  \providecommand\color[2][]{%
    \errmessage{(Inkscape) Color is used for the text in Inkscape, but the package 'color.sty' is not loaded}%
    \renewcommand\color[2][]{}%
  }%
  \providecommand\transparent[1]{%
    \errmessage{(Inkscape) Transparency is used (non-zero) for the text in Inkscape, but the package 'transparent.sty' is not loaded}%
    \renewcommand\transparent[1]{}%
  }%
  \providecommand\rotatebox[2]{#2}%
  \ifx\svgwidth\undefined%
    \setlength{\unitlength}{248.8bp}%
    \ifx\svgscale\undefined%
      \relax%
    \else%
      \setlength{\unitlength}{\unitlength * \real{\svgscale}}%
    \fi%
  \else%
    \setlength{\unitlength}{\svgwidth}%
  \fi%
  \global\let\svgwidth\undefined%
  \global\let\svgscale\undefined%
  \makeatother%
  \begin{picture}(1,1.08038595)%
    \put(0,0){\includegraphics[width=\unitlength]{graphofRtrees1b.pdf}}%
  \end{picture}%
\endgroup%

%% file: graphofRtrees1d.pdf_tex
\begingroup%
  \makeatletter%
  \providecommand\color[2][]{%
    \errmessage{(Inkscape) Color is used for the text in Inkscape, but the package 'color.sty' is not loaded}%
    \renewcommand\color[2][]{}%
  }%
  \providecommand\transparent[1]{%
    \errmessage{(Inkscape) Transparency is used (non-zero) for the text in Inkscape, but the package 'transparent.sty' is not loaded}%
    \renewcommand\transparent[1]{}%
  }%
  \providecommand\rotatebox[2]{#2}%
  \ifx\svgwidth\undefined%
    \setlength{\unitlength}{248.8bp}%
    \ifx\svgscale\undefined%
      \relax%
    \else%
      \setlength{\unitlength}{\unitlength * \real{\svgscale}}%
    \fi%
  \else%
    \setlength{\unitlength}{\svgwidth}%
  \fi%
  \global\let\svgwidth\undefined%
  \global\let\svgscale\undefined%
  \makeatother%
  \begin{picture}(1,1.08038595)%
    \put(0,0){\includegraphics[width=\unitlength]{graphofRtrees1d.pdf}}%
  \end{picture}%
\endgroup%

%% file: graphofRtrees1e.pdf_tex
\begingroup%
  \makeatletter%
  \providecommand\color[2][]{%
    \errmessage{(Inkscape) Color is used for the text in Inkscape, but the package 'color.sty' is not loaded}%
    \renewcommand\color[2][]{}%
  }%
  \providecommand\transparent[1]{%
    \errmessage{(Inkscape) Transparency is used (non-zero) for the text in Inkscape, but the package 'transparent.sty' is not loaded}%
    \renewcommand\transparent[1]{}%
  }%
  \providecommand\rotatebox[2]{#2}%
  \ifx\svgwidth\undefined%
    \setlength{\unitlength}{496.8bp}%
    \ifx\svgscale\undefined%
      \relax%
    \else%
      \setlength{\unitlength}{\unitlength * \real{\svgscale}}%
    \fi%
  \else%
    \setlength{\unitlength}{\svgwidth}%
  \fi%
  \global\let\svgwidth\undefined%
  \global\let\svgscale\undefined%
  \makeatother%
  \begin{picture}(1,0.38808374)%
    \put(0,0){\includegraphics[width=\unitlength]{graphofRtrees1e.pdf}}%
  \end{picture}%
\endgroup%

%% file: graphofRtrees1f.pdf_tex
\begingroup%
  \makeatletter%
  \providecommand\color[2][]{%
    \errmessage{(Inkscape) Color is used for the text in Inkscape, but the package 'color.sty' is not loaded}%
    \renewcommand\color[2][]{}%
  }%
  \providecommand\transparent[1]{%
    \errmessage{(Inkscape) Transparency is used (non-zero) for the text in Inkscape, but the package 'transparent.sty' is not loaded}%
    \renewcommand\transparent[1]{}%
  }%
  \providecommand\rotatebox[2]{#2}%
  \ifx\svgwidth\undefined%
    \setlength{\unitlength}{496.8bp}%
    \ifx\svgscale\undefined%
      \relax%
    \else%
      \setlength{\unitlength}{\unitlength * \real{\svgscale}}%
    \fi%
  \else%
    \setlength{\unitlength}{\svgwidth}%
  \fi%
  \global\let\svgwidth\undefined%
  \global\let\svgscale\undefined%
  \makeatother%
  \begin{picture}(1,0.38808374)%
    \put(0,0){\includegraphics[width=\unitlength]{graphofRtrees1f.pdf}}%
  \end{picture}%
\endgroup%

%% file: example01-dualgraph.pdf_tex
\begingroup%
  \makeatletter%
  \providecommand\color[2][]{%
    \errmessage{(Inkscape) Color is used for the text in Inkscape, but the package 'color.sty' is not loaded}%
    \renewcommand\color[2][]{}%
  }%
  \providecommand\transparent[1]{%
    \errmessage{(Inkscape) Transparency is used (non-zero) for the text in Inkscape, but the package 'transparent.sty' is not loaded}%
    \renewcommand\transparent[1]{}%
  }%
  \providecommand\rotatebox[2]{#2}%
  \ifx\svgwidth\undefined%
    \setlength{\unitlength}{312bp}%
    \ifx\svgscale\undefined%
      \relax%
    \else%
      \setlength{\unitlength}{\unitlength * \real{\svgscale}}%
    \fi%
  \else%
    \setlength{\unitlength}{\svgwidth}%
  \fi%
  \global\let\svgwidth\undefined%
  \global\let\svgscale\undefined%
  \makeatother%
  \begin{picture}(1,0.66666667)%
    \put(0,0){\includegraphics[width=\unitlength]{example01-dualgraph.pdf}}%
  \end{picture}%
\endgroup%

%% file: example01-valuationspace.pdf_tex
\begingroup%
  \makeatletter%
  \providecommand\color[2][]{%
    \errmessage{(Inkscape) Color is used for the text in Inkscape, but the package 'color.sty' is not loaded}%
    \renewcommand\color[2][]{}%
  }%
  \providecommand\transparent[1]{%
    \errmessage{(Inkscape) Transparency is used (non-zero) for the text in Inkscape, but the package 'transparent.sty' is not loaded}%
    \renewcommand\transparent[1]{}%
  }%
  \providecommand\rotatebox[2]{#2}%
  \ifx\svgwidth\undefined%
    \setlength{\unitlength}{312bp}%
    \ifx\svgscale\undefined%
      \relax%
    \else%
      \setlength{\unitlength}{\unitlength * \real{\svgscale}}%
    \fi%
  \else%
    \setlength{\unitlength}{\svgwidth}%
  \fi%
  \global\let\svgwidth\undefined%
  \global\let\svgscale\undefined%
  \makeatother%
  \begin{picture}(1,0.66666667)%
    \put(0,0){\includegraphics[width=\unitlength]{example01-valuationspace.pdf}}%
  \end{picture}%
\endgroup%

%% file: example01-blockvertex.pdf_tex
\begingroup%
  \makeatletter%
  \providecommand\color[2][]{%
    \errmessage{(Inkscape) Color is used for the text in Inkscape, but the package 'color.sty' is not loaded}%
    \renewcommand\color[2][]{}%
  }%
  \providecommand\transparent[1]{%
    \errmessage{(Inkscape) Transparency is used (non-zero) for the text in Inkscape, but the package 'transparent.sty' is not loaded}%
    \renewcommand\transparent[1]{}%
  }%
  \providecommand\rotatebox[2]{#2}%
  \ifx\svgwidth\undefined%
    \setlength{\unitlength}{312bp}%
    \ifx\svgscale\undefined%
      \relax%
    \else%
      \setlength{\unitlength}{\unitlength * \real{\svgscale}}%
    \fi%
  \else%
    \setlength{\unitlength}{\svgwidth}%
  \fi%
  \global\let\svgwidth\undefined%
  \global\let\svgscale\undefined%
  \makeatother%
  \begin{picture}(1,0.66666667)%
    \put(0,0){\includegraphics[width=\unitlength]{example01-blockvertex.pdf}}%
    \put(1.01282051,0.20512821){\color[rgb]{0,0,0}\makebox(0,0)[lb]{\smash{$a_1$}}}%
  \end{picture}%
\endgroup%

%% file: example01-dualgraph-4points.pdf_tex
\begingroup%
  \makeatletter%
  \providecommand\color[2][]{%
    \errmessage{(Inkscape) Color is used for the text in Inkscape, but the package 'color.sty' is not loaded}%
    \renewcommand\color[2][]{}%
  }%
  \providecommand\transparent[1]{%
    \errmessage{(Inkscape) Transparency is used (non-zero) for the text in Inkscape, but the package 'transparent.sty' is not loaded}%
    \renewcommand\transparent[1]{}%
  }%
  \providecommand\rotatebox[2]{#2}%
  \ifx\svgwidth\undefined%
    \setlength{\unitlength}{312bp}%
    \ifx\svgscale\undefined%
      \relax%
    \else%
      \setlength{\unitlength}{\unitlength * \real{\svgscale}}%
    \fi%
  \else%
    \setlength{\unitlength}{\svgwidth}%
  \fi%
  \global\let\svgwidth\undefined%
  \global\let\svgscale\undefined%
  \makeatother%
  \begin{picture}(1,0.66666667)%
    \put(0,0){\includegraphics[width=\unitlength]{example01-dualgraph-4points.pdf}}%
  \end{picture}%
\endgroup%